\documentclass{article}

\usepackage[english]{babel} % English formatting
\usepackage{enumitem}
\usepackage[utf8]{inputenc} % Standard encoding
\usepackage[a4paper,left=2.5cm,right=2.5cm,bottom=3cm]{geometry} % Page formatting
\usepackage{indentfirst} % Indents the first paragraph
\usepackage{amsmath,amssymb,amsfonts,amsthm} % Maths type package
\numberwithin{equation}{section}
\usepackage{bm} % Bold font maths
\usepackage{graphicx} % Advanced graphics package
\usepackage[export]{adjustbox} 
\usepackage{pdflscape} % Make pages landscape
\usepackage{fancyhdr} % Fancy headers
\usepackage{natbib} % Bibliography
\setcitestyle{numbers,square}
\usepackage{flafter} % Reference any 'float'
\usepackage[framemethod=tikz]{mdframed} % Box off stuff
\usepackage{color} % Colour support
	\definecolor{yellow-green}{rgb}{0.6, 0.8, 0.2}
	\definecolor{viridian}{rgb}{0.25, 0.51, 0.43}
\usepackage{wrapfig} % Text flowing around figures
\usepackage{lipsum} % Generates meaningless text

\usepackage[all]{xy}
\usepackage[auth-lg,noblocks]{authblk}
\usepackage{framed}
\usepackage{tikz}
\usepackage{quiver}
\usetikzlibrary{hobby}
\usepackage{thmtools}
\usepackage{thm-restate}

\usepackage{mathrsfs}

\graphicspath{ {Images/} } % Folder for images 

\usepackage{multicol} % Multi Column Environment
\usepackage{float} % Lets you add images to multi-column environment

\usepackage{nicematrix}
\usetikzlibrary{patterns}
\usetikzlibrary{matrix,decorations.pathreplacing}
\usetikzlibrary{decorations.pathreplacing,calligraphy}
\usepackage[colorlinks=true,linkcolor=cyan,citecolor=viridian]{hyperref} % Link colours
\hypersetup{
    pdfcreator={},
    pdfproducer={LaTeX},
    hypertexnames=false,  % 尝试切换这个选项
    linktocpage=true,
    colorlinks=true
}
\usepackage{prettyref}
\usepackage[nameinlink]{cleveref}
\title{Infinitesimal Rigidity of Cyclic Surfaces and Alternating Surfaces}

\author{Qiongling Li\thanks{Chern Institute of Mathematics and LPMC, Nankai University, Tianjin 300071, China, \href{mailto:qiongling.li@nankai.edu.cn}{qiongling.li@nankai.edu.cn}} \mbox{~and} Junming Zhang\thanks{Chern Institute of Mathematics and LPMC, Nankai University, Tianjin 300071, China, \href{mailto:junmingzhang@mail.nankai.edu.cn}{junmingzhang@mail.nankai.edu.cn}}}

\date{\vspace{-2em}}

\newtheorem{theorem}{Theorem}[section]
\newtheorem{corollary}[theorem]{Corollary}
\newtheorem{proposition}[theorem]{Proposition}
\newtheorem{lemma}[theorem]{Lemma}
\newtheorem{definition}[theorem]{Definition}
\newtheorem{remark}[theorem]{Remark}
\newtheorem{example}[theorem]{Example}

\newrefformat{lem}{Lemma \ref{#1}}
\newrefformat{coro}{Corollary \ref{#1}}
\newrefformat{thm}{Theorem \ref{#1}}

\newrefformat{prop}{Proposition \ref{#1}}
\newrefformat{rem}{Remark \ref{#1}}
\newrefformat{example}{Example \ref{#1}}
\newrefformat{defn}{Definition \ref{#1}}
\newrefformat{section}{Section \ref{#1}}
\newrefformat{conj}{Conjecture \ref{#1}}
\newrefformat{apdx}{Appendix \ref{#1}}
\newrefformat{fact}{Fact \ref{#1}}
\newcommand{\dd}{{\mathrm d}}
\newcommand{\EE}{{\mathcal{E}}}

\newcommand{\KK}{{\mathcal{K}}}

\newcommand{\VV}{{\mathcal{V}}}
\newcommand{\SO}{{\mathrm{SO}}}

\newcommand{\iu}{{\mathrm i}}

\newcommand{\vol}{{\mathrm{d}\operatorname{vol}}}

\DeclareMathOperator{\Hom}{Hom}

\begin{document}

\pagenumbering{gobble} % keep title page without a number
\maketitle

\pagenumbering{arabic} % start page numbers again
\setcounter{section}{0}
\setcounter{page}{1}
\vspace{-1em}

\begin{abstract}
We study the infinitesimal rigidity of equivariant minimal maps from the universal cover of a smooth oriented surface (possibly non-compact) into a Riemannian symmetric space, focusing on representations arising from cyclic harmonic bundles. By developing a unified Lie-theoretic framework that connects cyclic surfaces and cyclic harmonic bundles over Riemann surfaces, we prove the infinitesimal rigidity for irreducible cyclic surfaces under admissible smooth variations, including both compactly supported deformations and $L^p$-integrable variations on non-compact surfaces. As a geometric application, we introduce $n$-alternating surfaces in $\mathbb H^{p,q}$ and establish their correspondence with a special class of cyclic surfaces. This yields an infinitesimal rigidity theorem that conceptually unifies and extends known rigidity results for maximal space-like surfaces, alternating holomorphic curves, and $A$-surfaces in certain $\mathbb H^{p,q}$. 
\end{abstract}

\large

\tableofcontents

\section{Introduction}

Let $\Sigma$ be a smooth oriented surface without boundary (possibly non-compact), and let $G$ be a semisimple complex Lie group. In this paper, we study the infinitesimal rigidity of a $\rho$-equivariant minimal map $f_\rho$ from the universal cover $\widetilde{\Sigma}$ into the symmetric space $\operatorname{Sym}(G)$, with a particular focus on representations $\rho \colon \pi_1(\Sigma) \to G$ arising from cyclic harmonic Higgs bundles.

For closed surfaces, equivariant minimal maps arise naturally in higher Teichm\"uller theory. Labourie \cite{labourie2008cross} established the existence of $f_\rho$ for Anosov representations. This existence property holds across all known higher Teichm\"uller spaces, including the Hitchin components \cite{hitchin1992lie,fock2006moduli,labourie2006anosov}, the maximal components \cite{burger2010surface,burger2005maximal}, and the $\Theta$-positive components \cite{guichard2025generalizing,guichard2026positivity,beyrer2024positivity}. 

On the other hand, the uniqueness of $f_\rho$---the generalized Labourie conjecture \cite{labourie2008cross}---has remained a long-standing open problem. While global uniqueness holds for rank $2$ split real forms for Hitchin representations \cite{schoen1993role,loftin2001affine,labourie2007flat,bonsante2010maximal,bosante2020anti,labourie2017cyclic} and across several other higher Teichm\"uller spaces \cite{collier2016maximal,alessandrini2019geometry,collier2019geometry}, it is known to fail  when the Lie group has real rank $\geqslant 3$ \cite{markovic2022non,sagman2026unstable,markovic2025unstable}.

A common strategy for tackling uniqueness involves lifting $f_\rho$ to a special surface in an auxiliary homogeneous space. Prominent examples include: (i) hyperbolic affine spheres in $\mathbb{R}^3$ for $\mathrm{SL}_3\mathbb{R}$-Hitchin representations \cite{labourie2007flat,loftin2001affine}; (ii) maximal space-like surfaces in the pseudo-hyperbolic space $\mathbb{H}^{2,q}$ for maximal $\SO^0_{2,q+1}$-representations \cite{collier2019geometry} (see also \cite{bonsante2010maximal,bosante2020anti} when $q=1$); and (iii) cyclic surfaces for representations in the cyclic loci of Hitchin or maximal components \cite{labourie2017cyclic,collier2016maximal,alessandrini2019geometry}. Instead of studying global uniqueness directly, Labourie \cite{labourie2017cyclic} pioneered a powerful alternative: proving \emph{infinitesimal rigidity}, which establishes the vanishing of normal first-order deformations (\textbf{Jacobi fields}). This approach has since been widely adopted \cite{collier2016maximal,alessandrini2019geometry}.

Since global uniqueness fails, studying infinitesimal rigidity along particular specified variation families remains valuable. Following Labourie's strategy, Collier--Toulisse \cite{collier2024holomorphic} established the infinitesimal rigidity of $\alpha_1$-cyclic surfaces mappings into a homogeneous space of $\mathrm{G}_2$, which correspond to alternating holomorphic curves in $\mathbb{H}^{4,2}$ as well as cyclic $\mathrm{G}_2$-Higgs bundles (see also \cite{baraglia2009g2, nie2024cyclic, Evans2025G2Hitchin}). 

From a geometric perspective inspired by Chern's Frenet frames \cite{chern1970minimal} and Bryant's superminimal surfaces \cite{bryant1982conformal}, Nie \cite{nie2024cyclic} defined $A$-surfaces in $\mathbb{H}^{2k,2k}$ or $\mathbb{H}^{2k+2,2k}$ whose monodromy lies in the cyclic locus of the $\SO_{n,n+1}^0$-representation variety. Departing from Labourie's approach, Nie proved their infinitesimal rigidity via the control of geometric data. Similarly, Rungi--Tamburelli \cite{rungi2025complex} studied isotropic $\mathbf{P}$-alternating surfaces in para-complex manifolds (related to $\mathrm{SL}_{2n+1}\mathbb{R}$ representations) and established rigidity under comparable bounds. Notably, the analytical techniques in both works naturally apply to compactly supported deformations on non-compact surfaces.

In this paper, we provide a unified Lie-theoretic framework that conceptually encompasses all these infinitesimal rigidity results via the geometry of cyclic surfaces. Furthermore, we generalize these rigidity phenomena to non-compact surfaces, accommodating both compactly supported deformations and variations satisfying suitable analytic integrability conditions.

\subsection{Cyclic surfaces}

Let $G$ be a connected semisimple complex Lie group with finite-dimensional Lie algebra $\mathfrak{g}$, and let $\Theta\in\operatorname{Aut}(G)$ be an automorphism of order $m\geqslant3$. Its derivative $\theta$ defines a $\mathbb{Z}/m\mathbb{Z}$-grading on $\mathfrak{g}$:
\[\mathfrak{g}_{\overline{j}}:=\{X\in\mathfrak{g}\mid\theta X =\zeta_m^jX\},\]
where $\zeta_m:=\exp(2\pi\mathrm{i}/m)$. A $\theta$-invariant Cartan involution $\tau$ determines a maximal compact subgroup $H^\Theta$ of the $\Theta$-fixed subgroup $G^\Theta$. The associated \textbf{cyclic space} is defined as the homogeneous space $\mathbb{X}_{\Theta}:=G/H^{\Theta}$.

We study \textbf{cyclic surfaces} in $\mathbb{X}_{\Theta}$, or more generally in $G/U$ for any reductive subgroup $U<H^\Theta$. Geometrically, a cyclic surface $f\colon\Sigma\to G/U$ behaves like a holomorphic curve: $\Sigma$ inherits a complex structure $\mathsf{j}$, and the Lie bracket of its differential with itself takes values in $\mathfrak{u} = \operatorname{Lie}(U)$. Furthermore, $f$ is \textbf{$\mathfrak{c}$-cyclic} if its differential is nowhere vanishing along a chosen $U$-invariant line $\mathfrak{c}\subset\mathfrak{g}_{\overline{1}}$ (see \prettyref{defn:cycsurface}). 

An equivariant ($\mathfrak{c}$-)cyclic surface is a pair $(\rho,f)$ where $f\colon\widetilde{\Sigma}\to G/U$ is a ($\mathfrak{c}$-)cyclic surface on the universal cover, equivariant with respect to a representation $\rho\colon\pi_1(\Sigma)\to G$.

Our main theorem establishes the infinitesimal rigidity of these surfaces:
\begin{theorem}\label{thm:theorem}(\prettyref{thm:main})
    Let $(\rho_t, f_t)_{t\in (-\varepsilon, \varepsilon)}$ with $\dot{\rho}_0=0$ be an admissible smooth variation of an equivariant irreducible cyclic surface $(\rho,f)$ on $\Sigma$ in $\mathbb{X}_\Theta$. If the variation arises from a smooth variation of equivariant $\mathfrak{c}$-cyclic surfaces, 
    then there exists a smooth path of diffeomorphisms $(\psi_t)_{t\in (-\varepsilon, \varepsilon)}$ in $\operatorname{Diff}^0(\Sigma)$ with $\psi_0 = \mathrm{id}_\Sigma$ such that the reparameterized variation $f'_t = f_t\circ\psi_t$ satisfies $\dot{f}'_0=0$.
\end{theorem}
\begin{remark}
For closed surfaces, Collier--Toulisse--Wentworth
\cite[Theorem G]{collier2025higgs} prove an infinitesimal rigidity result for cyclic Higgs bundles by a different method, namely by constructing and analyzing a joint moduli space of Higgs bundles. Through the cyclic surface--cyclic Higgs bundle correspondence, their theorem can be viewed as a Higgs-bundle counterpart of the closed-surface case of the result above.
\end{remark}

Here, a cyclic surface is termed irreducible if its image does not lie entirely within an orbit $L\cdot x$ for any proper Levi subgroup $L\subsetneq G$ and basepoint $x\in \mathbb{X}_{\Theta}$. The term ``admissible'' refers to specific analytic integrability conditions (see \prettyref{defn:admissible} and Proposition \ref{prop:admissible_cyclic_sufficient}); notably, every smooth variation with compact support is automatically admissible.

\subsection{Cyclic Higgs bundles}

Our proof analyzes the infinitesimal automorphism spaces of Higgs bundles. For closed surfaces, the non-Abelian Hodge correspondence \cite{donaldson1985anti,donaldson1987twisted,uhlenbeck1986existence,hitchin1987self,corlette1988flat,simpson1992higgs} conventionally uses Higgs bundles to probe the topology of character varieties. We extend this machinery to general non-compact Riemann surfaces. 

While previous works \cite{labourie2017cyclic,collier2016maximal,alessandrini2019geometry,collier2019geometry} linked cyclic surfaces and cyclic Higgs bundles (studied extensively in \cite{simpson2009katz,baraglia2009g2,baraglia2015cyclic}) in specific target spaces, we generalize this equivalence to a broader setup. 

Fix a complex structure $\mathsf{j}$ on $\Sigma$, yielding the Riemann surface $X = (\Sigma, \mathsf{j})$. A \textbf{$\Theta$-cyclic harmonic $G$-Higgs bundle} $(\mathbb{E},\varphi,h)$ consists of:
\begin{itemize}
    \item a holomorphic principal $G^\Theta$-bundle $\mathbb{E}\to X$;
    \item a holomorphic section (the Higgs field) $\varphi$ of $\mathbb{E}[\mathfrak{g}_{\overline{1}}]\otimes\KK_X$, where $\mathbb{E}[\mathfrak{g}_{\overline{1}}]$ denotes the associated vector bundle of $\mathbb{E}$ via the adjoint action $G^\Theta\to\mathrm{GL}(\mathfrak{g}_{\overline{1}})$ and $\KK_X$ is the canonical line bundle;
    \item a harmonic metric $h$ on $\mathbb{E}$ satisfying the Hitchin equation
    \[F(\nabla^h)-[\varphi\wedge\tau_h(\varphi)]=0,\]
    where $\nabla^h$ is the Chern connection, $F(\nabla^h)$ is its curvature, and $\tau_h$ is the conjugate-linear involution combining $h$ and the Cartan involution. 
\end{itemize}

We say $(\mathbb{E},\varphi,h)$ reduces to a reductive subgroup $U<H^\Theta$ if $h$ further defines a reduction of the structure group to $U$.

We establish the following equivalence:
\begin{theorem}\label{thm:equiv}(\prettyref{thm:correspondence})
    There is a natural bijection between:
    \begin{itemize}
        \item $\mathsf{j}$-isomorphism classes of equivariant immersed (resp.\ $\mathfrak{c}$-cyclic) surfaces from $\Sigma$ to $G/U$ with induced complex structure $\mathsf{j}$, and
        \item isomorphism classes of $\Theta$-cyclic harmonic $G$-Higgs bundles reducing to $U$ whose Higgs fields are nowhere vanishing (resp.\ nowhere vanishing along $\mathfrak{c}$) over $X=(\Sigma,\mathsf{j})$.
    \end{itemize}
\end{theorem}

\subsection{Alternating surfaces}

We introduce \emph{$n$-alternating surfaces} in $\mathbb{H}^{p,q}$, where the signature $(p,q)$ is $(n, n+k-1)$ when $n$ is even, and $(n+k, n-1)$ when $n$ is odd, for $k \ge 0$. This framework serves to unify and generalize maximal space-like surfaces in $\mathbb{H}^{2,q}$, alternating holomorphic curves in $\mathbb{H}^{4,2}$, and $A$-surfaces in $\mathbb{H}^{2k,2k}$ or $\mathbb{H}^{2k+2,2k}$. An $n$-alternating surface features an adapted orthogonal splitting of its complexified ambient trivial bundle into $n+1$ components, where the connection shifts between adjacent subbundles. In Theorem \ref{thm:alternating-harmonic-correspondence}, we establish a bijection between equivariant $n$-alternating surfaces in $\mathbb{H}^{p,q}$ and a specific class of both cyclic $\operatorname{SO}^0_{p,q+1}$-harmonic bundles and cyclic surfaces.

Translating \prettyref{thm:theorem} yields the infinitesimal rigidity for $n$-alternating surfaces:

\begin{theorem}\label{intro:alternationtheorem}(Theorem \ref{thm:alternationtheorem})
    Let $(\rho_t, f_t)_{t\in (-\varepsilon, \varepsilon)}$ with $\dot{\rho}_0=0$ be an admissible smooth variation of an equivariant irreducible $n$-alternating surface $(\rho,f) \colon \Sigma \to \mathbb{H}^{p,q}$. Then there exists a smooth path of diffeomorphisms $(\psi_t)_{t\in (-\varepsilon, \varepsilon)}$ in $\operatorname{Diff}^0(\Sigma)$ with $\psi_0 = \mathrm{id}_\Sigma$ such that the reparameterized variation $f'_t = f_t\circ\psi_t$ satisfies $\dot{f}'_0=0$.
\end{theorem}

Here, an $n$-alternating surface is termed irreducible if its image does not lie entirely within an orbit $L\cdot x$ for any proper Levi subgroup $L\subsetneq \operatorname{SO}^0_{p,q+1}$ and basepoint $x\in \mathbb{H}^{p,q}$. As before, ``admissible'' refers to specific analytic integrability conditions (see Definition \ref{defn:admissiblealternating} and Proposition \ref{prop:admissiblealternating}), and every smooth variation with compact support is automatically admissible.

\subsection*{Organization}
The paper is organized as follows. Section \ref{sec:prelimnaries} reviews the algebraic preliminaries of graded Lie algebras and cyclic spaces. Section \ref{sec:Higgs bundles} covers the cyclic Higgs bundles and the Hitchin--Kobayashi correspondence. Section \ref{sec:cyclicsurfaceshiggsbundles} establishes the dictionary between equivariant cyclic surfaces and cyclic Higgs bundles, proving Theorem \ref{thm:equiv}. Section \ref{sec:infinitesimalrigidity} defines variation admissibility and establishes the analytical estimates necessary to prove \prettyref{thm:theorem}. Finally, Section \ref{sec:alternating} introduces $n$-alternating surfaces in $\mathbb H^{p,q}$ and proves Theorem \ref{intro:alternationtheorem}. In the appendix, we provide several families of examples.

\subsection*{Acknowledgement}
Both authors are partially supported by the National Key R\&D
Program of China No. 2022YFA1006600, the Fundamental Research Funds for the
Central Universities, and Nankai Zhide foundation. Q. Li is sponsored by the
Alexander von Humboldt Foundation. J. Zhang is supported by NSF of China grant No. 125B2007.

\section{Preliminaries}
\label{sec:prelimnaries}
In this section, we assume that $\mathfrak{g}$ (resp. $\mathfrak{g}_{\mathbb{R}}$) is a finite-dimensional semisimple complex (resp. real) Lie algebra and $G$ (resp. $G_{\mathbb{R}}$) is a Lie group whose Lie algebra is $\mathfrak{g}$ (resp. $\mathfrak{g}_{\mathbb{R}}$). We will always assume $G$ to be connected in this article.

\subsection{Basic concepts}

We recommend reading \cite{knapp1996lie} and \cite[Section 2.4]{garcia2009hitchin} for references.

\paragraph{Real form}  A real form of the complex Lie group $G$ is defined to be the fixed point set $G^\Lambda$ of an anti-holomorphic involution $\Lambda\colon G\to G$. Similarly, a real form of the complex Lie algebra $\mathfrak{g}$ is defined to be the fixed point set $\mathfrak{g}^\lambda$ of a conjugate-linear involution $\lambda\colon \mathfrak{g}\to \mathfrak{g}$. We will sometimes refer to the involution $\Lambda$ or $\lambda$ itself as a real form. 

\paragraph{Cartan decomposition and Cartan involution}

The maximal compact subgroup of $G_\mathbb{R}$ is unique up to conjugation. A choice of maximal compact subgroup $H_{\mathbb{R}}\leqslant G_\mathbb{R}$ defines a Lie algebra involution $\tau\colon\mathfrak{g}_{\mathbb{R}}\to\mathfrak{g}_{\mathbb{R}}$ called \textbf{Cartan involution}. The $(+1)$-eigenspace of the Cartan involution is the Lie algebra $\mathfrak{h}_{\mathbb{R}}$ of the maximal compact $H_{\mathbb{R}}$ and the $(-1)$-eigenspace $\mathfrak{m}_{\mathbb{R}}$ is the subspace perpendicular to $\mathfrak{h}_{\mathbb{R}}$ with respect to the Killing form $B$. Hence a choice of maximal compact defines a Cartan decomposition $\mathfrak{g}_{\mathbb{R}} = \mathfrak{h}_{\mathbb{R}} \oplus \mathfrak{m}_{\mathbb{R}}$. A Cartan decomposition
satisfies the bracket relations
\[[\mathfrak{h}_{\mathbb{R}},\mathfrak{h}_{\mathbb{R}}]\subseteq\mathfrak{h}_{\mathbb{R}},[\mathfrak{h}_{\mathbb{R}},\mathfrak{m}_{\mathbb{R}}]\subseteq\mathfrak{m}_{\mathbb{R}},[\mathfrak{m}_{\mathbb{R}},\mathfrak{m}_{\mathbb{R}}]\subseteq\mathfrak{h}_{\mathbb{R}}.\]
Furthermore, the form
$B_\tau(\bullet,\ast):=-B(\tau(\bullet),\ast)$ is positive definite and $\operatorname{Ad}(H_\mathbb{R})$-invariant.

\paragraph{Parabolic subgroups and Levi subgroups}

We consider the Cartan decomposition $\mathfrak{g}_\mathbb{R}=\mathfrak{h}_\mathbb{R}\oplus\mathfrak{m}_\mathbb{R}$ of a semisimple real Lie group $G_\mathbb{R}$. Let $\mathsf{m}\in\mathfrak{m}_\mathbb{R}$ which is a semisimple element. We define the corresponding parabolic subgroup
\[P_{\mathsf{m}}:=\{g\in G_\mathbb{R}\mid \operatorname{Ad}(\exp(t\mathsf{m}))(g)\mbox{ is bounded as }t\to\infty\}\leqslant G_\mathbb{R}\]
and the corresponding Levi subgroup
\[L_{\mathsf{m}}:=\{g\in G_\mathbb{R}\mid \operatorname{Ad}(g)(\mathsf{m})=\mathsf{m}\}\leqslant G_\mathbb{R},\]
which is a reductive Lie group. Their Lie algebras are given by
\[\mathfrak{p}_{\mathsf{m}}:=\{\mathsf{x}\in \mathfrak{g}_\mathbb{R}\mid \operatorname{Ad}(\exp(t\mathsf{m}))(\mathsf{x})\mbox{ is bounded as }t\to\infty\},\quad\ \mathfrak{l}_{\mathsf{m}}:=\{\mathsf{x}\in \mathfrak{g}_\mathbb{R}\mid [\mathsf{m},\mathsf{x}]=0\}.\]
For any parabolic subalgebra $\mathfrak{p}$, we call $\mathsf{m}\in\mathfrak{m}_{\mathbb{R}}$ an \textbf{antidominant element} if $\mathfrak{p}_\mathsf{m}\supseteq\mathfrak{p}$ and a \textbf{strictly antidominant element} if $\mathfrak{p}_\mathsf{m}=\mathfrak{p}$. Then the character $\chi_{\mathsf{m}}(\bullet):=B_{\tau}(\mathsf{m},\bullet)$ induced by a (strictly) antidominant element is called a \textbf{(strictly) antidominant character} of $\mathfrak{p}$.

In particular, let us consider the complex Lie group $G$ with the Cartan involution $\tau$, its Cartan decomposition $\mathfrak{g}=\mathfrak{h}\oplus\iu\mathfrak{h}$ and a real form $G_\mathbb{R}$ which is fixed by an anti-holomorphic involution $\Lambda$ compatible with the Cartan involution, i.e. $\Lambda\tau=\tau\Lambda$. Then $\tau$ is also a Cartan involution of $G_\mathbb{R}$. Furthermore, any Levi subgroup $L_\mathsf{m}$ of $G^\mathbb{R}$ is the same as the intersection of a $\Lambda$-invariant Levi subgroup and $G_\mathbb{R}$.

\paragraph{Root space decomposition}

Let $\mathfrak{t}$ be a Cartan subalgebra of $\mathfrak{g}$. Recall the following root space decomposition
\[\mathfrak{g}=\mathfrak{t}\oplus\bigoplus_{\alpha\in\Delta(\mathfrak{g},\mathfrak{t})}\mathfrak{g}_\alpha,\]
where $\Delta(\mathfrak{g},\mathfrak{t})\subseteq\mathfrak{t}^\vee:=\operatorname{Hom}_{\mathbb{C}}(\mathfrak{t},\mathbb{C})$ is the set of roots and
\[\mathfrak{g}_\alpha:=\{\mathsf{x}\in\mathfrak{g}\mid[\mathsf{t},\mathsf{x}]=\alpha(\mathsf{t})\cdot\mathsf{x}, \forall\mathsf{t}\in\mathfrak{t}\}.\]
Note that $\dim_{\mathbb{C}}(\mathfrak{g}_\alpha)=1$. We fix a choice of the positive roots $\Delta^+\subsetneq\Delta$ and the associated simple roots $\Pi$.

\subsection{\texorpdfstring{$\mathbb{Z}/m\mathbb{Z}$}{Z/mZ}-gradings of semisimple Lie algebras}\label{sec:grading}

We recommend reading \cite[Chapter \uppercase\expandafter{\romannumeral10}, \S5]{helgason1979differential} and \cite{garcia2024cyclic} for references. 

\begin{definition}
    Let $m\in\mathbb{N}^*$ be a positive integer. A $\mathbb{Z}/m\mathbb{Z}$-grading of $\mathfrak{g}$ is a decomposition as a direct sum of vector subspaces 
    \[\mathfrak{g}=\bigoplus_{\overline{j}\in\mathbb{Z}/m\mathbb{Z}}\mathfrak{g}_{\overline{j}}\]
    such that $[\mathfrak{g}_{\overline{j}},\mathfrak{g}_{\overline{k}}]\subseteq\mathfrak{g}_{\overline{j+k}}$.
\end{definition}

Let us fix a primitive $m$-th root of unity $\zeta_m=\exp(2\pi\iu/m)\in\mathbb{C}^*$. There is a correspondence between $\mathbb{Z}/m\mathbb{Z}$-gradings of $\mathfrak{g}$ and order $m$ automorphisms $\theta\in\operatorname{Aut}_m(\mathfrak{g})$. For instance, given a $\mathbb{Z}/m\mathbb{Z}$-grading $\mathfrak{g}=\bigoplus_{\overline{j}\in\mathbb{Z}/m\mathbb{Z}}\mathfrak{g}_{\overline{j}}$, then we can define an order $n$ automorphism $\theta$ by the rule $\theta|_{\mathfrak{g}_{\overline{j}}}\equiv \zeta_m^j\operatorname{id}_{\mathfrak{g}_{\overline{j}}}$. 
Conversely, given $\theta\in\operatorname{Aut}_m(\mathfrak{g})$, we obtain a $\mathbb{Z}/m\mathbb{Z}$-grading by taking the eigenspace
decomposition, that is, setting $\mathfrak{g}_{\overline{j}}^\theta:= \{X\in\mathfrak{g}\mid\theta X =\zeta_m^jX\}$. We call $\theta$ the \textbf{grading automorphism}.

The bracket relation $[\mathfrak{g}_{\overline{0}},\mathfrak{g}_{\overline{0}}]\subseteq \mathfrak{g}_{\overline{0}}$ implies that $\mathfrak{g}_{\overline{0}}$ is a Lie subalgebra of $\mathfrak{g}$, so that there is a corresponding connected subgroup $G_{\overline{0}}\subseteq G$, which is reductive. Moreover, $[\mathfrak{g}_{\overline{0}},\mathfrak{g}_{\overline{j}}]\subseteq \mathfrak{g}_{\overline{j}}$ implies
that the adjoint action of $G$ on $\mathfrak{g}$ restricts to a representation $G_{\overline{0}}\to\mathrm{GL}(\mathfrak{g}_{\overline{j}})$ for every $\overline{j}\in\mathbb{Z}/m\mathbb{Z}$. These representations were studied by Vinberg \cite{vinberg1976weyl} and the pairs $(G_{\overline{0}},\mathfrak{g}_{\overline{j}})$ are called \textbf{Vinberg $\theta$-pairs}, where $\theta\in\operatorname{Aut}_m(\mathfrak{g})$ denotes the grading automorphism. We will very
often assume that the $\mathbb{Z}/m\mathbb{Z}$-grading on $\mathfrak{g}$ is induced from an element of $\operatorname{Aut}(G)$ of order $m$, denoted by $\Theta$. In such case we denote by $G^\Theta$ the subgroup of fixed points under the action of $\Theta$ on $G$. The group $G^\Theta$ is reductive and has $G_{\overline{0}}$ as the identity component. Note that $\mathfrak{g}_{\overline{j}}$ is invariant under the action of $G^\Theta$. The pairs $(G^\Theta,\mathfrak{g}_{\overline{j}})$ are called \textbf{extended Vinberg $\theta$-pairs}.

Denote by $\operatorname{Int}(\mathfrak{g})$ the inner automorphism group of $\mathfrak{g}$. It is known that outer automorphism group $\operatorname{Out}(\mathfrak{g}):=\operatorname{Aut}(\mathfrak{g})/\operatorname{Int}(\mathfrak{g})$ is isomorphic to the automorphism group $\operatorname{Aut}(\Pi)$ of simple roots (or equivalently, isometry group of $\Pi$ with respect to the dual of the Killing form or the automorphism group of the Dynkin diagram). Thus $\operatorname{Out}(\mathfrak{g})\cong\mathfrak{S}_q$ which is the permutation group of $q$ elements for $q\in\{1,2,3\}$ (the latter only occurs in $\mathfrak{d}_4$-type, corresponding to $\mathfrak{so}_8\mathbb{C}$).

The proof of the following proposition can be found in \cite[Chapter \uppercase\expandafter{\romannumeral10}, Lemma 5.2]{helgason1979differential} or \cite[Theorem 3.72]{wallach2017geometric}.

\begin{proposition}\label{prop:conju}
    There exists a conjugate-linear involution $\tau\colon\mathfrak{g}\to\mathfrak{g}$, which is a Cartan involution when regarding $\mathfrak{g}$ as a real Lie algebra, such that $\tau\theta=\theta\tau$. As a corollary, $\tau(\mathfrak{g}_{\overline{j}})=\mathfrak{g}_{\overline{-j}}$.
\end{proposition}

In particular, when $m=2k$ is even, we have $\theta^k\in\operatorname{Aut}_2(\mathfrak{g})$ an involution. This gives us the conjugate-linear involution $\lambda:=\tau\circ\theta^k\colon\mathfrak{g}\to\mathfrak{g}$, hence also a real form $\mathfrak{g}^\lambda$, which is the subalgebra fixed by $\lambda$. 

Let $H\leqslant G$ be the corresponding maximal compact subgroup of $G$ and $H_{\overline{0}}:=H\cap G_{\overline{0}}$ the maximal compact subgroup of $G_{\overline{0}}$. Denote the corresponding Cartan decomposition of $\mathfrak{g}_{\overline{0}}$ by $\mathfrak{h}_{\overline{0}}\oplus\mathfrak{m}_{\overline{0}}$. 
We get the direct decompositions into semisimple and abelian
ideals
\[\mathfrak{h}_{\overline{0}}=[\mathfrak{h}_{\overline{0}}, \mathfrak{h}_{\overline{0}}]\oplus\mathfrak{z}_{\mathfrak{h}_{\overline{0}}},\qquad\mathfrak{g}_{\overline{0}}=[\mathfrak{g}_{\overline{0}}, \mathfrak{g}_{\overline{0}}]\oplus\mathfrak{z}_{\mathfrak{g}_{\overline{0}}}.\]
Taking a maximal abelian subalgebra $(\mathfrak{t}_{\overline{0}}^\mathfrak{\tau})'$ of $[\mathfrak{h}_{\overline{0}}, \mathfrak{h}_{\overline{0}}]$, the algebra $\mathfrak{t}_{\overline{0}}^\mathfrak{\tau}=(\mathfrak{t}_{\overline{0}}^\mathfrak{\tau})'\oplus\mathfrak{z}_{\mathfrak{h}_{\overline{0}}}$, is maximal abelian in $\mathfrak{h}_{\overline{0}}$
and its complexification $\mathfrak{t}_{\overline{0}}$ is maximal abelian in $\mathfrak{g}_{\overline{0}}$. Then the centralizer $\mathfrak{z}_{\mathfrak{g}}(\mathfrak{t}_{\overline{0}})$ of $\mathfrak{t}_{\overline{0}}$ in $\mathfrak{g}$ is a Cartan subalgebra in $\mathfrak{g}$. In particular, $\mathfrak{g}_{\overline{0}}\neq0$.

A nonzero pair $\boldsymbol{\alpha}=(\alpha,\overline{j})\in\mathfrak{t}_{\overline{0}}^\vee\times\mathbb{Z}/m\mathbb{Z}$ will be called a \textbf{root of $\mathfrak{g}$ with respect to $\mathfrak{t}_{\overline{0}}$} if the joint eigenspace
\[\mathfrak{g}_{\boldsymbol{\alpha}}:=\{\mathsf{x}\in\mathfrak{g}_{\overline{j}}\mid[\mathsf{t},\mathsf{x}]=\alpha(\mathsf{t})\cdot\mathsf{x},\forall\mathsf{t}\in\mathfrak{t}_{\overline{0}}\}\]
is nonzero. Hence $[\mathfrak{g}_{\boldsymbol{\alpha}},\mathfrak{g}_{\boldsymbol{\beta}}]\subseteq\mathfrak{g}_{\boldsymbol{\alpha}+\boldsymbol{\beta}}$. Denote by $\widetilde{\Delta}(\mathfrak{g},\mathfrak{t}_{\overline{0}})$ the set of all roots of $\mathfrak{g}$ with respect to $\mathfrak{t}_{\overline{0}}$ and by $\widetilde{\Delta}^0(\mathfrak{g},\mathfrak{t}_{\overline{0}})$ the set of all roots of $\mathfrak{g}$ with respect to $\mathfrak{t}_{\overline{0}}$ of the form $(0,\overline{j})$ for some $\overline{j}$.

We have the following graded root space decomposition:

\begin{equation}\label{eq:root}
    \mathfrak{g}=\mathfrak{t}_{\overline{0}}\oplus\bigoplus_{\boldsymbol{\alpha}\in\widetilde{\Delta}(\mathfrak{g},\mathfrak{t}_{\overline{0}})}\mathfrak{g}_{\boldsymbol{\alpha}}=\mathfrak{z}_{\mathfrak{g}}(\mathfrak{t}_{\overline{0}})\oplus\bigoplus_{\boldsymbol{\alpha}\in\widetilde{\Delta}(\mathfrak{g},\mathfrak{t}_{\overline{0}})\setminus\widetilde{\Delta}^0(\mathfrak{g},\mathfrak{t}_{\overline{0}})}\mathfrak{g}_{\boldsymbol{\alpha}}.
\end{equation}
For any root $\boldsymbol{\alpha}\in\widetilde{\Delta}(\mathfrak{g},\mathfrak{t}_{\overline{0}})\setminus\widetilde{\Delta}^0(\mathfrak{g},\mathfrak{t}_{\overline{0}})$, $\dim_{\mathbb{C}}\mathfrak{g}_{\boldsymbol{\alpha}}=1$ by \cite[Chapter \uppercase\expandafter{\romannumeral10}, Lemma 5.4]{helgason1979differential} and $\tau\left(\mathfrak{g}_{\boldsymbol{\alpha}}\right)=\mathfrak{g}_{-\boldsymbol{\alpha}}$. 

When $\theta$ is induced by $\Theta\in\operatorname{Aut}_m(G)$, we denote by $H^\Theta:=H\cap G^\Theta$ the maximal compact subgroup of $G^\Theta$. Similarly, when $m=2k$ is even, we have $\Theta^k\in\operatorname{Aut}_2(G)$, the anti-holomorphic involution $\Lambda:=\tau\circ\Theta^k$ and the real form $G^\Lambda<G$, which is the subgroup fixed by $\Lambda$. One can readily check that $H^\Theta<G^\Lambda$.

\begin{definition}\label{defn:cyclicroot}
    We say that $\boldsymbol{\alpha}=(\alpha,\overline{1})\in\widetilde{\Delta}(\mathfrak{g},\mathfrak{t}_{\overline{0}})\setminus\widetilde{\Delta}^0(\mathfrak{g},\mathfrak{t}_{\overline{0}})$ is a \textbf{$\Theta$-cyclic root} if $\mathfrak{g}_{\boldsymbol{\alpha}}$ is $\operatorname{Ad}(H^\Theta)$-invariant. Denote by $\widetilde{\Delta}_{\overline{1}}^\Theta(\mathfrak{g},\mathfrak{t}_{\overline{0}})$ the set of $\Theta$-cyclic roots.
\end{definition}

We have the $\operatorname{Ad}(H^\Theta)$-invariant decomposition
\[
\mathfrak{g}=\mathfrak{h}_{\overline{0}}\oplus\mathfrak{m}_{\overline{0}}\oplus\bigoplus_{\substack{\overline{j}\in\mathbb{Z}/m\mathbb{Z}\\\overline{j}\neq\overline{0}}}\mathfrak{g}_{\overline{j}}=\mathfrak{h}_{\overline{0}}\oplus\mathfrak{h}_{\overline{0}}^\perp.
\]

Let 
\[\mathfrak{g}_Z:=\bigoplus_{\substack{\overline{j}\in\mathbb{Z}/m\mathbb{Z}\\\overline{j}\neq\overline{0},\overline{1},-\overline{1}}}\mathfrak{g}_{\overline{j}}.\]
Then since $m\geqslant3$, 
\begin{equation}\label{eq:bracket}
    \begin{cases}
    [\mathfrak{g}_{\overline{1}},\mathfrak{g}_{\overline{1}}]\subseteq\mathfrak{g}_{-\overline{1}}\oplus\mathfrak{g}_Z,\\
    [\mathfrak{g}_{-\overline{1}},\mathfrak{g}_{-\overline{1}}]\subseteq\mathfrak{g}_{\overline{1}}\oplus\mathfrak{g}_Z,\\
    [\mathfrak{g}_{\overline{1}},\mathfrak{g}_{Z}]\subseteq\mathfrak{g}_{-\overline{1}}\oplus\mathfrak{g}_Z,\\
    [\mathfrak{g}_{-\overline{1}},\mathfrak{g}_{Z}]\subseteq\mathfrak{g}_{\overline{1}}\oplus\mathfrak{g}_Z
\end{cases}
\end{equation}

We denote by $\operatorname{pr}_{\overline{j}}$, $\operatorname{pr}_{\mathfrak{h}_{\overline{0}}}$ and $\operatorname{pr}_{\mathfrak{m}_{\overline{0}}}$, $\operatorname{pr}_Z$, $\operatorname{pr}_{\overline{1},-\overline{1}}$ the projection onto $\mathfrak{g}_{\overline{j}}$, $\mathfrak{h}_{\overline{0}}$ and $\mathfrak{m}_{\overline{0}}$, $\mathfrak{g}_Z$, $\mathfrak{g}_{\overline{1}}\oplus\mathfrak{g}_{-\overline{1}}$ with respect to the above decomposition respectively. Similarly, for
any $\Theta$-cyclic root $\boldsymbol{\alpha}\in\widetilde{\Delta}_{\overline{1}}^\Theta(\mathfrak{g},\mathfrak{t}_{\overline{0}})$, we denote by $\operatorname{pr}_{\boldsymbol{\alpha}}$ the orthogonal projection onto $\mathfrak{g}_{\boldsymbol{\alpha}}$ with respect to the inner product $B_\tau$. Note that the root space decomposition is orthogonal with respect to $B_\tau$, hence $\operatorname{pr}_{\boldsymbol{\alpha}}$ is the same as the projection onto $\mathfrak{g}_{\boldsymbol{\alpha}}$ with respect to the graded root space decomposition \prettyref{eq:root}.

\subsection{Geometry of cyclic space and other homogeneous spaces}\label{sec:cyclicspace}

Recall the notations in \prettyref{sec:grading}. Below we will always assume that $\Theta$ is an order $m$ automorphism of the connected complex semisimple Lie group $G$, where $m\geqslant3$. Let $U$ be a reductive subgroup of $H^\Theta$ with Lie algebra $\mathfrak{u}$. We will mainly consider the principal $U$-bundle $G\to G/U$ in this section. But we will focus on the following special case later:
\begin{definition}
    The \textbf{$\Theta$-cyclic space} $\mathbb{X}_{\Theta}$ is the homogeneous space $G/H^\Theta$.
\end{definition}

For any linear action of $U$ on a vector space $V$, we denote by $[V]:=G\times_{U}V$ the associated vector bundle over $G/U$.

Let $\omega_{MC}\in\mathcal{A}^1(G,[\mathfrak{g}])$ be the Maurer--Cartan form on $G$. The Maurer--Cartan equation is then
\begin{equation}\label{eq:MCtotal}
    \mathrm{d}\omega_{MC}+\dfrac{1}{2}[\omega_{MC}\wedge\omega_{MC}]=0.
\end{equation}
Since the decomposition
\[
\mathfrak{g}=\mathfrak{h}_{\overline{0}}\oplus\mathfrak{m}_{\overline{0}}\oplus\bigoplus_{\substack{\overline{j}\in\mathbb{Z}/m\mathbb{Z}\\\overline{j}\neq\overline{0}}}\mathfrak{g}_{\overline{j}}=\mathfrak{h}_{\overline{0}}\oplus\mathfrak{h}_{\overline{0}}^\perp=\mathfrak{u}\oplus\mathfrak{u}^\perp.
\]
is $\operatorname{Ad}(U)$-invariant, the projections $\operatorname{pr}_{\overline{j}}$, $\operatorname{pr}_{\mathfrak{m}_{\overline{0}}}$, $\operatorname{pr}_{\overline{1},-\overline{1}}$, $\operatorname{pr}_{\mathfrak{u}}$ extend to the differential forms taking value in $[\mathfrak{g}]$ naturally. 

\begin{definition}
    We call a one-dimensional subspace $\mathfrak{c}\subseteq\mathfrak{g}_{\overline{1}}$ a \textbf{$U$-cyclic subspace} if $\mathfrak{c}$ is $\operatorname{Ad}(U)$-invariant.
\end{definition}
Note that $\tau(\mathfrak{c})\in\mathfrak{g}_{\overline{1}}$ is also $\operatorname{Ad}(U)$-invariant and every $\Theta$-cyclic root $\boldsymbol{\alpha}$ induces an $H^\Theta$-cyclic subspace $\mathfrak{g}_{\boldsymbol{\alpha}}$. For a $U$-cyclic subspace $\mathfrak{c}$, since $\mathfrak{c}$ and $B_\tau$ are $\operatorname{Ad}(U)$-invariant, we can define the line subbundle $[\mathfrak{c}]\subseteq[\mathfrak{g}_{\overline{1}}]$ and $\operatorname{pr}_{\mathfrak{c}}$ extend to the differential forms taking value in $[\mathfrak{g}]$ naturally. And we can postcompose $\omega_{MC}$ with the projection on each factor. The projection of $\omega_{MC}$ onto $\mathfrak{u}$ gives a connection $A$ on the $U$-principal bundle $G\to G/U$. The $1$-form $\omega:=\omega_{MC}-A$ vanishes on the vertical tangent space of $G\to G/U$ and thus descends to a nowhere vanishing $1$-form on $G/U$ with values in the associated bundle $[\mathfrak{u}^{\perp}]$. In particular, the $\omega$ identifies the tangent bundle $\mathrm{T}(G/U)$ with $[\mathfrak{u}^{\perp}]$. 
%Furthermore, we have the decomposition
%\[
%\mathrm{T}\mathbb{X}_{\Theta}=[\mathfrak{m}_{\overline{0}}]\oplus\bigoplus_{\substack{\overline{j}\in\mathbb{Z}/m\mathbb{Z}\\\overline{j}\neq\overline{0}}}[\mathfrak{g}_{\overline{j}}].
%\]
By pulling back the associated connection on $[\mathfrak{u}^{\perp}]$ given by $A$ through the identification $\omega$, we obtain the canonical connection $\nabla^{\mathrm{can}}$ on $\mathrm{T}(G/U)$. Given $\mathsf{x},\mathsf{y}\in\mathfrak{u}^{\perp}$, we have the associated left invariant vector fields $\mathsf{x}^*,\mathsf{y}^*$ on $G/U$. 

The canonical connection $\nabla^{\mathrm{can}}$ satisfies that $\nabla^{\mathrm{can}}_{\mathsf{x}^*}\mathsf{y}^*=0$ and the torsion is given by $T^{\mathrm{can}}(\mathsf{x}^*,\mathsf{y}^*)=-[\mathsf{x}^*,\mathsf{y}^*]$. Denote by $T^{\mathrm{can}}_{\overline{1},-\overline{1}}$ the projection of the torsion on to $[\mathfrak{g}_{\overline{1}}\oplus\mathfrak{g}_{-\overline{1}}]$.

Let $\omega_{\overline{j}}:=\operatorname{pr}_{\overline{j}}(\omega)\in\mathcal{A}^{1}(G/U,[\mathfrak{g}_{\overline{j}}])$. Indeed $\omega_{\overline{0}}=\operatorname{pr}_{\mathfrak{m}_{\overline{0}}}(\omega)\in\mathcal{A}^1(G/U,[\mathfrak{m}_{\overline{0}}])$. Similarly, for any $U$-cyclic subspace $\mathfrak{c}$, let $\omega_{\mathfrak{c}}:=\operatorname{pr}_{\mathfrak{c}}(\omega)\in\mathcal{A}^{1}(G/U,[\mathfrak{c}])$. 
Projecting
Maurer--Cartan equation on $[\mathfrak{u}]$, $[\mathfrak{m}_{\overline{0}}]$ and $[\mathfrak{g}_{\overline{j}}]$ for nonzero $\overline{j}\in\mathbb{Z}/m\mathbb{Z}$, we get

\begin{equation}\label{eq:MC}
    \begin{cases}
    F_{A}+\dfrac{1}{2}\sum\limits_{\overline{j}\in\mathbb{Z}/m\mathbb{Z}}\operatorname{pr}_{\mathfrak{u}}\left([\omega_{\overline{j}}\wedge\omega_{-\overline{j}}]\right)=0,\\
    \mathrm{d}^A\omega_{\overline{0}}+\dfrac{1}{2}\sum\limits_{\overline{j}\in\mathbb{Z}/m\mathbb{Z}}\operatorname{pr}_{\mathfrak{m}_{\overline{0}}}\left([\omega_{\overline{j}}\wedge\omega_{-\overline{j}}]\right)=0,\\
    \mathrm{d}^A\omega_{\overline{j}}+\dfrac{1}{2}\sum\limits_{\overline{k}\in\mathbb{Z}/m\mathbb{Z}}[\omega_{\overline{k}}\wedge\omega_{\overline{j-k}}]=0, \quad\forall \mbox{ nonzero }\overline{j}\in\mathbb{Z}/m\mathbb{Z},
\end{cases}
\end{equation}
where $F_A$ denotes the curvature form of the connection $A$.

Note that the involution $\tau$ on $\mathfrak{g}$ also extends to $[\mathfrak{g}]\cong [\mathfrak{u}]\oplus \mathrm{T}(G/U)$ and $\mathcal{A}^i(G/U,[\mathfrak{g}])$.

\begin{definition}
    The \textbf{complex $U$-cyclic distribution} $\mathcal{D}_U\subset \mathrm{T}(G/U)$ is the intersection of the fixed points of $-\tau$ and $[\mathfrak{g}_{\overline{1}}\oplus\mathfrak{g}_{-\overline{1}}]$. When $U=H^\Theta$, we simply call $\mathcal{D}:=\mathcal{D}_{H^\Theta}$ the complex cyclic distribution.
\end{definition}

When $m=2k$ is even, we have the anti-holomorphic involution $\Lambda=\tau\circ\Theta^k$ and its derivative $\lambda$. Then $H^\Theta<G^\Lambda$ and the real form $\mathfrak{g}^\lambda$ is invariant under the adjoint action of $H^\Theta$ since $H^\Theta$ is fixed by both $\Theta$ and $\tau$. Therefore we can define the associated vector bundle $[\mathfrak{g}^\lambda]$ over $G/U$ for any reductive subgroup $U<H^\Theta$. On the other hand, $[\mathfrak{g}^\lambda]$ can be identified with the tangent bundle of the $G^\Lambda$-orbits in $G/U$ via the Maurer--Cartan form. In such case, the complex $U$-cyclic distribution $\mathcal{D}_U$ is tangent to $[\mathfrak{g}^\lambda]$ since both $\tau$ and $\Theta$ acts as $-1$ on the tangent of $\mathcal{D}_U$.

The $\operatorname{Ad}(U)$-invariant automorphism 
\[\begin{aligned}
    \mathcal{J}\colon\mathfrak{g}_{\overline{1}}\oplus\mathfrak{g}_{-\overline{1}}&\to\mathfrak{g}_{\overline{1}}\oplus\mathfrak{g}_{-\overline{1}}\\
    (\mathsf{x},\mathsf{y})&\mapsto(\iu\mathsf{x},-\iu\mathsf{y})
\end{aligned}\]
commutes with $\tau$ and $\mathcal{J}^2=-\operatorname{id}_{\mathfrak{g}_{\overline{1}}\oplus\mathfrak{g}_{-\overline{1}}}$. Thus, it induces an almost complex structure $\mathcal{J}$ on $\mathcal{D}_U$.

\section{Higgs bundles}\label{sec:Higgs bundles}

In this section, we recall the definition of (cyclic) Higgs bundles and Hitchin--Kobayashi correspondence. We recommend reading \cite{garcia2009hitchin}, \cite{garcia2019involutions} and \cite[Section 3]{garcia2024cyclic} for references. 

We fix a Riemann surface $X=(\Sigma,\mathsf{j})$ which admits a complete hyperbolic conformal metric and $\mathcal{K}_X$ is its
canonical line bundle. Note that $X$ is not assumed to be compact.

\subsection{Higgs pairs and stability}

Let $\widehat{G}$ be a connected complex reductive Lie group. Given a holomorphic principal $\widehat{G}$-bundle $\mathbb{E}$ over $X$. If $M$ is any set on which $\widehat{G}$ acts on the left, we denote by $\mathbb{E}[M]$ the twisted product $\mathbb{E}\times_{\widehat{G}} M$. If $M$ is a vector space (resp. complex variety) and the action of $\widehat{G}$ on $M$
is linear (resp. holomorphic) then $\mathbb{E}[M]\to X$ is a vector bundle (resp. holomorphic
fibration).

\begin{definition}
    Given a complex linear space $V$ with a holomorphic linear $\widehat{G}$-action. A \textbf{$(\widehat{G},V)$-Higgs pair} is a pair $(\mathbb{E},\varphi)$, where $\mathbb{E}$ is a holomorphic principal $\widehat{G}$-bundle on $X$, and $\varphi\in\mathrm{H}^0(X,\mathbb{E}[V]\otimes\KK_X)$. 
\end{definition}

When $V=\widehat{\mathfrak{g}}$ the Lie algebra of $\widehat{G}$ and $\widehat{G}$ acts on it by the adjoint action, the resulting $(\widehat{G},\widehat{\mathfrak{g}})$-Higgs pairs are called \textbf{$\widehat{G}$-Higgs bundles}.

\begin{definition}
    Let $U\subseteq\widehat{G}$ be a Lie subgroup, and $\mathbb{E}$ a principal $\widehat{G}$-bundle. A holomorphic reduction of structure group of $\mathbb{E}$ to $U$ is a holomorphic section $\sigma\in\mathrm{H}^0(X,\mathbb{E}[\widehat{G}/U])$.
\end{definition}

The natural map $\mathbb{E}\to\mathbb{E}[\widehat{G}/U]$ has a $U$-bundle structure. Thus, it is possible to pull back the $U$-bundle $\mathbb{E}\to\mathbb{E}[\widehat{G}/U]$ to get $\mathbb{E}_\sigma:=\sigma^*\mathbb{E}$, a $U$-bundle over $X$. Moreover, there is a canonical isomorphism $\mathbb{E}_\sigma[G]\cong\mathbb{E}$. The map $\mathbb{E}_\sigma= \sigma^*\mathbb{E}\to\mathbb{E}$ induced by the pullback is injective and gives a holomorphic subvariety $\mathbb{E}_\sigma\subseteq\mathbb{E}$.

\begin{definition}
    For a $\widehat{G}$-Higgs bundle $(\mathbb{E},\varphi)$ over $X$, its \textbf{infinitesimal automorphism space} is defined as
\[\operatorname{aut}(\mathbb{E},\varphi):=\{s\in\mathrm{H}^0(X,\mathbb{E}[\widehat{\mathfrak{g}}])\mid[s\wedge \varphi]=0\}.\]
\end{definition}

When $X$ is compact, we have the notions
of (semi,poly)stability for $\widehat{G}$-Higgs bundles. Our approach follows \cite{garcia2009hitchin}, where all these general notions are studied in detail. For our convenience, we only consider the case for the connected semisimple complex Lie group $\widehat{G}=G$.

Given $\mathsf{h}\in\iu\mathfrak{h}$ and a holomorphic reduction $\sigma$ of structure group of $\mathbb{E}$ to $P_\mathsf{h}$, we recall the definition of the \textbf{degree} of the $\widehat{G}$-bundle $\mathbb{E}$ with respect to $\sigma$ and $\mathsf{h}$. For this, define $H_{\mathsf{h}} =H\cap L_{\mathsf{h}}$ and $\mathfrak{h}_{\mathsf{h}} =\mathfrak{h}\cap \mathfrak{l}_{\mathsf{h}}$. Then, $H_{\mathsf{h}}$ is a maximal compact subgroup of $L_{\mathsf{h}}$, so the inclusion $H_{\mathsf{h}}\hookrightarrow L_{\mathsf{h}}$ is a homotopy equivalence. Since the inclusion $L_{\mathsf{h}}\hookrightarrow P_{\mathsf{h}}$ is also a homotopy equivalence, given a reduction $\sigma$ of the structure group of $\mathbb{E}$ to $P_{\mathsf{h}}$, one can further restrict the structure group of $\mathbb{E}$ to $L_{\mathsf{h}}$ in a unique way up to homotopy. Denote by $\mathbb{E}_\sigma'$ the resulting principal $H_{\mathsf{h}}$-bundle. Consider now a connection $A$ on $\mathbb{E}_\sigma'$ and let $F(A)\in\mathcal{A}^2(X,\mathbb{E}_\sigma'[\mathfrak{h}_{\mathsf{h}}])$ be its curvature. Then, $\chi_{\mathsf{h}}(F(A))$ is a $2$-form on $X$ with values in $\iu\mathbb{R}$, and
\[\deg\mathbb{E}(\sigma,\mathsf{h}):=\dfrac{\iu}{2\pi}\int_X\chi_{\mathsf{h}}(F(A))\]
is well-defined by Chern--Weil theory.

\begin{definition}
    A $G$-Higgs bundle $(\mathbb{E},\varphi)$ over a compact Riemann surface $X$ is:
    \begin{itemize}
        \item \textbf{semistable} if $\deg\mathbb{E}(\sigma,\mathsf{h})\geqslant 0$, for any parabolic subgroup $P$ of $G$, any non-trivial antidominant element $\mathsf{h}$ of $P$ and any reduction of structure group $\sigma$ of $\mathbb{E}$ to $P$ such that $\varphi\in\mathrm{H}^0(X,\mathbb{E}_\sigma[\mathfrak{p}]\otimes\KK_X)$;

        \item \textbf{stable} if $\deg\mathbb{E}(\sigma,\mathsf{h})>0$, for any non-trivial parabolic subgroup $P$ of $G$, any non-trivial antidominant element $\mathsf{h}$ of $P$ and any reduction of structure group $\sigma$ of $\mathbb{E}$ to $P$ such that $\varphi\in\mathrm{H}^0(X,\mathbb{E}_\sigma[\mathfrak{p}]\otimes\KK_X)$;

        \item \textbf{polystable} if it is semistable and if $\deg\mathbb{E}(\sigma,\mathsf{h})=0$, for some parabolic subgroup $P$, some non-trivial strictly antidominant element $\mathsf{h}$ of $P$ and some reduction of structure group $\sigma$ of $\mathbb{E}$ to $P$ such that $\varphi\in\mathrm{H}^0(X,\mathbb{E}_\sigma(\mathfrak{p})\otimes\KK_X)$, then there is a further holomorphic reduction of structure group $\sigma_L$ of $\mathbb{E}_\sigma$ to the Levi subgroup $L$ of $P$ such that $\varphi\in\mathrm{H}^0(X,\mathbb{E}_{\sigma_L}[\mathfrak{l}]\otimes\KK_X)$.
    \end{itemize}
\end{definition}

\begin{lemma}\label{lem:nonstable}
    Given a semistable $G$-Higgs bundle $(\mathbb{E},\varphi)$ over a compact Riemann surface $X$. If there exists a proper Levi subgroup $L<G$ and a holomorphic reduction of structure group $\sigma_L$ of $\mathbb{E}$ to $L$ such that $\varphi\in\mathrm{H}^0(X,\mathbb{E}_{\sigma_L}[\mathfrak{l}]\otimes\KK_X)$, then $(\mathbb{E},\varphi)$ is not stable.
\end{lemma}

\begin{proof}
    Let $L=L_{\mathsf{h}}$ for some $\mathsf{h}\in\iu\mathfrak{h}$. Then $L_{\mathsf{h}}=P_{\mathsf{h}}\cap P_{-\mathsf{h}}$. Hence $\sigma_L$ extends the holomorphic reductions of structure group $\sigma_{\pm}$ of $\mathbb{E}$ to $P_{\pm\mathsf{h}}$ respectively. Therefore, 
    \[0\geqslant-\deg\mathbb{E}(\sigma_{+},\mathsf{h})=\deg\mathbb{E}(\sigma_-,-\mathsf{h})\geqslant0\]
    by semistability and all ``$=$'' hold, which contradicts to the stablity.
\end{proof}

\subsection{Hitchin--Kobayashi correspondence}

Recall that a metric on $\mathbb{E}$ is a reduction of structure group from $G$ to $H$, which is equivalent to a global smooth section of $\mathbb{E}[G/H]$. 

\begin{definition}
    A \textbf{harmonic metric} of a $G$-Higgs bundle $(\mathbb{E},\varphi)$ is a metric $h$ on $\mathbb{E}$ satisfying that \begin{equation}\label{eq:harmonic}
        F(\nabla^h)-[\varphi\wedge\tau_h(\varphi)] = 0,
    \end{equation}
     or equivalently,
     \[\mathrm{D}^h =\nabla^h+\varphi-\tau_h(\varphi)\]
     is a flat $G$-connection on $\mathbb{E}$, where $\nabla^h$ denotes unique connection compatible with the holomorphic structure of $\mathbb{E}$ and the metric $h$, $F(\nabla^h)$ denotes its curvature and $\tau_h$ is the conjugation on $\mathcal{A}^\bullet(\mathbb{E}[\mathfrak{g}])$ defined by combining the metric $h$ and the Cartan involution $\tau$. 
\end{definition}

We call the triple $(\mathbb{E},\varphi,h)$ a \textbf{harmonic $G$-Higgs bundle} if $h$ is a harmonic metric of the $G$-Higgs bundle $(\mathbb{E},\varphi)$.

\begin{theorem}\label{thm:HK}
     A $G$-Higgs bundle $(\mathbb{E},\varphi)$ over a compact Riemann surface $X$ is polystable if and only if there exists a harmonic metric $h$ of $(\mathbb{E},\varphi)$. Moreover, such metric is unique (up to scaling) if and only if $(\mathbb{E},\varphi)$ is stable.
 \end{theorem}

\begin{remark}
    In \cite{garcia2009hitchin} and \cite[Section 3]{garcia2024cyclic}, the statement of the above theorem is under a more general set-up. For instance, they allow $G$ to be reductive, nonzero stability parameter and the Higgs pair twisted by an arbitrary holomorphic line bundle. 
\end{remark}

\subsection{Infinitesimal automorphism space via harmonic metric}

Let $(\mathbb{E},\varphi,h)$ be a harmonic $G$-Higgs bundle over $X$. In this section we study the infinitesimal automorphism space of $(\mathbb{E},\varphi)$ via the harmonic metric $h$. We usually denote by $\bullet^\dagger:=-\tau_h(\bullet)$.

Denote by $|\eta|_h^2:=B_{\tau_h}(\eta,\eta)$ the norm square of an arbitrary smooth section $\eta$ of $\mathbb{E}[\mathfrak{g}]$ with respect to the harmonic metric $h$. 

We first do some local computations. Let $(U,z)$ be a local coordinate chart of the Riemann surface $X$. Denote by $\Delta_z:=\partial_z\partial_{\bar z}$ the complex coordinate Laplacian. The following useful lemma is a slight modification of \cite[Lemma 2.4]{mochizuki2016asymptotic} and \cite[Lemma 3.3]{dai2024bounded}.

\begin{lemma}\label{lem:laplacian}
    Let $\eta$ be a local holomorphic section of $\mathbb{E}[\mathfrak{g}]$. Assume locally we have $\varphi=\phi\cdot\dd z$, $\nabla \eta=\eta_z'\cdot\dd z$. Then
        \[\Delta_z\log|\eta|_h^2=\dfrac{|[\phi^{\dagger}\wedge \eta]|_h^2-|[\phi\wedge \eta]|_h^2}{|\eta|_h^2}+\dfrac{|\eta_z'|_h^2\cdot|\eta|_h^2-|B_{\tau_h}(\eta_z',\eta)|^2}{|\eta|_h^4}.\]
\end{lemma}

\begin{proof}
    \[\begin{aligned}
        &\Delta_z\log|\eta|_h^2\\
        =&\partial_z\dfrac{B_{\tau_h}( \eta_z',\eta)}{|\eta|_h^2}\quad\mbox{(since }\eta\mbox{ is holomorphic)}\\
        =&\dfrac{\dfrac{B_{\tau_h}(\bar\partial^{\nabla}\partial^{\nabla}(\eta),\eta)}{\dd z\wedge\dd\bar z}+|\eta_z'|_h^2}{|\eta|_h^2}-\dfrac{|B_{\tau_h}(\eta_z',\eta)|^2}{|\eta|_h^4}\\
        =&\dfrac{B_{\tau_h}([[\phi\wedge \phi^{\dagger}]\wedge \eta],\eta)}{|\eta|_h^2}+\dfrac{|\eta_z'|_h^2\cdot|\eta|_h^2-|B_{\tau_h}(\eta_z',\eta)|^2}{|\eta|_h^4}\quad\mbox{(since }\nabla+\varphi+\varphi^{\dagger}\mbox{ is flat)}\\
        =&\dfrac{B_{\tau_h}([[\phi\wedge \eta]\wedge \phi^{\dagger}],\eta)-B_{\tau_h}([[\phi^{\dagger}\wedge \eta]\wedge \phi],\eta)}{|\eta|_h^2}+\dfrac{|\eta_z'|_h^2\cdot|\eta|_h^2-|B_{\tau_h}(\eta_z',\eta)|^2}{|\eta|_h^4}\quad\mbox{(by Jacobi identity)}\\
        =&\dfrac{|[\phi^{\dagger}\wedge \eta]|_h^2-|[\phi\wedge \eta]|_h^2}{|\eta|_h^2}+\dfrac{|\eta_z'|_h^2\cdot|\eta|_h^2-|B_{\tau_h}(\eta_z',\eta)|^2}{|\eta|_h^4}.
        \end{aligned}\]
\end{proof}

Let $g^{\mathcal{T}_X}$ be a conformal metric on $X$ with volume form $\vol_{g^{\mathcal{T}_X}}$. Let $B(x,r;g^{\mathcal{T}_X})$ be the geodesic ball centered at $x$ with radius $r$ on $(X,g^{\mathcal{T}_X})$. For any vector bundle $\VV$ equipped with an Hermitian metric $h_\VV$ over $X$. We denote by $L^p(X,\VV;g^{\mathcal{T}_X},h_\VV)$ the space of all sections whose $h_\VV$ norms are $L^p$-integrable on $X$ with respect to the metric $g^{\mathcal{T}_X}$.
%Also denote by $W^{1,p}(X,\VV;g^{\mathcal{T}_X},h_\VV)\subseteq L^p(X,g^{\mathcal{T}_X})$ the Sobolev space consisting of all $L^p$-sections whose derivative also has $L^p$-integrable norm with respect to the metrics $h_\VV$ and $g^{\mathcal{T}_X}$. 

In the following, we introduce a condition on $X$ which forces subharmonic functions to be constant.
\begin{definition}
    A Liouville-type pair on \(X\) is a pair $(g,\mathcal C)$ where $g$ is a complete conformal metric on $X$,
\(\mathcal C\) is a \(g\)-dependent (e.g. growth or integral) condition with
the property that every subharmonic nonnegative function satisfying
\(\mathcal C\) is constant. Moreover, we require that the condition $\mathcal C$ satisfies if $u\in \mathcal C$, then any function $v$ satisfying $v\leqslant u$ also lies in $\mathcal C.$
\end{definition}
Recall that a Riemann surface $X$ is defined to be potential-theoretically parabolic if every subharmonic function bounded above on $X$ is constant. For such Riemann surface, we can take the Liouville-type condition $\mathcal C$ to be ``bounded from above". In this case the condition is independent of the choice of metric.

\begin{remark}A typical example of potential-theoretically parabolic Riemann surface is the punctured surface $X = \overline{X} \setminus D$, where $\overline{X}$ is a compact Riemann surface and $D$ is a finite set of points. Furthermore, if $X$ is equipped with a complete Riemannian metric $g$, the boundedness requirement to construct a Liouville-type condition can be relaxed depending on the asymptotic geometry of its ends. \end{remark} 

\begin{corollary}\label{coro:Cauchy}
Let $\eta\in\operatorname{aut}(\mathbb{E},\varphi)$ be an infinitesimal automorphism. Then $\log|\eta|_h$ is subharmonic. Suppose that there exists a Liouville-type pair $(g,\mathcal C)$ on $X$ such that the function 
\(|\eta|_h^p\) satisfies the condition
\(\mathcal C\) for some $p>0$, then $|\eta|_h$ is constant and $\eta$ is antiholomorphic and $[\eta\wedge\varphi^\dagger]=0$.
\end{corollary}

\begin{proof}
    By \prettyref{lem:laplacian}, we obtain that on a local cooridinate chart $(U,z)$, 
    \[\Delta_z\log|\eta|_h^2=\dfrac{|[\phi^{\dagger}\wedge \eta]|_h^2}{|\eta|_h^2}+\dfrac{|\eta_z'|_h^2\cdot|\eta|_h^2-|B_{\tau_h}(\eta_z',\eta)|^2}{|\eta|_f^4}\geqslant0\]
    where the inequality follows from the Cauchy--Schwarz inequality. Therefore, \[\Delta_{g^{\mathcal{T}_X}}\log|\eta|_h\geqslant0,\]
    where $\Delta_{g^{\mathcal{T}_X}}$ is the Laplace–-Beltrami operator given by the conformal metric $\Delta_{g^{\mathcal{T}_X}}$.

        Since $\Delta_{g^{\mathcal{T}_X}}f^p\geqslant p f^p\cdot\Delta_{g^{\mathcal{T}_X}}\log f$ for any smooth nonnegative function $f$, we obtain that 
    for any $p>0,$\[\Delta_{g^{\mathcal{T}_X}}|\eta|_h^p\geqslant0.\]
    Thus $|\eta|_h^p$ is a non-negative subharmonic function on $X$.  By the assumption on $\eta$, we have $|\eta|_h^p$ is harmonic for some $p>0$.
    
  Thus $\Delta_z\log|\eta|_h^2=0$ and $|\eta|_h$ is constant.  This implies that $[\eta\wedge\varphi^\dagger]=0$ and $|\eta_z'|_h^2\cdot|\eta|_h^2=|B_{\tau_h}(\eta_z',\eta)|^2$. Moreover, it yields that
    \[0=\left|\bar\partial_z|\eta|_h^2\right|^2=|B_{\tau_h}(\eta_z',\eta)|^2=|\eta_z'|_h^2\cdot|\eta|_h^2.\]
    Hence $\eta_z'$ vanishes, which means, $\eta$ is also antiholomorphic.
\end{proof}

\begin{proposition}\label{prop:factor}
    Let $\eta\in\operatorname{aut}(\mathbb{E},\varphi)$ be an infinitesimal automorphism satisfies that there exists a Liouville-type pair $(g,\mathcal C)$ on $X$ such that
\(|\eta|_h^p\) satisfies
\(\mathcal C\) for some $p>0$. If $\eta\neq 0$, then there exists a proper Levi subgroup $L< G$ such that the harmonic bundle $(\mathbb{E},\varphi,h)$ reduces to $(\mathbb{E}_L,\varphi_L,h_L)$, where $(\mathbb{E}_L,\varphi_L)$ is an $L$-Higgs bundle and $h_L$ is a harmonic metric for it.
    
    As a corollary, if $X$ is compact and the Higgs bundle $(\mathbb{E},\varphi)$ is stable, then $\operatorname{aut}(\mathbb{E},\varphi)=0$.
\end{proposition}

This proposition was essentially proven when $X$ is compact in \cite[Proposition 2.14]{garcia2009hitchin}. Here we simplify their proof and generalize it to non-compact surfaces thanks to \prettyref{coro:Cauchy}. We will use the following lemma proven in \cite[Lemma 2.15]{garcia2009hitchin}.

\begin{lemma}\label{lem:algebra}
    Let $\mathsf{x}\in\mathfrak{g}$ be a semisimple element. There exists some $g\in G$ such that:
    \begin{enumerate}[label=(\roman*)]
        \item if we write $\mathsf{y}=\operatorname{Ad}(g^{-1})(\mathsf{x})=\mathsf{y}_{\mathrm{r}}+ \iu\mathsf{y}_{\mathrm{i}}$ with $\mathsf{y}_{\mathrm{r}}, \mathsf{y}_{\mathrm{i}}\in\mathfrak{h}$, then $[\mathsf{y}_{\mathrm{r}}, \mathsf{y}_{\mathrm{i}}] = 0$;

        \item $\operatorname{ker}\operatorname{ad}(\mathsf{x})= \operatorname{Ad}(g)(\operatorname{ker}\operatorname{ad}(\mathsf{y}_{\mathrm{r}})\cap\operatorname{ker}\operatorname{ad}(\mathsf{y}_{\mathrm{i}}))$.
    \end{enumerate}
\end{lemma}

\begin{proof}[Proof of \prettyref{prop:factor}]
    By \prettyref{coro:Cauchy}, we obtain that $\nabla\eta=0$ and $[\eta\wedge\varphi^\dagger]=0$. Note that $\xi$ can be viewed as a $G$-antiequivariant map $\mathbb{E}\to\mathfrak{g}$. Since $\nabla \xi=0$, its image is contained in a $G$-orbit $G\cdot \mathsf{x}\subseteq\mathfrak{g}$. 
    
    We consider the Jordan--Chevalley decomposition $\mathsf{x}=\mathsf{x}_{\mathrm{s}}+\mathsf{x}_{\mathrm{n}}$ of $\mathsf{x}$, which means
    \begin{enumerate}[label=(\roman*)]
        \item $\operatorname{ad}(\mathsf{x}_{\mathrm{s}})$ is semisimple and $\operatorname{ad}(\mathsf{x}_{\mathrm{n}})$ is nilpotent;

        \item $[\mathsf{x}_{\mathrm{s}},\mathsf{x}_{\mathrm{n}}]=0$;

        \item $\operatorname{ad}(\mathsf{x}_{\mathrm{s}})$ and $\operatorname{ad}(\mathsf{x}_{\mathrm{n}})$ are both polynomials without constant of $\operatorname{ad}(\mathsf{x})$.
    \end{enumerate} 
    Note that (\romannumeral2) implies that $B(\mathsf{x}_{\mathrm{s}},\mathsf{x}_{\mathrm{n}})=0$. This also defines the Jordan--Chevalley decomposition $\eta=\eta_{\mathrm{s}}+\eta_{\mathrm{n}}$. Since $\nabla\eta=0$, by (\romannumeral3) we obtain that $\nabla\eta_{\mathrm{s}}=0$,  $\nabla\eta_{\mathrm{n}}=0$ and $[\eta_{\mathrm{s}}\wedge\varphi]=0$, $[\eta_{\mathrm{n}}\wedge\varphi]=0$, $[\eta_{\mathrm{s}}\wedge\varphi^\dagger]=0$, $[\eta_{\mathrm{n}}\wedge\varphi^\dagger]=0$. 
    
    When $\eta_{\mathrm{s}}$ is nonzero, i.e. $\eta$ is not nilpotent, let $\mathsf{y}=\mathsf{y}_{\mathrm{r}}+ \iu\mathsf{y}_{\mathrm{i}}=\operatorname{Ad}(g^{-1})(\mathsf{x}_{\mathrm{s}})$ be the element given by \prettyref{lem:algebra}. Without loss of generality, we assume $\mathsf{y}_{\mathrm{i}}\neq0$, otherwise we can consider $\iu\eta$. We consider
    \[L:=L_{\iu \mathsf{y}_{\mathrm{i}}}=\{g\in G\mid\operatorname{Ad}(g)(\iu\mathsf{y}_{\mathrm{i}})=\iu\mathsf{y}_{\mathrm{i}}\},\quad\mathbb{E}_L:=\{e\in\mathbb{E}\mid\psi_{\mathrm{s}}(e)=\iu\mathsf{y}_{\mathrm{i}}\},\]
    then $\mathbb{E}_L$ gives an $L$-reduction of $\mathbb{E}$ and $L$ is the Levi subgroup of $G$ given by $\iu y_{\mathrm{i}}$. Since $L$ is preserved by the Cartan involution $\tau$ and $L$ contains the centralizer of $\mathsf{y}$ in $G$ by \prettyref{lem:algebra}, $\nabla\eta_{\mathrm{s}}=0$, $[\eta_{\mathrm{s}}\wedge\varphi]=0$ and $[\eta_{\mathrm{s}}\wedge\varphi^\dagger]=0$ implies that both of the Chern connection $\nabla$ of $h$ and the Higgs field $\varphi$ reduce to the Levi subgroup $L$. Hence $h$ also reduces.
    
    Now we assume $\eta_{\mathrm{s}}=0$, i.e. $\mathsf{x}=\mathsf{x}_{\mathrm{n}}$ is nilpotent but nonzero. Then $\mathsf{h}=[\mathsf{x}_{\mathrm{n}},\mathsf{x}_{\mathrm{n}}^\dagger]\neq0$ is semisimple. This induces a section $\zeta$ of $\mathbb{E}[G]$ with $\nabla \zeta=0$. By Jacobi identity, we also have $[\zeta\wedge\varphi]=0$ and $[\zeta\wedge\varphi^\dagger]=0$. Note that $\mathsf{h}\in\iu\mathfrak{h}$, we can then use the Levi subgroup determined by $\mathsf{h}$ and repeat the process above.

    When $X$ is compact and $(\mathbb{E},\varphi)$ is stable, $\eta\equiv0$ follows from \prettyref{lem:nonstable}.
\end{proof}

\subsection{Cyclic Higgs bundle}

\begin{definition}
    A $\Theta$-cyclic $G$-Higgs bundle is a $(G^\Theta,\mathfrak{g}_{\overline{1}})$-Higgs pair.
\end{definition}

Given a $\Theta$-cyclic $G$-Higgs bundle $(\mathbb{E},\varphi)$, there is a natural extension of structure group to obtain a $G$-Higgs bundle $(\mathbb{E}_G,\varphi_G)$. Also, if $(\mathbb{E},\varphi)$ is a $G$-Higgs bundle we say that it reduces to a
$\Theta$-cyclic $G$-Higgs bundle if $\mathbb{E}$ reduces to a $G^\Theta$-bundle $\mathbb{E}_{G^\Theta}$ and $\varphi$ takes values in $\mathbb{E}_{G^\Theta}[\mathfrak{g}_{\overline{1}}]\otimes\KK_X$. 

\begin{definition}
Let $(\mathbb{E},\varphi)$ be a $\Theta$-cyclic $G$-Higgs bundle. \begin{itemize}
        \item We call a metric $h$ of $(\mathbb{E},\varphi)$ a \textbf{$\Theta$-cyclic harmonic metric} if it is a metric of $\mathbb{E}$, i.e., a reduction of the structure group of $\mathbb E$ from $G^{\Theta}$ to $H^{\Theta}$, and induces a harmonic metric of $(\mathbb{E}_G,\varphi_G)$ by trivial extension. 
        \item We call $(\mathbb{E},\varphi,h)$ a \textbf{$\Theta$-cyclic harmonic $G$-Higgs bundle} if $h$ is a $\Theta$-cyclic harmonic metric of $(\mathbb{E},\varphi)$. 

        \item Let $U$ be a reductive Lie subgroup in $H^\Theta$. We call a $\Theta$-cyclic harmonic $G$-Higgs bundle $(\mathbb{E},\varphi,h)$ \textbf{reduces to $U$} if $h$ arises from a reduction of structure group of $\mathbb{E}$ from $G^\Theta$ to $U$. We denote by $\mathbb{E}_h$ the principal $U$-bundle.
    \end{itemize}
\end{definition}

\begin{definition}
    Let $(\mathbb{E}_1,\varphi_1,h_1)$ and $(\mathbb{E}_2,\varphi_2,h_2)$ be two $\Theta$-cyclic harmonic $G$-Higgs bundles reducing to $U$. We say that they are isomorphic if there exists an isomorphism $f\colon(\mathbb{E}_1)_{h_1}\to(\mathbb{E}_2)_{h_2}$ such that $\varphi_2=(\operatorname{Ad}(f)\otimes\operatorname{id}_{\KK_X})(\varphi_1)$.
\end{definition}

The following theorem was proven in \cite[Theorem 5.5 \& Proposition 5.7]{garcia2019involutions}.

\begin{theorem}\label{thm:cycHK}
     Let $(\mathbb{E},\varphi)$ be a $\Theta$-cyclic $G$-Higgs bundle. Then the trivial extension $(\mathbb{E}_G,\varphi_G)$ is polystable if and only if there exists a $\Theta$-cyclic harmonic metric $h$ of $(\mathbb{E},\varphi)$.
 \end{theorem}

\section{Cyclic surfaces}\label{sec:cyclicsurfaceshiggsbundles}

Let $G$ be a connected semisimple complex Lie group with finite-dimensional Lie algebra $\mathfrak{g}$, and let $\Theta\in\operatorname{Aut}(G)$ be an automorphism of order $m\geqslant3$. As in \prettyref{sec:cyclicspace}, we have the maximal compact subgroup $H^\Theta$ of $G^\Theta$ fixed by $\Theta$, a reductive subgroup $U<H^\Theta$ and the complex cyclic distribution $\mathcal{D}_U$ with the almost-complex structure $\mathcal{J}$. Let $\Sigma$ be a smooth surface. 

\begin{definition}\label{defn:holcurve}
    A smooth map $f\colon\Sigma\to G/U$ is called a \textbf{$\mathcal{J}$-holomorphic curve} if 
    \begin{enumerate}[label=(\roman*)]
        \item $f$ is tangent to the complex cyclic distribution $\mathcal{D}_U$;

        \item $\mathrm{d}f(\mathrm{T}\Sigma)$ is $\mathcal{J}$-invariant.
    \end{enumerate}
\end{definition}

\begin{definition}\label{defn:cycsurface}
    A $\mathcal{J}$-holomorphic curve $f\colon \Sigma\to G/U$ is called a \textbf{cyclic surface} if \[[f^*(\omega_{\overline{1}})\wedge f^*(\omega_{-\overline{1}})]\in\mathcal{A}^2(\Sigma,f^*[\mathfrak{u}]).\]
    Let $\mathfrak{c}$ be a $U$-cyclic subspace. A cyclic surface $f\colon \Sigma\to G/U$ is called a \textbf{${\mathfrak{c}}$-cyclic surface} if \[f^*(\omega_{\mathfrak{c}})\in\mathcal{A}^1(\Sigma,f^*[\mathfrak{c}])\] is nowhere vanishing.

    In particular, for a $\Theta$-cyclic root $\boldsymbol{\alpha}$, we simply call a $\mathfrak{g}_{\boldsymbol{\alpha}}$-cyclic surface in $\mathbb{X}_{\Theta}$ an $\boldsymbol{\alpha}$-cyclic surface.
\end{definition} 

\begin{proposition}\label{prop:basic}
Let $f\colon\Sigma\to G/U$ be a $\mathcal{J}$-holomorphic curve. 
    \begin{itemize}
    \item When $m=2k$ is even, we have the real form $\Lambda=\Theta^k\circ\tau$. Then the image of $f$ lies in a $G^\Lambda$-orbit of $G/U$.
    
    \item 
    If $f$ lies in the $\Theta$-cyclic space $\mathbb{X}_{\Theta}$, then $f$ is automatically a cyclic surface.

    \item If $f\colon\Sigma\to G/U$ be a $\mathfrak{c}$-cyclic surface on $\Sigma$ for some $U$-cyclic subspace $\mathfrak{c}$, then it is an immersion. 
\end{itemize}
\end{proposition}

\begin{proof}
    
\begin{itemize}
    \item Let $\lambda$ be the derivative of $\Lambda=\Theta^k\circ\tau$. Both $\theta^k$ and $\tau$ acts by $-1$ on the complex cyclic distribution $\mathcal{D}_U$, hence $\dd f$ is fixed by $\lambda$. Therefore, the image of $\dd f$ lies in $[\mathfrak{g}^\lambda]$ and $f$ lies in a $G^\Lambda$-orbit of $G/U$.
    
    \item 
    Since
    \[[f^*(\omega_{\overline{1}})\wedge f^*(\omega_{-\overline{1}})]=[f^*(\omega_{\overline{1}})\wedge f^*(-\tau(\omega_{\overline{1}}))]\]
    is $f^*\tau$-invariant, it is a $2$-form taking value in $f^*[\mathfrak{h}_{\overline{0}}]$. Hence $f$ is automatically a cyclic surface.

    \item Since $f^*(\omega_{\mathfrak{c}})$ is nowhere vanishing, we know $f^*(\omega_{\tau(\mathfrak{c})})$ is also nowhere vanishing. Therefore, $\dd f$ is of rank $2$ by $\dim_{\mathbb{C}}(\mathfrak{c})=1$, and $f\colon\Sigma\to G/U$ is an immersion. 
\end{itemize}

\end{proof}

Note that $[\mathcal{D}_U,\mathcal{D}_U]\subseteq[\mathfrak{g}_{-\overline{2}}\oplus\mathfrak{g}_{\overline{0}}\oplus\mathfrak{g}_{\overline{2}}]$ is perpendicular to $\mathcal{D}_U\subseteq[\mathfrak{g}_{-\overline{1}}\oplus\mathfrak{g}_{\overline{1}}]$. Hence the canonical connection induces a torsion-free connection on $\mathrm{T}\Sigma$ through a $\mathcal{J}$-holomorphic curve $f\colon\Sigma\to G/U$.

Let $p_U\colon G/U\to G/H$ be the natural projection from the homogeneous space to the symmetric space of $G$.

\begin{definition}
    Let $\Sigma$ be an oriented surface and $\widetilde{\Sigma}$ be its universal cover. An \textbf{equivariant $\mathfrak{c}$-cyclic surface} on $\Sigma$ is a pair $(\rho,f)$, where $\rho\colon\pi_1(\Sigma)\to G$ is a representation of the fundamental group of $\Sigma$, $f\colon\widetilde{\Sigma}\to G/U$ is a $\rho$-equivariant $\mathfrak{c}$-cyclic surface.
\end{definition}

\subsection{From cyclic surfaces to cyclic Higgs bundles}

\begin{proposition}\label{prop:surfacetoHiggs}
    Let $f\colon\Sigma\to G/U$ be an immersed cyclic surface. Denote by $\nabla:=f^*A$ the pullback connection on $f^*\mathrm{T}(G/U)$, $\varphi:=f^*(\omega_{\overline{1}})$, $\tau_f:=f^*\tau$. Pulling back the $U$-bundle $G\to G/U$ along $f$ we obtain a $U$-bundle $E_{U}:=f^*G\to \Sigma$. Extending its structure group trivially we get a $G^\Theta$-bundle $E:=f^*G[G^\Theta]\to\Sigma$ and a reduction of structure group $h$ of $E$ from $G^\Theta$ to $U$. Then
    \begin{enumerate}[label=(\roman*)]
        \item there exists a unique complex structure $\mathsf{j}$ on $\Sigma$ (which implies that $\Sigma$ is orientable) such that $\varphi$ is of type $(1,0)$;

        \item the Maurer--Cartan equations \prettyref{eq:MC} imply that
        \begin{equation}
            \begin{cases}
            F_\nabla - [\varphi\wedge\tau_f(\varphi)] = 0,\\
            \bar\partial^\nabla\varphi= 0,
            \end{cases}
        \end{equation}
    where $F_\nabla$ is the curvature form of the pullback connection $\nabla$ and $\bar\partial^\nabla$ is the $(0,1)$-part of $\mathrm{d}^\nabla$. Moreover, this implies that $(\mathbb{E},\varphi,h)$ is a $\Theta$-cyclic $G$-Higgs bundle reducing to $U$ over the Riemann surface $X:=(\Sigma,\mathsf{j})$, where $\mathbb{E}:=(E_U,\bar\partial^\nabla)$ is the holomorphic $G^\Theta$-bundle obtained by equipping $E$ with $\bar\partial^\nabla$;

    \item $p_U\circ f\colon\Sigma\to G/H$ is a minimal surface, where $G/H$ is equipped with the metric induced by $B_\tau$.
    \end{enumerate}
\end{proposition}

\begin{proof}
   We first prove the first item. Let $\mathsf{j}'$ be an arbitrary complex structure on $\Sigma$. Then by definition $\varphi$ is of type $(1,0)$ with respect to $\mathsf{j}'$ is equivalent to say the $-\iu$-eigenspace of $\mathsf{j}'$ on $\mathrm{T}^\mathbb{C}\Sigma$ is contained in $\mathfrak{g}_{-\overline{1}}$, which is the $-\iu$-eigenspace of $\mathcal{J}$. Hence this is equivalent to say that 
   \[\mathrm{d}f\circ\mathsf{j}' =\mathcal{J}\circ\mathrm{d}f.\]
   Since $f$ is an immersion, there exists a unique complex structure $\mathsf{j}$ on $\Sigma$ such that 
    \[\mathrm{d}f\circ\mathsf{j} =\mathcal{J}\circ\mathrm{d}f.\]

    For the second item, since $f$ is tangent to $\mathcal{D}_U$, $f^*(\omega_{\overline{j}})=0$ for every $\overline{j}\neq\overline{1},-\overline{1}$ and $f^*(\omega_{-\overline{1}})=-\tau_f(\varphi)$. Therefore, by pulling back 
    \[F_{A}+\dfrac{1}{2}\sum\limits_{\overline{j}\in\mathbb{Z}/m\mathbb{Z}}\operatorname{pr}_{\mathfrak{u}}\left([\omega_{\overline{j}}\wedge\omega_{-\overline{j}}]\right)=0\]
    through $f$ we obtain that
    \[F_{\nabla}-\operatorname{pr}_{\mathfrak{u}}\left([\varphi\wedge\tau_f(\varphi)]\right)=0.\]
    Since \[[\varphi\wedge\tau_f(\varphi)]\in\mathcal{A}^2(\Sigma,f^*[\mathfrak{u}])\]
    by definition, we obtain that
    \[F_\nabla - [\varphi\wedge\tau_f(\varphi)] = 0.\]

    When $m>3$, $\{\overline{k},\overline{1-k}\}\nsubseteq\{\overline{1},-\overline{1}\}$ for any $\overline{k}\in\mathbb{Z}/m\mathbb{Z}$, hence by pulling back 
    \[\mathrm{d}^A\omega_{\overline{1}}+\dfrac{1}{2}\sum\limits_{\overline{k}\in\mathbb{Z}/m\mathbb{Z}}[\omega_{\overline{k}}\wedge\omega_{\overline{1-k}}]=0\]
    through $f$ we obtain that
    $\dd^\nabla\varphi=0$, which is equivalent to $\bar\partial^\nabla\varphi= 0$ since $\varphi$ is of type $(1,0)$. Similarly, when $m=3$, $\{\overline{k},\overline{1-k}\}\nsubseteq\{\overline{1},-\overline{1}\}$ holds for $\overline{k}=\overline{0},\overline{1}$, so by pulling back 
    \[\mathrm{d}^A\omega_{\overline{1}}+\dfrac{1}{2}\sum\limits_{\overline{k}\in\mathbb{Z}/3\mathbb{Z}}[\omega_{\overline{k}}\wedge\omega_{\overline{1-k}}]=0\]
    through $f$ we obtain that
    \[\dd^\nabla\varphi+\dfrac{1}{2}[\tau_f(\varphi)\wedge\tau_f(\varphi)]=0.\]
    Note that $\tau_f(\varphi)$ is of type $(0,1)$, thus $\dd^\nabla\varphi=0$ and $\bar\partial^\nabla\varphi=0$.

    For the third item, it follows from the proof of \cite[Corollary 6.1.2]{labourie2017cyclic}. For instance, the tangent bundle $\mathrm{T}(G/H)$ of the symmetric space can be identified with the trivial vector bundle over $G/H$ with fiber $\mathfrak{m}$ by the form $\operatorname{pr}_{\mathfrak{m}}(\omega_{MC})$. Hence the differential of $p_U\circ f$ is identified with $f^*\circ p_U^*(\operatorname{pr}_{\mathfrak{m}}(\omega_{MC}))=\operatorname{pr}_{\mathfrak{m}}(f^*\omega)=\varphi-\tau_f(\varphi)$. Therefore, $\nabla(\dd(p_U\circ f))=\nabla(\varphi-\tau_f(\varphi))=0$ yields that $p\circ f$ is harmonic. Note that the Hopf differential of $p_U\circ f$ is
    \[\left(B_{\tau}((\dd(p_U\circ f)),(\dd(p_U\circ f)))\right)^{2,0}=B((\dd(p_U\circ f))^{1,0},(\dd(p_U\circ f))^{1,0})=(f^*B)(\varphi,\varphi)=0\]
    since $\mathfrak{g}_{\overline{1}}$ is isotropic with respect to the Killing form $B$. Thus $p_U\circ f$ is weakly conformal and harmonic, which implies $p_U\circ f$ is a branched minimal surface. Since $p_U\circ f$ is also an immersion, it is a minimal surface.
\end{proof}

If $f\colon\Sigma\to G/U$ is a $\mathfrak{c}$-cyclic surface, then $f^*[\mathfrak{c}]$ is a subbundle of $f^*[\mathfrak{g}_{\overline{1}}]=\mathbb{E}[\mathfrak{g}_{\overline{1}}]$. We have the following lemma for $f^*[\mathfrak{c}]$.

\begin{lemma}\label{lem:decomposition}
    Let $f\colon\Sigma\to G/U$ be a $\mathfrak{c}$-cyclic surface on $\Sigma$ in $G/U$. With the induced holomorphic structure in \prettyref{prop:surfacetoHiggs}, $f^*[\mathfrak{g}_{\overline{1}}]$ has a holomorphic decomposition $\mathcal{K}_X^\vee\oplus f^*[\mathfrak{g}_{\overline{1}}/\mathfrak{c}]$. And $\operatorname{pr}_{\mathfrak{c}}\colon\KK_X^\vee\to f^*[\mathfrak{c}]$ is an isomorphism. As a corollary, when $(\rho,f)$ is an equivariant $\mathfrak{c}$-cyclic surface on $\Sigma$ in $G/U$, the decomposition and the projection above descends to the surface $\Sigma$.
\end{lemma}

\begin{proof}
    Without loss of generality, we can assume $\rho$ is trivial. Note that $\mathrm{T}^\mathbb{C}\Sigma$ splits as holomorphic subbundle $\KK_X^\vee\oplus\mathcal{K}_X$ (the underlying bundles are the $\pm\iu$-eigensubbundles of $\mathsf{j}$ respectively) with respect to the holomorphic structure $\mathsf{j}$. Therefore, $\mathcal{K}_X^\vee$ is a smooth subbundle of $f^*[\mathfrak{g}_{\overline{1}}]$. Notice that the projection $\operatorname{pr}_{\mathfrak{c}}$ is actually the same as the form $f^*(\omega_{\mathfrak{c}})$, which is nowhere vanishing and holomorphic with respect to $\bar\partial^\nabla$, hence $\mathcal{K}_X^\vee$ is a holomorphic subbundle of $f^*[\mathfrak{g}_{\overline{1}}]$ and $\operatorname{pr}_{\mathfrak{c}}$ is the isomorphism. This yields the holomorphic decomposition $f^*[\mathfrak{g}_{\overline{1}}]\cong\KK_X^\vee\oplus f^*[\mathfrak{g}_{\overline{1}}/\mathfrak{c}]$. 
\end{proof}

As a corollary of \prettyref{prop:surfacetoHiggs} and \prettyref{lem:decomposition}, we obtain that:
\begin{corollary}\label{coro:surfacetoHiggs}
    Let $(\rho,f)$ be an equivariant immersed cyclic surface on $\Sigma$ in $G/U$. Then there exists a unique complex structure $\mathsf{j}$ on $\Sigma$ such that $(\mathbb{E},\varphi, h)$ defined as \prettyref{prop:surfacetoHiggs} descending to a $\Theta$-cyclic harmonic $G$-Higgs bundle reducing to $U$ over the Riemann surface $X:=(\Sigma,\mathsf{j})$ with $\varphi$ nowhere vanishing. Furthermore, if $(\rho,f)$ is an equivariant $\mathfrak{c}$-cyclic surface, then $\operatorname{pr}_{\mathfrak{c}}(\varphi)$ nowhere vanishing.
    
    Moreover, when $\Sigma$ is closed, the corresponding $G$-Higgs bundle is polystable.
\end{corollary}

We also consider the degeneration case. Let $(\rho,f)$ be an equivariant $\mathfrak{c}$-cyclic surface on $\Sigma$ in the $G/U$ with the associated $\Theta$-cyclic harmonic $G$-Higgs bundle $(\mathbb{E},\varphi,h)$ reducing to $U$ over the Riemann surface $X=(\Sigma,\mathsf{j})$. Then the image of $f$ lies in an $L$-orbit $\cong L/(L\cap U)$ for some Levi subgroup $L<G$ if and only if $(\mathbb{E},\varphi,h)$ reduces to an harmonic $L$-Higgs bundle $(\mathbb{E}_L,\varphi_L,h_L)$.

\subsection{From cyclic Higgs bundles to cyclic surfaces}

Let $X=(\Sigma,\mathsf{j})$ be a Riemann surface. The following proposition is a converse of \prettyref{coro:surfacetoHiggs}.

\begin{proposition}\label{prop:Higgstosurface}
    Given a $\Theta$-cyclic harmonic $G$-Higgs bundle $(\mathbb{E},\varphi,h)$ reducing to $U$ over $X$ with $\varphi$ nowhere vanishing, let $\rho:=\operatorname{Hol}(\mathrm{D}^h)$ be the holonomy representation of the flat connection $\mathrm{D}^h$ induced by $h$. Then the corresponding $\rho$-equivariant harmonic map $h_\rho\colon \widetilde{\Sigma}\to G/H$ lifts to a $\rho$-equivariant immersed cyclic surface $f_\rho\colon\widetilde{\Sigma}\to G/U$. Moreover, if $\operatorname{pr}_{\mathfrak{c}}(\varphi)$ is nowhere vanishing for some $U$-cyclic subspace $\mathfrak{c}\subseteq\mathfrak{g}_{\overline{1}}$, then $f_\rho$ is $\mathfrak{c}$-cyclic.
\end{proposition}

\begin{proof}
    This is just a reformulation of the definition of the $\Theta$-cyclic harmonic $G$-Higgs bundle reducing to $U$. The metric $h$ reduces the principal $G^\Theta$-bundle $\mathbb{E}$ to the principal $U$-bundle $\mathbb{E}_h$. Using $\mathrm{D}^h$ we obtain a trivialization of $\pi^*\mathbb{E}_h[G]\to\widetilde{\Sigma}$. Note that $\pi^*\mathbb{E}_h\subseteq\pi^*\mathbb{E}_h[G]$ is a reduction of the structure group from $G$ to $U$. Hence it gives an equivariant map $f\colon\widetilde{\Sigma}\to G/U$ since $\pi^*\mathbb{E}_h[G]$ is trivial. Moreover, its differential is given by $\pi^*(\varphi-\tau_h(\varphi))$, which implies the claim on immersion and $\mathfrak{c}$-cyclic.
\end{proof}

\begin{definition}
  Two equivariant $(\mathfrak{c}\text{-})$cyclic surfaces $(\rho_1,f_1)$ and $(\rho_2,f_2)$ on $\Sigma$ in $G/U$ are called \emph{isomorphic} if there exist $g\in G$ and $\psi\in\operatorname{Diff}^0(\Sigma)$ such that
  \[
    (\rho_1,f_1)
    =
    \bigl(\operatorname{Ad}_g\circ\rho_2,\,(g\cdot f_2)\circ\widetilde{\psi}\bigr),
  \]
  where $\widetilde{\psi}$ is the lift of $\psi$ to $\widetilde{\Sigma}$. We denote by $\sim$ the equivalence relation generated by this notion of isomorphism. The moduli space $\operatorname{Cyc}^{\mathfrak{c}}(\Sigma)$ of equivariant $\mathfrak{c}$-cyclic surfaces from $\Sigma$ is defined by
    \[\operatorname{Cyc}^{\mathfrak{c}}(\Sigma):= \{(\rho, f)\mbox{ equivariant }\mathfrak{c}\mbox{-cyclic surface}\}/\sim.\]
\end{definition}

\begin{definition}
  Two equivariant immersed $(\mathfrak{c}\text{-})$cyclic surfaces $(\rho_1,f_1)$ and $(\rho_2,f_2)$ on $\Sigma$ inducing the same complex structure $\mathsf{j}$ are called \emph{$\mathsf{j}$-isomorphic} if there exists $g\in G$ such that
  \[
    (\rho_1,f_1)
    =
    \bigl(\operatorname{Ad}_g\circ\rho_2,\,g\cdot f_2\bigr).
  \]
  We denote by $\sim_{\mathsf{j}}$ the equivalence relation generated by this notion of isomorphism.
\end{definition}

Since $\operatorname{Diff}^0(\Sigma)$ acts freely on the space of complex structures on $\Sigma$ when $\Sigma$ is closed of genus at least $2$, two equivariant $\mathfrak{c}$-cyclic surfaces from $\Sigma$ that induce the same complex structure $\mathsf{j}$ are $\mathsf{j}$-isomorphic if and only if they are isomorphic.

\begin{theorem}\label{thm:correspondence}
    There is a natural bijection between
    \begin{itemize}
        \item the $\mathsf{j}$-isomorphic classes of equivariant immersed (resp. $\mathfrak{c}$-)cyclic surfaces on $\Sigma$ in $G/U$ with the induced complex structure $\mathsf{j}$, and

        \item the isomorphic classes of $\Theta$-cyclic harmonic $G$-Higgs bundles $(\mathbb{E},\varphi,h)$ reducing to $U$ over $(\Sigma,\mathsf{j})$ with $\varphi$ (resp. $\operatorname{pr}_{\mathfrak{c}}(\varphi)$) nowhere vanishing.
    \end{itemize}
\end{theorem}

\begin{proof}
    If two equivariant immersed (resp. $\mathfrak{c}$-)cyclic surfaces $(\rho_i,f_i)$ on $\Sigma$ to $G/U$ with the induced complex structure $\mathsf{j}$ are $\mathsf{j}$-isomorphic, then there exists $g\in G$ such that $(\rho_1, f_1) = (\operatorname{Ad}_g\circ\rho_2,g\cdot f )$. Denote by $L_g$ the left multiplication on $G/U$. This implies that $f_1^*=f_2^*L_g^*$. Since the $U$-bundle $G\to G/U$ and the Maurer--Cartan form are both $G$-invariant, we obtain that the associated Higgs bundles in \prettyref{prop:surfacetoHiggs} are isomorphic.

    Conversely, if two $\Theta$-cyclic harmonic $G$-Higgs bundles $(\mathbb{E}_i,\varphi_i,h_i)$ reduding to $U$ over $(\Sigma,\mathsf{j})$ with $\varphi$ (resp. $\operatorname{pr}_{\mathfrak{c}}(\varphi)$) nowhere vanishing are isomorphic, then follow from \prettyref{prop:Higgstosurface}, the corresponding cyclic surfaces $f_i$ differ only by the choice of the identification of $\pi^*(\mathbb{E}_i)_{h_i}[G]\to\widetilde{\Sigma}$ with $\widetilde{\Sigma}\times G\to \widetilde{\Sigma}$. In other words, $(\rho_1,f_1)\sim_{\mathsf{j}}(\rho_2,f_2)$.
\end{proof}

\section{Infinitesimal rigidity}\label{sec:infinitesimalrigidity}

In this section, we will study the Jacobi field and prove the infinitesimal rigidity of cyclic surfaces. Given an oriented surface $\Sigma$ with its universal cover $\pi\colon\widetilde{\Sigma}\to \Sigma$ and an equivariant $\mathfrak{c}$-cyclic surface $(\rho,f)$ on $\Sigma$ in $G/U$.

\begin{definition}
    A \textbf{smooth variation} of the equivariant ($\mathfrak{c}$-)cyclic surface $(\rho,f)$ on $\Sigma$ is a pair $(\varrho,F)$ consists of
    \begin{itemize}
        \item a smooth path $\rho_t(\bullet)=\varrho(\bullet,t)$ in $\operatorname{Hom}(\pi_1(\Sigma),G)$, where $t\in(-\varepsilon,\varepsilon)$;

        \item a smooth map $F: \widetilde{\Sigma}\times (-\varepsilon, \varepsilon)\to G/U$ such that the map $f_t:=F(\bullet,t)$ is a ($\mathfrak{c}$-)cyclic surface which is $\rho_t$-equivariant for every $t$;
    \end{itemize}
    satisfying that $(\rho_0,f_0)=(\rho,f)$.

    We also denote by $(\rho_t, f_t)_{t\in (-\varepsilon, \varepsilon)}$ such a variation, $\dot{f_0} =\dd F(\partial_t)|_{t=0}$ the vector field along $f_0$ and $\dot{\rho_0}=\dd\varrho(\partial_t)$. 
\end{definition}

    Recall the notations defined in \prettyref{sec:cyclicspace}. The trivial Lie algebra bundle $[\mathfrak{g}]=[\mathfrak{u}]\oplus \mathrm{T}(G/U)$ over $G/U$ is equipped with the connection $A$ induced by the Maurer--Cartan form $\omega_{MC}=A+\omega$ and $\omega\in\mathcal{A}^1(G/U,[\mathfrak{u}^\perp])$. Let
\[\omega_Z:=\sum_{\substack{\overline{j}\in\mathbb{Z}/m\mathbb{Z}\\\overline{j}\neq\overline{0},\overline{1},-\overline{1}}}\omega_{\overline{j}}=\operatorname{pr}_Z(\omega).\]
Hence
\[\omega=\omega_{\overline{0}} + \omega_{\overline{1}} + \omega_{-\overline{1}} +\omega_Z\]
when $U=H^\Theta$.

Let $\nabla:=\pi_*f^*A$, $\varphi:=\pi_*f^*(\omega_{\overline{1}})$, $\varphi^\dagger=\pi_*f^*(\omega_{-\overline{1}})$, $\tau_f:=\pi_*f^*\tau$, $B_{\tau_f}:=\pi_*f^*B_\tau$. Note that $\varphi^\dagger=-\tau_f(\varphi)$ by definition. By \prettyref{prop:surfacetoHiggs}, $(\rho,f)$ induces a complex structure $\mathsf{j}$, a $\Theta$-cyclic harmonic $G$-Higgs bundle $(\mathbb{E},\varphi,h)$ reducing to $U$ over $X=(\Sigma,\mathsf{j})$.

\subsection{Variation of Pfaffian systems}

Consider a Lie algebra bundle $\widehat{E}\to M$ on a smooth manifold $M$, and let $\widehat{\nabla}$ be a compatible connection, that is $\widehat{\nabla}$ is a derivation for the Lie bracket). Consider $\boldsymbol{\Omega}=\{\Omega_1,\dots,\Omega_k\}$ with $\Omega_k\in\mathcal{A}^\bullet(M,\widehat{E})$. A \textbf{solution to the Pfaffian system} $\boldsymbol{\Omega}$ is a map $f\colon N\to M$ such that $f^*\Omega_i = 0$ for all $\Omega_i\in\boldsymbol{\Omega}$. Let $\mathcal{I}_{\boldsymbol{\Omega}}$ be the differential ideal generated by $\boldsymbol{\Omega}$. 

Note that $f$ is a solution to the Pfaffian system $\boldsymbol{\Omega}$ if and only if $f^*\widehat{\Omega}=0$ for any $\widehat{\Omega}\in\mathcal{I}_{\boldsymbol{\Omega}}$ (also see \cite[Lemma 6.10]{collier2024holomorphic}).

The following is proved in \cite[Proposition 7.1.4]{labourie2017cyclic}:

\begin{proposition}\label{prop:JacobiPff}
    Let $F\colon N\times(-\varepsilon,\varepsilon)\to M$ be a smooth path of solution to the Pfaffian system $\boldsymbol{\Omega}$, i.e. $F$ is smooth and $f_t(\bullet):=F(\bullet,t)$ are solutions to the Pfaffian system $\boldsymbol{\Omega}$. Let 
    $\dot{f_0}:=\dd F(\partial_t)|_{t=0}$ and $f:=f_0$. Then for any $\widehat{\Omega}\in\mathcal{I}_{\boldsymbol{\Omega}}$ we have
    \[\dd^{\widehat{\nabla}}\left(\iota_{\dot{f_0}}\widehat{\Omega}\right)=-\iota_{\dot{f_0}}\left(\dd^{\widehat{\nabla}}\widehat{\Omega}\right).\]
\end{proposition}

We define the $\Theta$-cyclic Pffafian system 
\[\boldsymbol{\Omega}_\Theta:=\left\{\omega_{\overline{0}},\omega_{\overline{1}}+\tau(\omega_{-\overline{1}}),\omega_{-\overline{1}}+\tau(\omega_{\overline{1}}),[\omega_{\overline{1}}\wedge\omega_{\overline{1}}],[\omega_{-\overline{1}}\wedge\omega_{-\overline{1}}],\omega_Z\right\}\]
on the trivial Lie algebra bundle $[\mathfrak{g}]$ over $G/U$ with flat connection induced by the Maurer--Cartan form. By definition and \prettyref{prop:surfacetoHiggs}, any $\mathfrak{c}$-cyclic surface is a solution to the $\Theta$-cyclic Pffafian system $\boldsymbol{\Omega}_{\Theta}$.

\subsection{Variation of \texorpdfstring{$\mathcal{J}$}{J}-holomorphic curves}

Let $(\varrho, F)$ be a smooth variation of the equivariant $\mathfrak{c}$-cyclic surface $(\rho,f)$ on $\Sigma$. By \prettyref{prop:surfacetoHiggs}, we have a smooth path of complex structure $(\mathsf{j}_t)_{t\in(-\varepsilon,\varepsilon)}$ on $\Sigma$, which is defined by $\mathsf{j}_t:=(\dd f_t)^{-1}\circ\mathcal{J}\circ\dd f_t$ with $\mathsf{j}=\mathsf{j}_0$. We have the equation
\begin{equation}\label{eq:complex}
    \mathcal{J}\circ\dd f_t(v)=\dd f_t\circ\mathsf{j}_t(v)
\end{equation}
holds for every $v\in \mathrm{T}\Sigma$. Let $\dot{f_0}:=\dd F(\partial_t)|_{t=0}$.

\begin{lemma}\label{lem:firstdf}
    \[\nabla^{\mathrm{can}}_{\dot{f_0}}\left(\dd f_t(v)\right)|_{t=0}=\nabla^{\mathrm{can}}_{\dd f(v)}\dot{f_0}+T^{\mathrm{can}}(\dot{f_0},\dd f(v)).\]
\end{lemma}

\begin{proof}
    \[\begin{aligned}
        &\nabla^{\mathrm{can}}_{\dot{f_0}}\left(\dd f_t(v)\right)|_{t=0}\\
        =&\nabla^{\mathrm{can}}_{\dot{f_0}}\nabla^{\mathrm{can}}_{\dd f(v)} f_t|_{t=0}\\
        =&\nabla^{\mathrm{can}}_{\dd f(v)}\nabla^{\mathrm{can}}_{\dot{f_0}} f_t|_{t=0}+[\dd f(v),\dot{f_0}](f_t)|_{t=0}+T^{\mathrm{can}}(\dot{f_0},\dd f(v))\quad\mbox{(by the definition of torsion)}\\
        =&\nabla^{\mathrm{can}}_{\dd f(v)}(\dd F(\partial_t))|_{t=0}+T^{\mathrm{can}}(\dot{f_0},\dd f(v))\quad\mbox{(since }\left[\partial_t,v\right]=0\mbox{)}\\
        =&\nabla^{\mathrm{can}}_{\dd f(v)}\dot{f_0}+T^{\mathrm{can}}(\dot{f_0},\dd f(v)).
    \end{aligned}\]
\end{proof}

\begin{lemma}\label{lem:Jacobihol}
    \[\mathcal{J}\left(\nabla^{\mathrm{can}}_{\dd f(v)}\dot{f_0}+T^{\mathrm{can}}_{\overline{1},-\overline{1}}(\dot{f_0},\dd f(v))\right)=\nabla^{\mathrm{can}}_{\dd f(\mathsf{j}(v))}\dot{f_0}+T^{\mathrm{can}}_{\overline{1},-\overline{1}}(\dot{f_0},\dd f(\mathsf{j}(v)))+\dd f\left(\dfrac{\dd \mathsf{j}_t}{\dd t}(v)\right).\]
\end{lemma}

\begin{proof}
    
We also use $\mathcal{J}$ to denote the zero extension of $\mathcal{J}$ on the whole $\mathrm{T}(G/U)$ with respect to the orthogonal decomposition. Apply $\nabla^{\mathrm{can}}_{\dot{f_0}}$ to \prettyref{eq:complex}. For the left hand side, we have
\[\begin{aligned}
    &\nabla^{\mathrm{can}}_{\dot{f_0}}\left(\mathcal{J}\circ\dd f_t(v)\right)\\
    =&\mathcal{J}\circ\nabla^{\mathrm{can}}_{\dot{f_0}}\left(\dd f_t(v)\right)\quad\mbox{(since }\mathcal{J}\mbox{ is }G\mbox{-invariant, }\nabla^{\mathrm{can}}\mathcal{J}=0\mbox{)}\\
    =&\mathcal{J}\left(\nabla^{\mathrm{can}}_{\dd f(v)}\dot{f_0}+T^{\mathrm{can}}(\dot{f_0},\dd f(v))\right)\quad\mbox{(by \prettyref{lem:firstdf})}\\
    =&\mathcal{J}\left(\nabla^{\mathrm{can}}_{\dd f(v)}\dot{f_0}+T^{\mathrm{can}}_{\overline{1},-\overline{1}}(\dot{f_0},\dd f(v))\right)
\end{aligned}\]

On the other side, we have
\[\begin{aligned}
    &\nabla^{\mathrm{can}}_{\dot{f_0}}\left(\dd f_t\circ\mathsf{j}_t(v)\right)\\
    =&\nabla^{\mathrm{can}}_{\dot{f_0}}\left(\dd f_t\left(\mathsf{j}(v)\right)\right)+\dd f\left(\dfrac{\dd \mathsf{j}_t}{\dd t}(v)\right)\\
    =&\nabla^{\mathrm{can}}_{\dd f(\mathsf{j}(v))}\dot{f_0}+T^{\mathrm{can}}(\dot{f_0},\dd f(\mathsf{j}(v)))+\dd f\left(\dfrac{\dd \mathsf{j}_t}{\dd t}(v)\right).\quad\mbox{(by \prettyref{lem:firstdf})}
\end{aligned}\]
By projecting onto $[\mathfrak{g}_{\overline{1}}\oplus\mathfrak{g}_{-\overline{1}}]$, we prove the lemma.
\end{proof}

Note that $\mathrm{T}\widetilde{\Sigma}\subseteq f^*\mathrm{T}(G/U)=f^*[\mathfrak{u}^{\perp}]$. Let $\xi:=f^*\dot{f_0}$, we can define the Lie bracket $[\xi,v]$ for any $v\in \mathrm{T}\widetilde{\Sigma}$. 

\begin{corollary}
    $\bar\partial^\nabla_v\xi-2\left(\operatorname{pr}_{\overline{1},-\overline{1}}([\xi,v])+\mathcal{J}\circ\operatorname{pr}_{\overline{1},-\overline{1}}([\xi,\mathsf{j}v])\right)$ is tangent to $f$.
\end{corollary}

\begin{proof}
    Note that $\bar\partial^\nabla=\dfrac{1}{2}\left(\nabla+\mathcal{J}\circ\nabla\circ\mathsf{j}\right)$ and the torsion is given by $T^{\mathrm{can}}(\mathsf{x}^*,\mathsf{y}^*)=-[\mathsf{x}^*,\mathsf{y}^*]$ for any left invariant vector fields $\mathsf{x}^*,\mathsf{y}^*$. The corollary follows directly from applying $\mathcal{J}$ to \prettyref{lem:Jacobihol}.
\end{proof}

\subsection{Jacobi field of cyclic surfaces}\label{sec:cycPff}

Below we will mainly focus on the smooth variation of equivariant cyclic surfaces in the $\Theta$-cyclic space $\mathbb{X}_\Theta$. Motivated by \prettyref{prop:JacobiPff} and \prettyref{lem:Jacobihol}, we give the following definition:

\begin{definition}\label{defn:Jacobi}
    A \textbf{Jacobi field} of the equivariant cyclic surface $(\rho,f)$ on $\Sigma$ in $\mathbb{X}_\Theta$ is a section $\xi$ of $\pi_*f^*\mathrm{T}\mathbb{X}_{\Theta}\to\Sigma$, such that for any $\Omega=\pi_*f^*\widehat{\Omega}$ with $\widehat{\Omega}\in\mathcal{I}_\Theta$, we have
    \begin{enumerate}[label=(\roman*)]
        \item $\dd^\nabla\left(\iota_{\xi}\Omega\right)=-\iota_\xi\left(\dd^\nabla\Omega\right)$,
    where $\iota_{\xi}\Omega:=\pi_* f^*(\iota_{f_*\pi^*\xi}\widehat{\Omega})$;

    \item $\bar\partial^\nabla\xi-2\left(\operatorname{pr}_{\overline{1},-\overline{1}}\circ\operatorname{ad}\xi+\mathcal{J}\circ\operatorname{pr}_{\overline{1},-\overline{1}}\circ\operatorname{ad}\xi\circ\mathsf{j}\right)\in\mathcal{A}^1(\Sigma,\mathrm{T}\Sigma).$
    \end{enumerate}
\end{definition}

Since $\mathrm{T}\mathbb{X}_\Theta\subseteq[\mathfrak{g}]$, a Jacobi field $\xi$ is also a smooth section of $\mathbb{E}[\mathfrak{g}]\to\Sigma$.

\begin{proposition}\label{prop:varJac}
    Given a smooth variation $(\varrho,F)$ of a cyclic surface $(\rho,f)$ on $\Sigma$ in $\mathbb{X}_\Theta$. When $\dot{\rho_0}=0$, $f^*\dot{f_0}\in\mathcal{A}^0(\widetilde{\Sigma},f^*\mathrm{T}\mathbb{X}_\Theta)$ is $\rho$-equivariant and hence it descends to a section $\pi_*f^*\dot{f_0}$ of $\pi_*f^*\mathrm{T}(G/U)\to\Sigma$.
\end{proposition}

\begin{proof}
    Note that $\rho_t(\gamma)\cdot f_t(x)=f_t(\gamma\cdot x)$ for every $\gamma\in\pi_1(\Sigma)$, $x\in\widetilde{\Sigma}$. By taking derivative we obtain that $\rho_0(\gamma)\cdot\dot{f_0}(x)=\dot{f_0}(\gamma\cdot x)$ since $\dot{\rho_0}=0$. Therefore, $f^*\dot{f_0}$ descends to a global section $\pi_*f^*\dot{f_0}$ of $\pi_*f^*\mathrm{T}\mathbb{X}_\Theta\to\Sigma$. By \prettyref{prop:JacobiPff} and \prettyref{lem:Jacobihol}, it is a Jacobi field of $(\rho,f)$.
\end{proof}

Recall that there is a norm $|\bullet|_h$ on $\mathbb{E}[\mathfrak{g}]$ associated with the pullback inner product $B_{\tau_f}=B_{\tau_h}$. Let $|\varphi|_h$ be the $(1,0)$-form which maps any tangent vector $v$ to $|\varphi(v)|_h$ and $|\varphi|_{h,g^{\mathcal{T}_X}}$ be its norm with respect to some conformal metric $g^{\mathcal{T}_X}$.

\begin{definition}\label{defn:admissible}
A Jacobi field
\(\xi\) of the equivariant cyclic surface \((\rho,f)\) is called
admissible if either 
(i) there exists a Liouville pair $(g,\mathcal C)$ on $X$ such that
the function \(|\xi|_h^2\) satisfies the condition
\(\mathcal C\) or 
(ii) there exists a complete conformal metric $g^{\mathcal{T}_X}$ such that      
        \[\liminf_{r\to+\infty}\frac{1}{r}\int_{B(x,r;g^{\mathcal{T}_X})}|\xi|_h^2 \cdot|\varphi|_{h, g^{\mathcal{T}_X}}\vol_{g^{\mathcal{T}_X}}=0,\]
        and
        \[\liminf_{r\to+\infty}\frac{1}{r}\int_{B(x,r;g^{\mathcal{T}_X})}|\xi|_h^p\vol_{g^{\mathcal{T}_X}}=0,\qquad\text{for some } p>0.\]

We say a smooth variation $(\rho_t,f_t)$ of $(\rho,f)$ with $\dot{\rho}_0=0$ is \textbf{admissible} if its associated Jacobi field $\xi = \pi_*\dot{f}_0$ (which is well-defined according to \prettyref{prop:varJac}) is admissible.
\end{definition}

\begin{proposition}\label{prop:admissible_cyclic_sufficient}
    A Jacobi field $\xi$ of the equivariant cyclic surface $(\rho,f)$ is admissible if one of the following conditions is satisfied:
   \begin{enumerate}
        \item when $X$ is a potential-theoretically parabolic Riemann surface, $|\xi|_h$ is bounded;
         \item there exists a complete conformal metric $g^{\mathcal{T}_X}$ such that 
        \[\liminf_{r\to+\infty}\frac{1}{r}\int_{B(x,r;g^{\mathcal{T}_X})}|\xi|_h^p\vol_{g^{\mathcal{T}_X}}=0,\qquad\text{for some } p>2;\]
         In particular, this holds if:
     \begin{itemize}
       \item $\xi$ is compactly supported;
         \item $|\xi|_h \in L^p(g^{\mathcal{T}_X})$ for some $p>2$ and some complete conformal metric $g^{\mathcal{T}_X}$;
        \item $X=\mathbb D$ and $|\xi|_h = O\big((1-|z|)^{1/2}\big)$; 
     \end{itemize}
        \item $|\xi|_h \in L^2(\Phi)$ if $\Phi=\overline{|\varphi|_h}\otimes|\varphi|_h$ defines a complete metric.
       \end{enumerate}
\end{proposition}
\begin{proof}
The first two cases satisfy admissibility condition (i) by checking the Liouville conditions. \\
Part 1 follows from the definition of potential-theoretical parabolic Riemann surface.\\
Part 2 follows from the theorem of Yau \cite[\S1 \& Theorem 3 \& Appendix]{yau1976some} stating that on a complete Riemannian manifold, any nongative subharmonic function has to be constant if for some $p>1$, $\liminf_{r\to+\infty}\dfrac{1}{r}\cdot\int_{B(x,r;g)}u^p\vol_{g}=0.$ Next, we verify the examples satisfy Part 2. 
    \begin{itemize}%[label=(\arabic*), leftmargin=*]
        \item $\xi$ is compactly supported or     $|\xi|_h \in L^p(g^{\mathcal{T}_X})$ for some $p>2$: 
       The global integral $\int_X |\xi|_h^p \vol_{g^{\mathcal{T}_X}}$ evaluates to a finite constant $M$. Thus, for any $r > 0$, the integral over the ball $B(x, r; g^{\mathcal{T}_X})$ is uniformly bounded by $M$. Dividing by $r$ and taking the limit gives:
        \[ \lim_{r \to +\infty} \frac{1}{r} \int_{B(x,r;g^{\mathcal{T}_X})} |\xi|_h^p \vol_{g^{\mathcal{T}_X}} = \lim_{r \to +\infty} \frac{M}{r} = 0. \]
        \item $X=\mathbb D$ and $|\xi|_h = O\big((1-|z|)^{1/2}\big)$:
        We equip the unit disk with the complete hyperbolic Poincar\'e metric $g^{\mathcal{T}_X} = \frac{4|\mathrm{d}z|^2}{(1-|z|^2)^2}$. The geodesic distance from the origin is $r = \log\big(\frac{1+|z|}{1-|z|}\big)$, which yields the asymptotic relation $1-|z| \asymp e^{-r}$ as $r \to +\infty$.
        The hypothesis $|\xi|_h = O\big((1-|z|)^{1/2}\big)$ therefore translates to an exponential decay bound $|\xi|_h \leqslant C e^{-r/2}$. 
        In geodesic polar coordinates, the volume form is $\vol_{g^{\mathcal{T}_X}} = \sinh(s) \mathrm{d}s \mathrm{d}\theta \leqslant \frac{1}{2} e^s \mathrm{d}s \mathrm{d}\theta$. For any choice of $p > 2$, we evaluate the total $L^p$ integral over $\mathbb{D}$:
        \[ \int_{\mathbb{D}} |\xi|_h^p \vol_{g^{\mathcal{T}_X}} \leqslant 2\pi \int_0^{+\infty} (C e^{-s/2})^p \sinh(s) \,\mathrm{d}s \leqslant \pi C^p \int_0^{+\infty} e^{(1 - p/2)s} \,\mathrm{d}s. \]
        Because $p > 2$, the exponent $1 - p/2$ is strictly negative, ensuring the global integral converges to a finite constant. Normalizing this uniformly bounded ball integral by dividing by $r$ and letting $r \to +\infty$ yields $0$, satisfying Part 2.

        \end{itemize}
Part 3 is by checking admissibility condition (ii).  $|\xi|_h \in L^2(\Phi)$ if $\Phi=\overline{|\varphi|_h}\otimes|\varphi|_h$ defines a complete metric:
        Set $g^{\mathcal{T}_X} = \Phi$. Under this metric, the pointwise conformal norm of the Higgs field is unity: $|\varphi|_{h, \Phi} = |\varphi|_h / \Phi^{1/2} \equiv 1$.
        We verify the second term of criterion (ii) of \prettyref{defn:admissible} by setting $p = 2$. 
        Because $|\xi|_h \in L^2(\Phi)$, the total global integral $\int_X |\xi|_h^2 \vol_\Phi$ evaluates to a finite constant $M$.
        For the first term of the third criterion:
        \[ \frac{1}{r} \int_{B(x,r;\Phi)} |\xi|_h^2\cdot |\varphi|_{h, \Phi} \vol_\Phi = \frac{1}{r} \int_{B(x,r;\Phi)} |\xi|_h^2 \vol_\Phi \leqslant \frac{M}{r}. \]
        As $r \to +\infty$, this structurally guarantees the required limit is $0$. 
        The second term required is identically generated by our choice $p=2$.
    This exhausts all cases and completes the proof.
\end{proof}

\begin{remark}
For example, for cyclic Higgs bundle parametrized by holomorphic $n$-differentials $q_n$ in the $\mathrm{SL}_n\mathbb R$-Hitchin section over a hyperbolic surface $X$, by the work of Li--Mochizuki \cite{li2025complete}, there exists a diagonal harmonic metric $h$ such that the induced metric $\Phi$ is complete. 
\end{remark}

\begin{definition}
  An equivariant cyclic surface $(\rho,f)$ on $\Sigma$ in $\mathbb{X}_\Theta$ is called \emph{irreducible} if the image of $f$ is not contained in any $L$-orbit
  \[
    L/(L\cap H^\Theta),
  \]
  where $L<G$ is a Levi subgroup.
\end{definition}

\begin{proposition}\label{prop:realorbit}
    When $m=2k$ is even, we have the real form $\Lambda=\Theta^k\circ\tau$. Then an equivariant cyclic surface $(\rho,f)$ on $\Sigma$ in $\mathbb{X}_\Theta$ is not irreducible if and only if it lies in an $L_\mathbb{R}$-orbit, where $L_\mathbb{R}$ is a Levi subgroup of $G^\Lambda$.
\end{proposition}

\begin{proof}
    Recall \prettyref{prop:basic}, we know that the image of $\dd f$ lies in $[\mathfrak{g}^\lambda]$, where $\lambda$ is the derivative of $\Lambda$. If $f$ is not irreducible, then it lies in an $L_{\mathsf{h}}$-orbit of $\mathbb{X}_\Theta$, where $\mathsf{h}\in\iu\mathfrak{h}$. Therefore, $\dd f$ takes value in $[\mathfrak{l}_{\mathsf{h}}\cap\mathfrak{g}^\lambda]$. Thus $f$ lies in the $L_{\mathbb{R}}$-orbit, where $L_{\mathbb{R}}:=L_{\mathsf{h}+\lambda(\mathsf{h})}\cap G^\Lambda\geqslant L_{\mathsf{h}}\cap L_{\lambda(\mathsf{h})}\cap G^\Lambda$ and $L_{\mathbb{R}}<G^\Lambda$. Conversely, if $f$ lies in an $L_\mathbb{R}$-orbit, then $f$ also lies in the orbit of its complexification.
\end{proof}

\begin{definition}
  An equivariant cyclic surface $(\rho,f)$ on $\Sigma$ in $\mathbb{X}_\Theta$ is said to arise from an equivariant $\mathfrak{c}$-cyclic surface if there exist a subgroup $U<H^\Theta$, a $U$-cyclic subspace $\mathfrak{c}$ and a smooth variation $(\rho_t, \hat{f_t})_{t\in (-\varepsilon, \varepsilon)}$ of equivariant $\mathfrak{c}$-cyclic surface such that every $f_t$ is the projection of $\hat{f_t}$ from $G/U$ to $\mathbb{X}_\Theta$.
\end{definition}

Our main theorem is the following infinitesimal rigidity result.

\begin{theorem}\label{thm:main}
  Let $(\rho_t,f_t)_{t\in(-\varepsilon,\varepsilon)}$ be an admissible smooth variation of an equivariant irreducible cyclic surface $(\rho,f)$ on $\Sigma$ in $\mathbb{X}_\Theta$, with $\dot\rho_0=0$. Suppose that the variation arises from a smooth variation of equivariant $\mathfrak{c}$-cyclic surfaces. Then there exists a smooth path $(\psi_t)$ in $\operatorname{Diff}^0(\Sigma)$ with $\psi_0=\operatorname{id}_\Sigma$ such that, for
  \[
    f'_t=f_t\circ\psi_t,
  \]
  one has $\dot f'_0=0$.
\end{theorem}

\begin{remark}
    Suppose that the associated Riemann surface $X=\overline{X}\setminus D$. Then if the image of $f$ lies in an $L_{\mathsf{h}}$-orbit $\cong L_{\mathsf{h}}/(L_{\mathsf{h}}\cap H^\Theta)$ for a proper Levi subgroup $L_{\mathsf{h}}<G$ with $\mathsf{h}\in\iu\mathfrak{h}$, $(\rho,\exp(t\mathsf{h})\cdot{f})$ gives an admissible smooth variation which is not tangent to the surface.
\end{remark}

The proof follows the ideas introduced by Labourie in \cite{labourie2017cyclic}. We postpone the proof until \prettyref{sec:inf}.

Below we fix a Jacobi field $\xi$ of the equivariant cyclic surface $(\rho,f)$ on $\Sigma$ in $\mathbb{X}_\Theta$ decomposing as
\[\xi =\sum_{\overline{j}\in\mathbb{Z}/m\mathbb{Z}}\xi_{\overline{j}},\]
where $\xi_{\overline{j}}=\iota_{\xi}(\pi_*f^*\omega_{\overline{j}})$ is a section of $\pi_*f^*[\mathfrak{g}_{\overline{j}}]$. Since $\xi_0$ is a section of $\pi_*f^*[\mathfrak{m}_{\overline{0}}]$, $\tau_f(\xi_{\overline{0}})=-\xi_{\overline{0}}$. Define
\[Z:=\sum_{\substack{\overline{j}\in\mathbb{Z}/m\mathbb{Z}\\\overline{j}\neq\overline{0},\overline{1},-\overline{1}}}\xi_{\overline{j}}=\iota_\xi(\pi_*f^*\omega_Z).
\]
Hence
\[\xi=\xi_{\overline{0}} + \xi_{\overline{1}} + \xi_{-\overline{1}} +Z.\] 

\subsection{Computations of derivatives}

Keeping the notations in \prettyref{sec:cycPff}. We will produce a sequence of lemmata in the following two sections to prove the infinitesimal rigidity. We first do some computations of derivatives.

\begin{lemma}\label{lem:firstd0}
    \[\begin{cases}
\partial^\nabla\xi_{\overline{0}}=\left[\dfrac{\xi_{-\overline{1}}+\tau_f(\xi_{\overline{1}})}{2}\wedge\varphi\right],\\
        \quad\\       \bar\partial^\nabla\xi_{\overline{0}}=\left[\dfrac{\xi_{\overline{1}}+\tau_f(\xi_{-\overline{1}})}{2}\wedge\varphi^\dagger\right].
    \end{cases}\]
\end{lemma}

\begin{proof}
    \[\begin{aligned}
        \nabla\xi_{\overline{0}}=&\dd^\nabla\operatorname{pr}_{\mathfrak{m}_{\overline{0}}}\left(\iota_{\xi}(f^*\omega_{\overline{0}})\right)\\
        =&\dd^\nabla\left(\iota_{\xi}(f^*\omega_{\overline{0}})\right)\\
        =&-\iota_\xi\left(\dd^\nabla(f^*\omega_{\overline{0}})\right)\quad\mbox{(since }\omega_{\overline{0}}\in\boldsymbol{\Omega}_\Theta\mbox{)}\\
        =&-\iota_\xi f^*\left(\dd^A(\omega_{\overline{0}})\right)\\
        =&\dfrac{1}{2}\operatorname{pr}_{\mathfrak{m}_{\overline{0}}}\left(\iota_\xi f^*\left(\sum\limits_{\overline{j}\in\mathbb{Z}/m\mathbb{Z}}{[\omega_{\overline{j}}\wedge\omega_{-\overline{j}}]}\right)\right)\quad\mbox{(by Maurer--Cartan equation \prettyref{eq:MC})}\\
        =&\dfrac{1}{2}\operatorname{pr}_{\mathfrak{m}_{\overline{0}}}\left(\sum_{\overline{j}\in\mathbb{Z}/m\mathbb{Z}}\left([\xi_{\overline{j}}\wedge f^*\omega_{-\overline{j}}]+[\xi_{\overline{-j}}\wedge f^*\omega_{\overline{j}}]\right)\right)\\
        =&\operatorname{pr}_{\mathfrak{m}_{\overline{0}}}\left([\xi_{\overline{1}}\wedge\varphi^\dagger]+[\xi_{\overline{-1}}\wedge\varphi]\right)\\
        =&\left[\dfrac{\xi_{-\overline{1}}+\tau_f(\xi_{\overline{1}})}{2}\wedge\varphi\right]+\left[\dfrac{\xi_{\overline{1}}+\tau_f(\xi_{-\overline{1}})}{2}\wedge\varphi^\dagger\right]\quad\mbox{(by }\operatorname{pr}_{\mathfrak{m}_{\overline{0}}}(\bullet)=(\bullet-\tau_f(\bullet))/2\mbox{)}.
    \end{aligned}\]
    Hence the lemma follows by taking $(1,0)$-part and $(0,1)$-part.
\end{proof}

\begin{lemma}\label{lem:firstd1}
    \[
    \begin{cases}
        \left(\iota_\xi\dd^\nabla\varphi\right)^{(1,0)}=-[\xi_{\overline{0}}\wedge\varphi],\\
        \left(\iota_\xi\dd^\nabla\varphi\right)^{(0,1)}=-\operatorname{pr}_{\overline{1}}\left(\left[(Z+\xi_{-\overline{1}})\wedge\varphi^\dagger\right]\right),
    \end{cases}
    \begin{cases}
        \left(\iota_\xi\dd^\nabla\varphi^{\dagger}\right)^{(0,1)}=-[\xi_{\overline{0}}\wedge\varphi^\dagger],\\
        \left(\iota_\xi\dd^\nabla\varphi^{\dagger}\right)^{(1,0)}=-\operatorname{pr}_{-\overline{1}}\left(\left[(Z+\xi_{\overline{1}})\wedge\varphi\right]\right).
    \end{cases}\]
\end{lemma}

\begin{proof}
    \[\begin{aligned}
        &\iota_\xi\dd^\nabla\varphi\\
        =&\iota_\xi f^*\left(\dd^A\omega_{\overline{1}}\right)\\
        =&-\iota_\xi\left( f^*\operatorname{pr}_{\overline{1}}\left(\dfrac{1}{2}[\omega\wedge\omega]\right)\right)\quad\mbox{(by Maurer--Cartan equation \prettyref{eq:MCtotal})}\\
        =&-\iota_\xi\left( f^*\left([\omega_{\overline{0}}\wedge\omega_{\overline{1}}]+\operatorname{pr}_{\overline{1}}\left([\omega_Z\wedge\omega_{-\overline{1}}]+\dfrac{1}{2}[\omega_{-\overline{1}}\wedge\omega_{-\overline{1}}]+\dfrac{1}{2}[\omega_Z\wedge\omega_Z]\right)\right)\right)\quad\mbox{(by \prettyref{eq:bracket})}\\
        =&-[\xi_{\overline{0}}\wedge\varphi]-\operatorname{pr}_{\overline{1}}\left(\left[(Z+\xi_{-\overline{1}})\wedge\varphi^\dagger\right]\right).
    \end{aligned}\]
    By taking $(1,0)$-part and $(0,1)$-part we obtain the first two equations. Symmetrically, we get the last two equations.
\end{proof}

\begin{lemma}\label{lem:firstdr}
    \[\begin{cases}
        \partial^\nabla\left(\xi_{\overline{1}}+\tau_f(\xi_{-\overline{1}})\right)=2[\xi_{\overline{0}}\wedge\varphi],\\
        \bar\partial^\nabla\left(\xi_{-\overline{1}}+\tau_f(\xi_{\overline{1}})\right)=2[\xi_{\overline{0}}\wedge\varphi^\dagger].
    \end{cases}\]
\end{lemma}

\begin{proof}
    On the one hand, since $\omega_{\overline{1}}+\tau(\omega_{-\overline{1}})\in\boldsymbol{\Omega}_\Theta$,  
    \[\iota_\xi\left(\dd^\nabla\left(\varphi+\tau_f(\varphi^\dagger)\right)\right)=-\dd^\nabla\left(\iota_\xi\left(\varphi+\tau_f(\varphi^\dagger)\right)\right)=-\dd^\nabla\left(\xi_{\overline{1}}+\tau_f(\xi_{-\overline{1}})\right).\]
    On the other hand, by \prettyref{lem:firstd1}
    \[\begin{aligned}
        &\iota_\xi\left(\dd^\nabla\left(\varphi+\tau_f(\varphi^\dagger)\right)\right)\\
        =&\iota_\xi\dd^\nabla\varphi+\tau_f\left(\iota_\xi\dd^\nabla\varphi^\dagger\right)\\
        =&-[\xi_{\overline{0}}\wedge\varphi]-\operatorname{pr}_{\overline{1}}\left(\left[(Z+\xi_{-\overline{1}})\wedge\varphi^\dagger\right]\right)+\tau_f\left(-[\xi_{\overline{0}}\wedge\varphi^\dagger]-\operatorname{pr}_{-\overline{1}}\left(\left[(Z+\xi_{\overline{1}})\wedge\varphi\right]\right)\right).
    \end{aligned}\]
    Hence
    \[\partial^\nabla\left(\xi_{\overline{1}}+\tau_f(\xi_{-\overline{1}})\right)=[\xi_{\overline{0}}\wedge\varphi]+\tau_f\left([\xi_{\overline{0}}\wedge\varphi^\dagger]\right)=2[\xi_{\overline{0}}\wedge\varphi]\]
    by $\tau_f(\xi_{\overline{0}})=-\xi_{\overline{0}}$ and $\tau_f(\varphi)=-\varphi^\dagger$. Similarly we have the other equation.
\end{proof}

\begin{lemma}\label{lem:firstdZ}
    \[\begin{cases}
        \partial^\nabla Z=\operatorname{pr}_Z\left(\left[\left(Z+\xi_{\overline{1}}\right)\wedge\varphi\right]\right),\\
        \bar\partial^\nabla Z=\operatorname{pr}_Z\left(\left[\left(Z+\xi_{-\overline{1}}\right)\wedge\varphi^\dagger\right]\right).\\
    \end{cases}\]
\end{lemma}

\begin{proof}
    \[\begin{aligned}
        &\nabla Z\\=&\dd^\nabla\left(\iota_\xi f^*\omega_Z\right)\\
        =&-\iota_\xi f^*\left(\dd^A\omega_{Z}\right)\quad\mbox{(since }\omega_{Z}\in\boldsymbol{\Omega}_\Theta\mbox{)}\\
        =&\iota_\xi\left( f^*\operatorname{pr}_{Z}\left(\dfrac{1}{2}[\omega\wedge\omega]\right)\right)\quad\mbox{(by Maurer--Cartan equation \prettyref{eq:MCtotal})}\\
        =&\iota_\xi\Bigg( f^*\bigg([\omega_{\overline{0}}\wedge\omega_{Z}]+\operatorname{pr}_{Z}\big([\omega_Z\wedge\omega_{\overline{1}}]+[\omega_Z\wedge\omega_{-\overline{1}}]+\dfrac{1}{2}[\omega_{\overline{1}}\wedge\omega_{\overline{1}}]\\
        &\quad\quad\quad\quad\quad\quad\quad\quad\quad\quad\quad\quad\quad+\dfrac{1}{2}[\omega_{-\overline{1}}\wedge\omega_{-\overline{1}}]+\dfrac{1}{2}[\omega_Z\wedge\omega_Z]\big)\bigg)\Bigg)\quad\mbox{(by \prettyref{eq:bracket})}\\
        =&\operatorname{pr}_{Z}\left(\left[(Z+\xi_{\overline{1}})\wedge\varphi\right]+\left[(Z+\xi_{-\overline{1}})\wedge\varphi^\dagger\right]\right).
    \end{aligned}\]
    Hence the lemma follows by taking $(1,0)$-part and $(0,1)$-part.
\end{proof}

\begin{lemma}\label{lem:firstdint}
    \[\begin{cases}
        [\bar\partial^\nabla\xi_{\overline{1}}\wedge\varphi]=\left[\operatorname{pr}_{\overline{1}}\left([(Z+\xi_{-\overline{1}})\wedge\varphi^\dagger]\right)\wedge\varphi\right]\\
        [\partial^\nabla\xi_{-\overline{1}}\wedge\varphi^\dagger]=\left[\operatorname{pr}_{-\overline{1}}\left([(Z+\xi_{\overline{1}})\wedge\varphi]\right)\wedge\varphi^\dagger\right].
    \end{cases}\]
\end{lemma}

\begin{proof}
    On the one hand,
    \[\begin{aligned}
        &\dd^\nabla\left(\iota_\xi\left[\varphi\wedge\varphi\right]\right)\\
        =&-\iota_\xi\left(\dd^\nabla\left[\varphi\wedge\varphi\right]\right)\quad\mbox{(since }\omega_{\overline{1}}\wedge\omega_{\overline{1}}\in\boldsymbol{\Omega}_\Theta\mbox{)}\\
        =&-\iota_\xi\left(2\left[\dd^\nabla\varphi\wedge\varphi\right]\right)\\
        =&-2\left(\left[\iota_\xi\dd^\nabla\varphi\wedge\varphi\right]-\left[\dd^\nabla\varphi\wedge\iota_\xi\left(\varphi\right)\right]\right)\\
        =&-2\left[\iota_\xi\dd^\nabla\varphi\wedge\varphi\right]\quad(\mbox{since }\dd^\nabla\varphi=0\mbox{ by \prettyref{prop:surfacetoHiggs}})\\
        =&-2\left[\left(\iota_\xi\dd^\nabla\varphi\right)^{(0,1)}\wedge\varphi\right]\quad(\mbox{since }\varphi\mbox{ is of }(1,0)\mbox{-type})\\
        =&2\left[\operatorname{pr}_{\overline{1}}\left(\left[(Z+\xi_{-\overline{1}})\wedge\varphi^\dagger\right]\right)\wedge\varphi\right].
    \end{aligned}\]
    On the other hand,
    \[\dd^\nabla\left(\iota_\xi\left[\varphi\wedge\varphi\right]\right)=2\dd^\nabla[\xi_{\overline{1}}\wedge\varphi]=2\left[\bar\partial^\nabla\xi_{\overline{1}}\wedge\varphi\right].\]
    Hence the first equation follows. Similarly, we obtain the second equation. 
\end{proof}

\subsection{Subharmonicity of certain terms}
Let $g$ be a conformal metric on $X$. The following lemma is a direct consequence of the Bochner formula and the
adjointness of the commutator operators.
\begin{lemma}\label{lem:subharmonic}
    Let $\eta$ be a section of $\pi_*f^*[\mathfrak{g}]$ satisfying that
    \[B_{\tau_h}(\bar\partial^\nabla\partial^\nabla\eta,\eta)=B_{\tau_h}([[\eta\wedge\varphi^\dagger]\wedge\varphi],\eta),\]
    \[B_{\tau_h}(\partial^\nabla\bar\partial^\nabla\eta,\eta)=B_{\tau_h}([[\eta\wedge\varphi]\wedge\varphi^{\dagger}],\eta),\]
     then \[\Delta_{g}|\eta|_h^2=|[\eta\wedge\varphi]|_{h,g}^2+|[\eta\wedge\varphi^{\dagger}]|_{h,g}^2+|\nabla \eta|_{h,g}^2.\]
\end{lemma}

\begin{corollary}\label{coro:secondd0} Let $\eta\in\{\xi_{\overline{0}},\xi_{\overline{1}}+\tau_f(\xi_{-\overline{1}}),Z\}$. Then 
\[\Delta_g|\eta|_h^2=|[\eta\wedge\varphi]|_{h,g}^2+|[\eta\wedge\varphi^{\dagger}]|_{h,g}^2+|\nabla \eta|_{h,g}^2.\]
As a corollary, $|\eta|_h^2$ is subharmonic.  

Moreover, $\big|d|\eta|_h^2\big|_g\leqslant  C|\xi|_h^2\cdot|\varphi|_{h,g},$ for a fixed constant $C.$
\end{corollary}
\begin{proof}
We just need to show that each $\eta$ here satisfies the assumption in \prettyref{lem:subharmonic}.
    \[\begin{aligned}
        \partial^\nabla\bar\partial^\nabla\xi_{\overline{0}}
        &=\partial^\nabla\left[\dfrac{\xi_{\overline{1}}+\tau_f(\xi_{-\overline{1}})}{2}\wedge\varphi^\dagger\right]\quad\mbox{(by \prettyref{lem:firstd0})}\\
        &=\left[\partial^\nabla\left(\dfrac{\xi_{\overline{1}}+\tau_f(\xi_{-\overline{1}})}{2}\right)\wedge\varphi^\dagger\right]+\left[\dfrac{\xi_{\overline{1}}+\tau_f(\xi_{-\overline{1}})}{2}\wedge\partial^\nabla\varphi^\dagger\right]\\
        &=\left[\partial^\nabla\left(\dfrac{\xi_{\overline{1}}+\tau_f(\xi_{-\overline{1}})}{2}\right)\wedge\varphi^\dagger\right]\quad\mbox{(by \prettyref{prop:surfacetoHiggs})}\\
        &=[[\xi_{\overline{0}}\wedge\varphi]\wedge\varphi^\dagger]\quad\mbox{(by \prettyref{lem:firstdr})}.
    \end{aligned}\]
    Similarly, we have
    \[\bar\partial^\nabla\partial^\nabla\xi_{\overline{0}}=[[\xi_{\overline{0}}\wedge\varphi^\dagger]\wedge\varphi].\]

    \[\begin{aligned}
        \partial^\nabla\bar\partial^\nabla\left(\xi_{\overline{1}}+\tau_f(\xi_{-\overline{1}})\right)
        &=2\partial^\nabla[\xi_{\overline{0}}\wedge\varphi^\dagger]\quad\mbox{(by \prettyref{lem:firstdr})}\\
        &=2[\partial^\nabla\xi_{\overline{0}}\wedge\varphi^\dagger]\quad\mbox{(by \prettyref{prop:surfacetoHiggs})}\\
        &=\left[\left[\left(\xi_{\overline{1}}+\tau_f(\xi_{-\overline{1}})\right)\wedge\varphi\right]\wedge\varphi^\dagger\right]  \quad\mbox{(by \prettyref{lem:firstd0})}. 
    \end{aligned}\]
    Similarly, we have
    \[\bar\partial^\nabla\partial^\nabla\left(\xi_{\overline{1}}+\tau_f(\xi_{-\overline{1}})\right)=\left[\left[\left(\xi_{\overline{1}}+\tau_f(\xi_{-\overline{1}})\right)\wedge\varphi^\dagger\right]\wedge\varphi\right].\]

    \[\begin{aligned}
        \bar\partial^\nabla\partial^\nabla Z
        &=\bar\partial^\nabla\left(\operatorname{pr}_Z\left(\left[\left(Z+\xi_{\overline{1}}\right)\wedge\varphi\right]\right)\right)\quad\mbox{(by \prettyref{lem:firstdZ})}\\
        &=\operatorname{pr}_Z\left(\left[\bar\partial^\nabla\left(Z+\xi_{\overline{1}}\right)\wedge\varphi\right]\right)\quad\mbox{(by \prettyref{prop:surfacetoHiggs})}\\
        &=\operatorname{pr}_Z\Biggl(\left[\operatorname{pr}_Z\left(\left[\left(Z+\xi_{-\overline{1}}\right)\wedge\varphi^\dagger\right]\right)\wedge\varphi\right]\\
        &\quad\quad+\left[\operatorname{pr}_{\overline{1}}\left(\left[\left(Z+\xi_{-\overline{1}}\right)\wedge\varphi^\dagger\right]\right)\wedge\varphi\right]\Biggr)\quad\mbox{(by \prettyref{lem:firstdZ} and \prettyref{lem:firstdint}})\\
        &=\operatorname{pr}_Z\left(\left[\left[\left(Z+\xi_{-\overline{1}}\right)\wedge\varphi^\dagger\right]\wedge\varphi\right]\right)\quad\mbox{(by \prettyref{eq:bracket}})\\
        &=\operatorname{pr}_Z\left(\left[\left[Z\wedge\varphi^\dagger\right]\wedge\varphi\right]\right).
    \end{aligned}\]

    Similarly, we have
    \[\partial^\nabla\bar\partial^\nabla Z=\operatorname{pr}_Z\left([[Z\wedge\varphi]\wedge\varphi^\dagger]\right).\]
    Because $\operatorname{pr}_Z$ is an orthogonal projection onto $Z$, taking the inner product of the above results with $Z$ cleanly absorbs the projection, satisfying the hypothesis of \prettyref{lem:subharmonic}.

    For the moreover part, note that by \prettyref{lem:firstd0}, \prettyref{lem:firstdr}, and \prettyref{lem:firstdZ}, the first-order derivatives $\nabla \eta$ are linear combinations of the Lie brackets of the components of $\xi$ with $\varphi$ and $\varphi^\dagger$. Since $\eta$ is a component of $\xi$, we have $|\eta|_h \leqslant |\xi|_h$ and $|\nabla \eta|_h \leqslant C' |\xi|_h |\varphi|_{h,g}$. This directly implies $\big|d|\eta|_h^2\big|_g \leqslant 2|\eta|_h |\nabla\eta|_h \leqslant C|\xi|_h^2 |\varphi|_{h,g}$.
\end{proof}

\begin{lemma}\label{lem:firstdhol}
    If $|Z|_h$ and $|\xi_{\overline{0}}|_h$ are constant, then
    \[\begin{cases}
        \left[\xi_{\overline{1}}\wedge\varphi\right]=0,\\
        \left[\xi_{-\overline{1}}\wedge\varphi^\dagger\right]=0.
    \end{cases}\]
\end{lemma}

\begin{proof}
    When $m>3$, by the assumption that $|Z|_h$ is constant and \prettyref{coro:secondd0}, we have $\nabla Z = 0$ and $[Z\wedge\varphi]=0$. Since $\partial^\nabla Z = \operatorname{pr}_Z\left(\left[\left(Z+\xi_{\overline{1}}\right)\wedge\varphi\right]\right)$ by \prettyref{lem:firstdZ}, it follows that
    \[[\xi_{\overline{1}}\wedge\varphi]=\operatorname{pr}_Z\left([\xi_{\overline{1}}\wedge\varphi]\right)=\partial^\nabla Z - \operatorname{pr}_Z\left([Z\wedge\varphi]\right)=0.\]
    Similarly, $\left[\xi_{-\overline{1}}\wedge\varphi^\dagger\right]=\operatorname{pr}_Z\left(\left[\xi_{-\overline{1}}\wedge\varphi^\dagger\right]\right)=0$. Thus the lemma follows directly.

    When $m=3$, note that 
    \[\dd^A\omega_{\overline{1}}=-[\omega_{\overline{0}}\wedge\omega_{\overline{1}}]-\dfrac{1}{2}[\omega_{-\overline{1}}\wedge\omega_{-\overline{1}}]\]
    lies in the differential ideal $\mathcal{I}_{\boldsymbol{\Omega}_\Theta}$. Therefore,
    \[\begin{aligned}
        \dd^\nabla\left(\iota_\xi\left(\dd^\nabla\varphi\right)\right)
        &=-\iota_\xi\left(\dd^\nabla\left(\dd^\nabla\varphi\right)\right)\\
        &=\iota_\xi f^*\left(\dd^A\left([\omega_{\overline{0}}\wedge\omega_{\overline{1}}]+\dfrac{1}{2}[\omega_{-\overline{1}}\wedge\omega_{-\overline{1}}]\right)\right)\quad\mbox{(by \prettyref{eq:MC})}\\
        &=-\iota_\xi f^*\biggl(\left[\left[\omega_{-\overline{1}}\wedge\omega_{\overline{1}}\right]\wedge\omega_{\overline{1}}\right]+\dfrac{1}{2}\left[\left[\omega_{\overline{0}}\wedge\omega_{\overline{0}}\right]\wedge\omega_{\overline{1}}\right]\\
        &\quad\quad-\left[\omega_{\overline{0}}\wedge\left[\omega_{\overline{1}}\wedge\omega_{\overline{0}}\right]\right]-\dfrac{1}{2}\left[\omega_{\overline{0}}\wedge\left[\omega_{-\overline{1}}\wedge\omega_{-\overline{1}}\right]\right]\\
        &\quad\quad+\left[\left[\omega_{\overline{0}}\wedge\omega_{-\overline{1}}\right]\wedge\omega_{-\overline{1}}\right]+\dfrac{1}{2}\left[\left[\omega_{\overline{1}}\wedge\omega_{\overline{1}}\right]\wedge\omega_{-\overline{1}}\right]\biggr)\quad\mbox{(by \prettyref{eq:MC})}\\
        &=\left[\left[\varphi^\dagger\wedge\xi_{\overline{1}}\right]\wedge\varphi\right]-\left[\left[\varphi^\dagger\wedge\varphi\right]\wedge\xi_{\overline{1}}\right]-\left[\left[\xi_{\overline{1}}\wedge\varphi\right]\wedge\varphi^\dagger\right].
    \end{aligned}\]
    On the other hand,
    \[\begin{aligned}
        \dd^\nabla\left(\iota_\xi\left(\dd^\nabla\varphi\right)\right)
        &=\dd^\nabla\left(\iota_\xi f^*\left(\dd^A\omega_{\overline{1}}\right)\right)\\
        &=-\dd^\nabla\left(\iota_\xi f^*\left([\omega_{\overline{0}}\wedge\omega_{\overline{1}}]+\dfrac{1}{2}[\omega_{-\overline{1}}\wedge\omega_{-\overline{1}}]\right)\right)\quad\mbox{(by \prettyref{eq:MC})}\\
        &=-\dd^\nabla\left(\left[\xi_{\overline{0}}\wedge\varphi\right]+\left[\xi_{-\overline{1}}\wedge\varphi^\dagger\right]\right)\\
        &=-\left[\bar\partial^\nabla\xi_{\overline{0}}\wedge\varphi\right]-\left[\partial^\nabla\xi_{-\overline{1}}\wedge\varphi^\dagger\right]\quad\mbox{(by \prettyref{prop:surfacetoHiggs})}\\
        &=-\left[\partial^\nabla\xi_{-\overline{1}}\wedge\varphi^\dagger\right]\quad\mbox{(since } |\xi_{\overline{0}}|_h \mbox{ is constant, \prettyref{coro:secondd0} implies } \bar\partial^\nabla\xi_{\overline{0}}=0\mbox{)}\\
        &=-\left[\left[\xi_{\overline{1}}\wedge\varphi\right]\wedge\varphi^\dagger\right]\quad\mbox{(by \prettyref{lem:firstdint})}.
    \end{aligned}\]

    By comparing the above two sequences, we cancel the common term $-\left[\left[\xi_{\overline{1}}\wedge\varphi\right]\wedge\varphi^\dagger\right]$ to obtain that 
    \[\left[\left[\varphi^\dagger\wedge\xi_{\overline{1}}\right]\wedge\varphi\right]-\left[\left[\varphi^\dagger\wedge\varphi\right]\wedge\xi_{\overline{1}}\right]=0.\]
    Coupled with the Jacobi identity 
    \[\left[\left[\varphi^\dagger\wedge\xi_{\overline{1}}\right]\wedge\varphi\right]-\left[\left[\varphi^\dagger\wedge\varphi\right]\wedge\xi_{\overline{1}}\right]-\left[\left[\xi_{\overline{1}}\wedge\varphi\right]\wedge\varphi^\dagger\right]=0,\]
    it trivially yields that
    \[\left[\left[\xi_{\overline{1}}\wedge\varphi\right]\wedge\varphi^\dagger\right]=0.\]
    Hence
    \[0=B_{\tau_f}\left(\left[\left[\xi_{\overline{1}}\wedge\varphi\right]\wedge\varphi^\dagger\right],\xi_{\overline{1}}\right)=B_{\tau_f}\left(\left[\xi_{\overline{1}}\wedge\varphi\right],\left[\xi_{\overline{1}}\wedge\varphi\right]\right)\]
    implies that $\left[\xi_{\overline{1}}\wedge\varphi\right]=0$. Similarly we have $\left[\xi_{-\overline{1}}\wedge\varphi^\dagger\right]=0$. 
\end{proof}

\begin{lemma}\label{lem:admissibleimply}
Let $\eta\in\{\xi_{\overline{0}},\xi_{\overline{1}}+\tau_f(\xi_{-\overline{1}}),Z\}$. If $\xi$ is admissible, then $|\eta|_h$ is constant and $\nabla\eta=0$, $[\eta\wedge\varphi]=0$, $[\eta\wedge\varphi^\dagger]=0$.
\end{lemma}
\begin{proof}
By \prettyref{coro:secondd0}, $|\eta|_h^2$ is subharmonic. Suppose $\xi$ satisfies the (i) condition of admissibility. Since $|\eta|_h \leqslant C|\xi|_h$, we obtain $|\eta|_h$ is harmonic.  

Suppose $\xi$ satisfies the (ii) condition of admissibility. So we have \[\liminf_{r\to+\infty}\dfrac{1}{r}\cdot\int_{B(x,r;g)} |\xi|_h^2\cdot|\varphi|_{h,g}\vol_g =0,
\] then by \prettyref{coro:secondd0}, we have 
 \[\liminf_{r\to+\infty}\dfrac{1}{r}\cdot\int_{B(x,r;g)} \big|d|\eta|_h^2\big|_g\vol_g =0.
\] Recall the theorem of Yau
\cite[\S1 \& Theorem 3 \& Appendix]{yau1976some} stating that over a complete Riemannian manifold $(M^n, g)$, any subharmonic function has to be harmonic if 
$\liminf_{r\to+\infty}\dfrac{1}{r}\cdot\int_{B(x,r;g)} |du|_g\vol_g =0$.
Hence, it implies that $|\eta|_h^2$ is harmonic.

Consequently, its Laplacian vanishes ($\partial_z\partial_{\bar z}|\eta|_h^2=0$). By  Corollary \ref{coro:secondd0}, we have $|\eta|_h$ is constant and $\nabla\eta=0$, $[\eta\wedge\varphi]=0$, $[\eta\wedge\varphi^\dagger]=0$.
\end{proof}

\subsection{Proof of the infinitesimal rigidity}\label{sec:inf}

In this section, we prove our main result, \prettyref{thm:main}.

\begin{proof}[Proof of \prettyref{thm:main}]
    By \prettyref{prop:varJac}, we can define $\xi:=\pi_*\dot{f_0}$, which is an admissible Jacobi field of $(\rho_0,f_0)$. 
        
     Let $\eta\in\{\xi_{\overline{0}},\xi_{\overline{1}}+\tau_f(\xi_{-\overline{1}}),Z\}$. By Lemma \ref{lem:admissibleimply}, we obtain that $\nabla\eta=0$, $[\eta\wedge\varphi]=0$ and $[\eta\wedge\varphi^{\dagger}]=0$. Therefore, $\eta\in\operatorname{aut}(\mathbb{E},\varphi)$ is an infinitesimal automorphism. Moreover, we see $\eta$ satisfies the assumption in Proposition \ref{prop:factor} since $\xi$ is admissible. By \prettyref{prop:factor}, we obtain that $\eta$ must vanish since $f_0$ is irreducible. Thus $\xi=\xi_{\overline{1}}-\tau_{f_0}(\xi_{\overline{1}})\in\pi_*f^*([\mathfrak{g}_{\overline{1}}\oplus\mathfrak{g}_{-\overline{1}}])$. 
     
     Since $(\rho_t,f_t)$ arises from a smooth variation $(\rho_t,\hat{f_t})$ of equivariant $\mathfrak{c}$-cyclic surfaces, 
    \[\dot{f_0}=\operatorname{pr}_{\overline{1},-\overline{1}}\left(\dfrac{\dd}{\dd t}\bigg|_{t=0}\hat{f_t}\right).\] With respect to the holomorphic decomposition $f^*[\mathfrak{g}_{\overline{1}}]=\hat{f_0}^*[\mathfrak{g}_{\overline{1}}]=\mathcal{K}_X^\vee\oplus \hat{f_0}^*[\mathfrak{g}_{\overline{1}}/\mathfrak{c}]$ proven in \prettyref{prop:surfacetoHiggs}, we can decompose $\xi_{\overline{1}}$ as $\xi_{\overline{1}}^{\mathrm{T}}+\xi_{\overline{1}}^{\perp}$. Since $f$ is an immersion, there exists a unique vector field $v$ on $S$ such that $f^*\dd f(v)=-\xi_{\overline{1}}^{\mathrm{T}}-(\xi_{\overline{1}}^{\mathrm{T}})^\dagger$. Let $\psi_t$ be the flow generated by $v$. Now $(\rho_t,f_t':=f_t\circ\psi_t)$ is still an admissible smooth variation of $(\rho,f)$, but with Jacobi field $\zeta$ satisfying that $\zeta=\zeta_{\overline{1}}-\tau_{f_0}(\zeta_{\overline{1}})$ and $\zeta_{\overline{1}}^{\mathrm{T}}\equiv0$.
     
     Note that $\operatorname{pr}_{\overline{1},-\overline{1}}([\zeta,v])$ is always $0$ for any $v\in\mathrm{T}S$ since $[\zeta,v]$ lies in $\pi_*f^*([\mathfrak{g}_{\overline{2}}\oplus\mathfrak{g}_{\overline{0}}\oplus\mathfrak{g}_{-\overline{2}}])$. Therefore, by \prettyref{defn:Jacobi}, we have $\bar\partial^\nabla\zeta\in\mathcal{A}^1(\Sigma,\mathrm{T}\Sigma)$. But also $\bar\partial^\nabla\zeta_{\overline{1}}=\bar\partial^\nabla\zeta_{\overline{1}}^\perp$ takes value perpendicular to $\mathrm{T}\Sigma$. This yields that $\bar\partial^\nabla\zeta_{\overline{1}}=0$. 
     
     By \prettyref{lem:firstdhol}, we have $\left[\zeta_{\overline{1}}\wedge\varphi\right]=0$. Applying \prettyref{prop:factor} to $\zeta_{\overline{1}}=\zeta_{\overline{1}}^{\perp}$ we obtain that $\zeta_{\overline{1}}\equiv0$. Therefore, $\zeta\equiv0$ and $\dot{f_0'}= f_*\zeta\equiv0$.
\end{proof}

As a corollary, we can apply \prettyref{thm:main} to closed surfaces:

\begin{corollary}\label{coro:close}
    Let $(\rho_t, f_t)_{t\in (-\varepsilon, \varepsilon)}$ be a smooth variation of an equivariant cyclic surface $(\rho,f)$ on a closed surface $S$ in $\mathbb{X}_\Theta$ arising from a smooth variation of equivariant $\mathfrak{c}$-cyclic surface. We assume that $(\rho_t)_{t\in (-\varepsilon, \varepsilon)}$ descends to a smooth path $([\rho_t])_{t\in (-\varepsilon, \varepsilon)}$ in the character variety $\mathfrak{X}(S,G)$ with $[\dot{\rho_0}]=0$. If $(\rho,f)$ induces a stable $\Theta$-cyclic $G$-Higgs bundle, then there exists a smooth path  $(g_t,\psi_t)$ in $G\times\operatorname{Diff}^0(\Sigma)$ such that $\dot{f'_0}=0$ where $f'_t = g_t\circ f_t\circ\psi_t$.
\end{corollary}

\begin{proof}
    Since $[\dot{\rho_0}]=0$, there exists a smooth path $(g_t)$ in $G$ such that $\dfrac{\dd}{\dd t}\bigg|_{t=0}g_t\rho_tg_t^{-1}=0$. Then this is a direct corollary of \prettyref{thm:main} and \prettyref{prop:factor}.
\end{proof}

Similar to \cite{labourie2017cyclic,collier2024holomorphic,rungi2025complex}, we can translate it to the language of moduli space. Let $S$ be a closed oriented surface whose genus $\geqslant 2$. \prettyref{coro:close} can be interpreted as
\begin{theorem}\label{thm:immersion}
    The holonomy map $\operatorname{Hol}\colon\operatorname{Cyc}^{\mathfrak{c}}(S)\to\mathfrak{X}(S,G)$ which sends $[(\rho,f)]$ to $[\rho]$ is an immersion.
\end{theorem}

\section{\texorpdfstring{$n$}{n}-alternating surfaces in \texorpdfstring{$\mathbb H^{p,q}$}{Hpq}}\label{sec:alternating}

In this section, we introduce \(n\)-alternating surfaces in pseudo-Riemannian space forms. Using the general theory developed in \prettyref{sec:cyclicsurfaceshiggsbundles}, we relate them to cyclic harmonic bundles and cyclic surfaces, and then apply this correspondence to admissible variations. 

Let \(n\ge2\). 

\subsection{Definition and basic properties}
\begin{definition}
Let \(M\) be a pseudo-Riemannian manifold, and let \(\Sigma\subset M\) be an oriented smooth immersed spacelike surface. Denote by \(\langle\cdot,\cdot\rangle\) the metric on \(\mathrm{T}M|_\Sigma\), and by \(\nabla\) the Levi-Civita connection of \(M\) restricted to \(\Sigma\). An \emph{\(n\)-alternating surface} is an immersion together with an orthogonal splitting
\[
\mathrm{T}M|_\Sigma=V_1\oplus\cdots\oplus V_n
\]
satisfying the following conditions:
\begin{itemize}
    \item \(V_1=\mathrm{T}\Sigma\), hence \(\mathrm{N}\Sigma=V_2\oplus\cdots\oplus V_n\);
    \item \(\operatorname{rank}(V_i)=2\) for \(1\le i\le n-1\), and \(\operatorname{rank}(V_n)=\dim M-(2n-2)\ge1\);
    \item \(V_i\) is positive definite for odd \(i\), and negative definite for even \(i\);
    \item for \(Y\in C^\infty(\Sigma,\mathrm{T}\Sigma)\) and \(\xi\in C^\infty(\Sigma,V_j)\), one has \((\nabla_Y\xi)^{V_i}=0\) whenever \(|i-j|\ge2\);
    \item for \(2\le i\le n\), the \(\operatorname{Hom}(V_{i-1},V_i)\)-component of \(\nabla\) defines \(\alpha_i\in\Omega^1(\Sigma,\operatorname{Hom}(V_{i-1},V_i))\). If \(n\ge3\), then \(\alpha_i\) is weakly conformal for \(2\le i\le n-1\). If \(n=2\), we require \(\alpha_2\), equivalently the second fundamental form, to be trace-free.
\end{itemize}
Here weak conformality means for every \(Y\in T_x\Sigma\), the map \(\alpha_i(Y):(V_{i-1})_x\to(V_i)_x\) is either zero or conformal. We call the surface \emph{non-degenerate} if, for each \(2\le i\le n-1\), \(\alpha_i\) is not identically zero. Unless stated otherwise, \(n\)-alternating surfaces are assumed to be non-degenerate.

If \(\alpha_n\) is also weakly conformal, the surface is called \emph{superconformal}.
\end{definition}

For $2\leq i\leq n$, denote by $\beta_i$ the $\Hom(V_i,V_{i-1})$-component of $\nabla$. By metric compatibility, it is the negative adjoint of $\alpha_i$ with respect to $\langle\cdot,\cdot\rangle$ on $V_i$ and $V_{i-1}.$

\begin{remark}
 The second fundamental form is \(\Pi=\alpha_2\). Hence every \(n\)-alternating surface is minimal: for \(n\ge3\), this follows from the weak conformality of \(\alpha_2\); for \(n=2\), it is part of the definition.
\end{remark}

Let \(\mathbb R^{p,q+1}\) be endowed with a quadratic form of signature \((p,q+1)\), and set
\[
\mathbb H^{p,q}=\{x\in\mathbb R^{p,q+1}\mid \langle x,x\rangle=-1\}.
\]

\begin{example}
\begin{enumerate}[label=(\roman*)]
    \item \(2\)-alternating surfaces in \(\mathbb H^{2,q}\) correspond to maximal spacelike surfaces in \(\mathbb H^{2,q}\), as studied in \cite{collier2019geometry}.
    \item Alternating holomorphic curves in \(\mathbb H^{4,2}\), introduced in \cite{collier2024holomorphic}, are \(3\)-alternating surfaces.
    \item \(n\)-alternating surfaces in \(\mathbb H^{p,q}\) satisfying \(\dim V_n=2\) correspond to the \(A\)-surfaces in \(\mathbb H^{p,q}\) introduced in \cite{nie2024cyclic}.
\end{enumerate}
\end{example}
\begin{remark}
    Set \(k=\operatorname{rank}(V_n)-1\). The signature \((p,q)\) is \((n,n+k-1)\) for \(n\) even, and \((n+k,n-1)\) for \(n\) odd.
\end{remark}

\begin{lemma}\label{lem:orientability-splitting}
Let \(X\) be a Riemann surface.
\begin{enumerate}[label=(\roman*)]
    \item Let \(W\) be a real rank-\(2\) oriented vector bundle over \(X\), endowed with a positive-definite metric and an orientation-preserving metric connection. Let \(H_W\) and \(Q_W\) be the Hermitian and complex bilinear extensions of the metric to \(W^{\mathbb C}\). Then there is an Hermitian holomorphic line bundle \((L,h)\) such that
    \[
    (W^{\mathbb C},Q_W,H_W)\cong\left(L\oplus L^{-1},\begin{pmatrix}0&1\\1&0\end{pmatrix},\operatorname{diag}(h,h^{-1})\right),
    \]
    where the isomorphism $\overline{L}\cong L^{-1}$ is induced by the metric $h.$
    \item Let \(W_1\) be as in \((\mathrm{i})\), with \(W_1^{\mathbb C}=L_1\oplus L_1^{-1}\). Let \(W_2\) be a real rank-\(2\) bundle with a definite metric and a metric connection, and let \(\alpha\in\Omega^1(X,\operatorname{Hom}_{\mathbb R}(W_1,W_2))\) be weakly conformal. Suppose \(\alpha^{1,0}\) vanishes on \(L_1^{-1}\) and is holomorphic. If \(\alpha^{1,0}|_{L_1}\not\equiv0\), then \(W_2\) admits an orientation such that \(W_2^{\mathbb C}=L_2\oplus L_2^{-1}\) and
    \[
    \alpha^{1,0}\in H^0(X,\operatorname{Hom}(L_1,L_2)\otimes \KK_X),\qquad
    \alpha^{0,1}=\overline{\alpha^{1,0}}\in\Omega^{0,1}(X,\operatorname{Hom}(L_1^{-1},L_2^{-1})).
    \]
\end{enumerate}
\end{lemma}

\begin{proof}
For \((\mathrm{i})\), the metric and orientation define a unique orthogonal complex structure \(J\) on \(W\). Let \(L, L^{-1}\subset W^{\mathbb C}\) be the \(+i, -i\)-eigenbundles respectively. Then \(Q_W\) vanishes on \(L\), and non-degeneracy gives \(W^{\mathbb C}=L\oplus\overline L\). We identify \(\overline L\) with \(L^{-1}\). Since \(J\) is uniquely determined by the metric and the orientation, and the connection preserves both, it preserves \(J\) and its \((0,1)\)-part defines a holomorphic structure on \(L\).

For \((\mathrm{ii})\), weak conformality implies that \(\alpha^{1,0}(L_1)\) is isotropic. Away from the zero divisor of \(\alpha^{1,0}|_{L_1}\), its image defines a holomorphic isotropic line subbundle of \(W_2^{\mathbb C}\). The saturated image extends across the isolated zeros to a holomorphic isotropic line bundle \(L_2\subset W_2^{\mathbb C}\). This line determines the orientation on \(W_2\). The statement for \(\alpha^{0,1}\) follows from reality.
\end{proof}

Let \(f:\Sigma\to\mathbb H^{p,q}\) be an \(n\)-alternating surface and \(X\) be the Riemann surface determined by the induced metric and the orientation. Set
\[
E=f^*\mathrm{T}\mathbb R^{p,q+1}=V_0\oplus V_1\oplus\cdots\oplus V_n,\qquad V_0=\mathbb R f.
\]
We identify \(df\) with \(\alpha_1\in\Omega^1(\Sigma,\operatorname{Hom}(V_0,V_1))\) and denote by $\beta_1\in\Omega^1(\Sigma,\operatorname{Hom}(V_1,V_0))$ the negative adjoint of $\alpha_1$ with respect to $\langle\cdot,\cdot\rangle$. For $0\leq i\leq n$, let \(h_i=(-1)^{i+1}\langle\cdot,\cdot\rangle|_{V_i}\), and let \(H_i\) be its Hermitian extension of $h_i$ and \(Q_i\) be the complex bilinear extension of $h_i$. We write \(H=\bigoplus_iH_i\) and \(Q=\bigoplus_i(-1)^{i+1}Q_i\). Thus \(H\) is positive definite and $Q$ is the complex bilinear extension of $\langle\cdot,\cdot\rangle$. By metric compatibility, \(\beta_i\) is the adjoint of \(\alpha_i\) with respect to the positive definite metrics \(h_{i-1}\) and \(h_i\); equivalently, it is the negative adjoint with respect to the ambient indefinite form.

Write the trivial flat connection on \(E^{\mathbb C}\) as \(D=\nabla^H+\Psi_H\), where \(\nabla^H\) is \(H\)-unitary and \(\Psi_H\) is \(H\)-self-adjoint. Then \(\nabla^H\) is the complexification of the induced metric connection on the splitting, \(Q\) is holomorphic, and the flatness of $D$ gives \(d^{\nabla^H}\Psi_H=0\) and \(F(\nabla^H)+[\Psi_H\wedge\Psi_H]=0\). Moreover, \(\Psi_H=\alpha+\beta\), where \(\alpha=\sum_{i=1}^n\alpha_i\) and \(\beta=\sum_{i=1}^n\beta_i\).

\begin{lemma}\label{lem:symmetric-indexing}
There exist Hermitian holomorphic line bundles \(L_i\), \(1\le i\le n-1\), such that \[V_i^{\mathbb C}=L_i\oplus L_i^{-1}, \quad Q_{V_i}=\begin{pmatrix}0&1\\1&0\end{pmatrix}.\] Moreover, \(V_n^{\mathbb C}\) has an orthogonal holomorphic structure satisfying \(\det(V_n^{\mathbb C})=\mathcal O\). We have \(L_1=\mathrm{T}^{1,0}X\simeq \KK_X^\vee\), \((V_0^{\mathbb C},Q_0,H_0)\simeq(\mathcal O,1,1)\), \(\alpha_1^{1,0}:\mathcal O\to L_1\otimes \KK_X\) and
\(\beta_1^{1,0}:L_1^{-1}\to\mathcal O\otimes \KK_X\) are the tautological sections. Furthermore,
\[
\alpha_i^{1,0}\in H^0(X,\operatorname{Hom}(L_{i-1},L_i)\otimes \KK_X),\quad
\alpha_i^{0,1}=\overline{\alpha_i^{1,0}}\in\Omega^{0,1}(X,\operatorname{Hom}(L_{i-1}^{-1},L_i^{-1}))
\]
\[
\beta_i^{1,0}=(\alpha_i^{1,0})^{\dagger}\in H^0(X,\operatorname{Hom}(L_i^{-1},L_{i-1}^{-1})\otimes \KK_X),\quad
\beta_i^{0,1}=\overline{\beta_i^{1,0}}\in\Omega^{0,1}(X,\operatorname{Hom}(L_i,L_{i-1}))
\]
for \(2\le i\le n-1\), and
\[
\alpha_n^{1,0}\in H^0(X,\operatorname{Hom}(L_{n-1},V_n^{\mathbb C})\otimes \KK_X),\quad
\alpha_n^{0,1}=\overline{\alpha_n^{1,0}}\in\Omega^{0,1}(X,\operatorname{Hom}(L_{n-1}^{-1},V_n^{\mathbb C})).
\]
\[
\beta_n^{1,0}=(\alpha_n^{1,0})^{\dagger}\in H^0(X,\operatorname{Hom}(V_n^{\mathbb C},L_{n-1}^{-1})\otimes \KK_X),\quad
\beta_n^{0,1}=\overline{\beta_n^{1,0}}\in\Omega^{0,1}(X,\operatorname{Hom}(V_n^{\mathbb C},L_{n-1})).
\]
Here, $\dagger$ means the adjoint with respect to $Q_i$ and $Q_{i-1}$.
\end{lemma}
\begin{proof}
We know \(V_0^{\mathbb C}\simeq\mathcal O\) and \(L_1=\mathrm{T}^{1,0}X\simeq \KK_X^{-1}\). The trace-free condition on the second fundamental form gives \(\Pi(\partial_z,\partial_{\bar z})=0\), hence \(\alpha_2^{1,0}\) vanishes on \(L_1^{-1}\). The Codazzi equation in the space form, equivalently the \((0,1)\)-part of \(d^{\nabla^H}\Psi_H=0\), implies that \(\alpha_2^{1,0}\) is holomorphic.

If \(n=2\), this already gives the final-block statement for \(V_2=V_n\). Assume now \(n\ge3\). Since \(V_2\) has rank \(2\), \prettyref{lem:orientability-splitting} induces an orientation on \(V_2\) and gives \(V_2^{\mathbb C}=L_2\oplus L_2^{-1}\).

We prove by induction. Suppose the statement has been proved up to \(V_i\), where \(2\le i\le n-1\). The \(\operatorname{Hom}(L_{i-1},V_{i+1}^{\mathbb C})\)-component of \(F(\nabla^H)+[\Psi_H\wedge\Psi_H]=0\) gives \(\alpha_{i+1}^{0,1}\wedge\alpha_i^{1,0}=0\). Since \(\alpha_i^{1,0}\) is not identically zero, this implies \(\alpha_{i+1}^{0,1}|_{L_i}=0\) on a dense open set, hence everywhere. By reality, \(\alpha_{i+1}^{1,0}\) vanishes on \(L_i^{-1}\).

The \(V_i\to V_{i+1}\) component of \(d^{\nabla^H}\Psi_H=0\), restricted to \(L_i\), gives \((\nabla^H)^{0,1}\alpha_{i+1}^{1,0}=0\). Thus \(\alpha_{i+1}^{1,0}\) is holomorphic. If \(i+1<n\), \prettyref{lem:orientability-splitting} gives \(V_{i+1}^{\mathbb C}=L_{i+1}\oplus L_{i+1}^{-1}\). If \(i+1=n\), the same argument gives the final-block statement without splitting \(V_n^{\mathbb C}\).

Finally, note that \(\det E^{\mathbb C}\) is holomorphically trivial. Since \(\det(V_i^{\mathbb C})=\mathcal O\) for \(i<n\), it follows that \(\det(V_n^{\mathbb C})=\mathcal O\).

With respect to the positive definite forms \(h_{i-1}\) and \(h_i\), equivalently to their complex bilinear extensions \(Q_{i-1}\) and \(Q_i\), one has \(\beta_i=\alpha_i^\dagger\). This is the same as saying that \(\beta_i\) is the negative adjoint of \(\alpha_i\) with respect to the ambient indefinite form. Thus $\beta_i^{1,0}=(\alpha_i^{1,0})^{\dagger}$ and $\beta_i^{0,1}=(\alpha_i^{0,1})^{\dagger}$. The rest follows.
\end{proof}

\begin{definition}
For an \(n\)-alternating surface in \(\mathbb H^{p,q}\), let \(\Phi=\alpha_n\circ\alpha_{n-1}\circ\cdots\circ\alpha_1\), viewed as a section of
\[
\operatorname{Hom}(V_0,V_n)\otimes (\mathrm{T}^*\Sigma)^{\otimes n}.
\] The higher Hopf differential \(q_{2n}(\Sigma)\) is the \((2n,0)\)-part of \((-1)^{n+1}\langle\Phi,\Phi\rangle_{V_n}\).
\end{definition}
\begin{lemma}\label{lem:superconformal}
Let \(f\) be an \(n\)-alternating surface in \(\mathbb H^{p,q}\). Then:
\begin{enumerate}[label=(\roman*)]
    \item the higher Hopf differential \(q_{2n}(\Sigma)\) is holomorphic on \(X\);
    \item \(f\) is superconformal if and only if \(q_{2n}(\Sigma)=0\).
\end{enumerate}
\end{lemma}
\begin{proof}
Write \(\gamma_i=\alpha_i^{1,0}\) and set \(L_0=\mathcal O\).
By \prettyref{lem:symmetric-indexing}, for \(1\le i\le n-1\), \(\gamma_i\) is a holomorphic line-bundle map \(L_{i-1}\to L_i\otimes \KK_X\), and it is not identically zero. Locally, using \(Q_n=(-1)^{n+1}\langle\cdot,\cdot\rangle|_{V_n}\), one has
\[
q_{2n}
=
(\gamma_1\cdots\gamma_{n-1})^2Q_n(\gamma_n,\gamma_n).
\] This shows that \(q_{2n}\) is holomorphic, since all factors are holomorphic.
Since \(\gamma_1\cdots\gamma_{n-1}\) is not identically zero, \(q_{2n}=0\) is equivalent to
\[
Q_{V_n}(\gamma_n,\gamma_n)=0.
\] 

This last condition says that \(\gamma_n(L_{n-1})\) is isotropic. Together with the reality relation for \(\alpha_n^{0,1}\), this is equivalent to weak conformality of the underlying real map \(\alpha_n(Y):V_{n-1}\to V_n\) for every real tangent vector \(Y\). Thus \(\alpha_n\) is weakly conformal if and only if \(q_{2n}=0\). 
\end{proof}

In the following, we show that the alternating splitting of a non-degenerate \(n\)-alternating surface is unique, so there is no need to specify it as additional data. 
\begin{lemma}\label{lem:unique-alternating-splitting}
Let \(f:\Sigma\to\mathbb H^{p,q}\) be a non-degenerate \(n\)-alternating immersion. Then its alternating splitting $V_i$'s is uniquely determined by \(f\).
\end{lemma}
\begin{proof}
Firstly, the bundle $V_1$ is determined by the tangential map $df.$ For $2\leq i\leq n-1$, the bundle $V_i$ is determined by $V_1,\cdots, V_{i-1}$ inductively since $\alpha_i$ is determined by the connection $\nabla$ acting on $\oplus_{k=1}^{i-1}V_k$ and the image of $\alpha_i^{1,0}$ is full in $V_i$ except at the isolated zeros. The last bundle $V_n$ is obtained by taking the perpendicular bundle of $\oplus_{i=1}^{n-1}V_i$ in the pullback tangent bundle $f^*\mathbb H^{p,q}$ with respect to $\langle\cdot,\cdot\rangle.$
\end{proof}

\begin{definition}
An equivariant \(n\)-alternating surface in \(\mathbb H^{p,q}\) over \(\Sigma\) is a pair \((\rho,f)\), where \(\rho:\pi_1(\Sigma)\to\SO^0_{p,q+1}\) is a representation and \(f:\widetilde\Sigma\to\mathbb H^{p,q}\) is a \(\rho\)-equivariant \(n\)-alternating surface.
\end{definition}

\begin{definition}
Two equivariant \(n\)-alternating surfaces \((\rho_1,f_1)\) and \((\rho_2,f_2)\) are \emph{isomorphic} if there exist \(g\in\SO^0_{p,q+1}\) and \(\psi\in\operatorname{Diff}^0(\Sigma)\) such that
\[
(\rho_1,f_1)=\bigl(\operatorname{Ad}_g\circ\rho_2,(g\cdot f_2)\circ\widetilde\psi\bigr).
\]
They are \(\mathsf j\)-isomorphic if they induce the same complex structure \(\mathsf j\) and the same condition holds with \(\psi=\operatorname{id}_\Sigma\).
\end{definition}

When \(\Sigma\) is a closed surface of genus at least \(2\), the group \(\operatorname{Diff}^0(\Sigma)\) acts freely on the space of complex structures. Consequently, two equivariant \(n\)-alternating surfaces inducing the same complex structure \(\mathsf j\) are \(\mathsf j\)-isomorphic if and only if they are isomorphic.

We next describe the associated cyclic harmonic bundles. Let \(\mathcal V\) be an orthogonal holomorphic bundle of rank \(p+q-2n+2\), with \(\det\mathcal V=\mathcal O\), and let \(h_{\mathcal V}\) be a compatible Hermitian metric. Let \(L_1=\KK_X^\vee\). For \(2\le i\le n-1\), let \((L_i,h_i)\) be Hermitian holomorphic line bundles and let \(\gamma_i\in H^0(X,L_{i-1}^{-1}L_i\otimes\KK_X)\), with \(\gamma_i\not\equiv0\). Let \(\gamma_n\in H^0(X,L_{n-1}^{-1}\mathcal V\otimes\KK_X)\).  Set
\[
\mathcal E=\mathcal V\oplus L_{n-1}^{-1}\oplus\cdots\oplus L_1^{-1}\oplus\mathcal O\oplus L_1\oplus\cdots\oplus L_{n-1}.
\]
Define the orthogonal form \(Q\) on \(\mathcal E\) by
\[
Q|_{\mathcal O}=-1,\qquad
Q|_{L_i\oplus L_i^{-1}}
=
(-1)^{i+1}
\begin{pmatrix}
0&1\\
1&0
\end{pmatrix},
\qquad
Q|_{\mathcal V}=(-1)^{n+1}Q_{\mathcal V}.
\]
Let \(\gamma_n^\dagger:\mathcal V\to L_{n-1}^{-1}\otimes\KK_X\) be the block determined by the condition that the Higgs field is \(Q\)-skew-symmetric.
Equivalently, it is the adjoint of \(\gamma_n\) with respect to \(Q_{\mathcal V}\) and the natural pairing between \(L_{n-1}\) and \(L_{n-1}^{-1}\).
We consider the following cyclic \(\SO^0_{p,q+1}\)-harmonic bundle:
\begin{equation}\label{eqn:cyclicharmonicbundle}
\begin{aligned}
\begin{tikzcd}[column sep=small]
&&&\mathcal V\arrow[dlll, "\gamma_n^\dagger"']&&&\\
L_{n-1}^{-1}\arrow[r, "\gamma_{n-1}"']&
\cdots\arrow[r, "\gamma_2"']&
L_1^{-1}\arrow[r, "1"']&
\mathcal O\arrow[r, "1"']&
L_1\arrow[r, "\gamma_2"']&
\cdots\arrow[r, "\gamma_{n-1}"']&
L_{n-1}\arrow[lllu, "\gamma_n"']
\end{tikzcd}\\
H=h_{\mathcal V}\oplus h_{n-1}^{-1}\oplus\cdots\oplus h_1^{-1}\oplus 1\oplus h_1\oplus\cdots\oplus h_{n-1}.
\end{aligned}
\end{equation}
\begin{definition}\label{defn:gauge}
Two cyclic harmonic bundles of the form \eqref{eqn:cyclicharmonicbundle} are isomorphic if there are holomorphic isometric isomorphisms \(g_i:L_i\to L_i'\) and an orthogonal holomorphic isometric isomorphism \(g_{\mathcal V}:\mathcal V\to\mathcal V'\) intertwining the corresponding maps \(\gamma_i, \gamma_i'\).
\end{definition}

\begin{theorem}\label{thm:alternating-harmonic-correspondence}
There is a one-to-one correspondence between:
\begin{itemize}
    \item isomorphism classes, in the sense of \prettyref{defn:gauge}, of \(\SO^0_{p,q+1}\)-harmonic bundles \((\mathcal E,\theta,H)\) over \(X=(\Sigma,\mathsf j)\) of the form \eqref{eqn:cyclicharmonicbundle};
    \item \(\mathsf j\)-isomorphism classes of equivariant non-degenerate \(n\)-alternating surfaces in \(\mathbb H^{p,q}\);
    \item \(\mathsf j\)-isomorphism classes of equivariant cyclic surfaces in \(\mathbb X_\Theta\) associated with the cyclic grading described in \prettyref{example:alternating}.
\end{itemize}
Under \(V_i^{\mathbb C}=L_i\oplus L_i^{-1}\) and \(V_n^{\mathbb C}=\mathcal V\), one has
\(\gamma_i=\alpha_i^{1,0}\) for \(2\le i\le n\), while the two central arrows correspond to \(\alpha_1^{1,0}\) and \(\beta_1^{1,0}\). Moreover,
\[
q_{2n}(\Sigma)=\frac1{2n}\operatorname{tr}(\theta^{2n}). 
\]
\end{theorem}

\begin{proof}
The correspondence between cyclic harmonic bundles and cyclic surfaces follows from the general theory in \prettyref{sec:cyclicsurfaceshiggsbundles}. It remains to compare cyclic harmonic bundles with \(n\)-alternating surfaces.

Let \(f:\widetilde\Sigma\to\mathbb H^{p,q}\) be an equivariant \(n\)-alternating surface. By \prettyref{lem:symmetric-indexing},
\[
E^{\mathbb C}=V_n^{\mathbb C}\oplus L_{n-1}^{-1}\oplus\cdots\oplus L_1^{-1}\oplus\mathcal O\oplus L_1\oplus\cdots\oplus L_{n-1}.
\]
The flat connection decomposes as \(D=\nabla^H+\Psi_H\). The \((1,0)\)-part of \(\Psi_H\) consists of the tautological arrows \(\mathcal O\to L_1\otimes\KK_X\) and \(L_1^{-1}\to\mathcal O\otimes\KK_X\), the maps \(\alpha_i^{1,0}:L_{i-1}\to L_i\otimes\KK_X\) for \(2\le i\le n-1\), their \(Q\)-adjoints on the negative side, and the final pair \(\alpha_n^{1,0}:L_{n-1}\to V_n^{\mathbb C}\otimes\KK_X\) and its adjoint \(V_n^{\mathbb C}\to L_{n-1}^{-1}\otimes\KK_X\). These maps are holomorphic, and the flatness equation gives the harmonic metric equation. Thus \((E^{\mathbb C},(\nabla^H)^{0,1},\Psi_H^{1,0},H)\) is a harmonic bundle of the required form. The construction is invariant under the isomorphisms in the definition of equivariant surfaces, hence gives a well-defined map on isomorphism classes.

Conversely, start with a harmonic bundle of the form \eqref{eqn:cyclicharmonicbundle}. The Hermitian metric \(H\) and the holomorphic quadratic form \(Q\) determine a real structure \(\tau\). Its fixed locus \(E_{\mathbb R}\) carries a flat connection \(D\) and a parallel bilinear form of signature \((p,q+1)\). The prescribed signs in the normalization of \(Q\), together with the real structure determined by \(H\), give this signature. On \(\widetilde\Sigma\), choose a \(D\)-parallel trivialization \(E_{\mathbb R}\simeq\widetilde\Sigma\times\mathbb R^{p,q+1}\). The real section \(1\in\mathcal O\) defines \(f:\widetilde\Sigma\to\mathbb R^{p,q+1}\). Since the central line has quadratic form \(-1\), \(f\) takes values in \(\mathbb H^{p,q}\).

The monodromy of \(E_{\mathbb R}\) is a representation \(\rho:\pi_1(\Sigma)\to\SO^0_{p,q+1}\). With respect to the parallel trivialization, deck transformations act by \(\rho\), so \(f(\gamma x)=\rho(\gamma)f(x)\). Moreover, \(df(Y)=D_Y(1)=\Psi_H(Y)(1)\). The only adjacent components of \(\Psi_H(1)\) lie in \(L_1\oplus L_1^{-1}\), identified with \(\mathrm{T}^{1,0}\widetilde\Sigma\oplus \mathrm{T}^{0,1}\widetilde\Sigma\); hence \(df\) identifies \(\mathrm{T}\widetilde\Sigma\) with \(V_1\), and \(f\) is a spacelike immersion.

The real form of the decomposition gives \(E_{\mathbb R}=V_0\oplus\cdots\oplus V_n\), with alternating signatures and the required ranks. Since \(D\) has only adjacent blocks, the non-adjacent connection components vanish. The adjacent components are given by the two tautological central arrows and by
\(\alpha_i=\gamma_i+\tau(\gamma_i)\) for \(2\le i\le n\).

When $n\geq 3$, for \(2\le i\le n-1\), the maps \(\gamma_i:L_{i-1}\to L_i\otimes\KK_X\) are complex linear and nonzero on a dense open set, so the corresponding real maps are weakly conformal and non-degenerate. Thus \(f\) is a non-degenerate \(n\)-alternating surface.  When \(n=2\), the same construction gives
\(\alpha_2^{1,0}:L_1\to V_2^{\mathbb C}\otimes\KK_X\) and
\(\alpha_2^{0,1}:L_1^{-1}\to V_2^{\mathbb C}\otimes\overline{\KK}_X\), with no mixed component. Hence the associated second fundamental form has vanishing \((1,1)\)-part, equivalently it is trace-free. Thus \(f\) is a non-degenerate \(n\)-alternating surface. Different choices of parallel trivialization differ by an element of \(\SO^0_{p,q+1}\), so they give isomorphic equivariant surfaces. 

The identity for \(q_{2n}\) follows by evaluating \(\theta^{2n}\) along the cyclic chain; the factor \(2n\) comes from the \(2n\) possible starting vertices.
\end{proof}

\begin{remark}
\begin{enumerate}[label=(\roman*)]
    \item The higher Hopf differential \(q_{2n}(\Sigma)\) characterizes the spectrum of the Higgs field.
    \item When \(\dim\mathcal V=1\), so that the harmonic bundle has real form \(\SO^0_{n,n}\), one has \(q_{2n}(\Sigma)=q_n^2\) for a holomorphic \(n\)-differential \(q_n\). Equivalently, \(q_n\) is obtained from \(\alpha_n^{1,0}\circ\cdots\circ\alpha_1^{1,0}\).
    \item Nilpotent harmonic bundles of the form \eqref{eqn:cyclicharmonicbundle} correspond to equivariant superconformal alternating surfaces inducing \(X\).
\end{enumerate}
\end{remark}

\begin{remark}
The correspondence extends to branched immersed \(n\)-alternating surfaces. Dropping the immersion assumption replaces \(L_1=\KK_X^{-1}\) by \(L_1=\KK_X^{-1}\otimes\mathcal O(D)\), where \(D\) is the divisor of zeros of \(\partial f\).
\end{remark}

\subsection{Cyclic surfaces for \texorpdfstring{$G=\SO_{p+q+1}\mathbb C$}{G=SO(p+q+1,C)}}

Let \(f:\Sigma\to\mathbb H^{p,q}\) be an \(n\)-alternating surface. Using the ambient trivialization of \(f^*\mathrm{T}\mathbb R^{p,q+1}\otimes\mathbb C\) as $\mathbb R^{p,q+1}\otimes \mathbb C$, the splitting
\[
V_n^{\mathbb C}\oplus L_{n-1}^{-1}\oplus\cdots\oplus L_1^{-1}\oplus L_0\oplus L_1\oplus\cdots\oplus L_{n-1}
\]
defines a map \(F:\Sigma\to\mathbb X_\Theta\), where \(L_0=V_0^{\mathbb C}\), \(L_n=V_n^{\mathbb C}\) and \(L_{-i}=L_i^{-1}\). Here \(\mathbb X_\Theta=G/H^\Theta\) is the space of orthogonal splittings of the fixed vector space \(\mathbf V=\mathbb C^{p+q+1}\) with the same algebraic type, where \(\mathbf L_i\) has weight \(\bar i\in\mathbb Z/2n\mathbb Z\) of the same dimension as $L_i$. 
For an open domain $U\subset \Sigma$, a local adapted frame \(\hat F:U\to G\) associates smoothly to each point an adapted orthonormal frame which respects this splitting and the orthogonal structure. Its Maurer--Cartan form decomposes as \(\hat F^{-1}d\hat F=\hat A+\hat\omega\), where \(\hat A\in\Omega^1(U,\mathfrak g_{\bar0})\), $\hat\omega\in \Omega^1(U,\mathfrak g_{\bar 1}\oplus \mathfrak g_{-\bar 1})$. Note that $\hat\omega=\hat\alpha+\hat\beta$, where $\hat\alpha, \hat\beta$ are the lift of $\alpha, \beta$. By \prettyref{lem:symmetric-indexing}, we have
\[
\hat\omega^{1,0}=\hat\alpha^{1,0}+\hat\beta^{1,0}\in\Omega^{1,0}(U,\mathfrak g_{\bar1}),\qquad
\hat\omega^{0,1}=\hat\alpha^{0,1}+\hat\beta^{0,1}\in\Omega^{0,1}(U,\mathfrak g_{-\bar1}).
\]
This is the \(\mathcal J\)-holomorphicity condition for \(F\) and thus  \(F\) is a cyclic surface. In the equivariant case, the same construction on \(\widetilde\Sigma\) gives a \(\rho\)-equivariant cyclic surface.  By \prettyref{prop:realorbit}, the image of \(F\) lies in a
real orbit
\[
\mathbb X_\Theta^{\mathbb R}=G_{\mathbb R}\cdot o\subset\mathbb X_\Theta,
\qquad
G_{\mathbb R}=\SO^0_{p,q+1}.
\]
On this real orbit there is a natural \(G_{\mathbb R}\)-equivariant
central-line projection
\[
\ell_0:\mathbb X_\Theta^{\mathbb R}
\longrightarrow
\mathbb H^{p,q}/\{\pm1\},
\]
defined by taking the negative real line determined by the zero-weight
summand:
\[
\ell_0
\left(
L_n\oplus L_{n-1}^{-1}\oplus\cdots\oplus L_1^{-1}
\oplus L_0
\oplus L_1\oplus\cdots\oplus L_{n-1}
\right)
=
[L_0\cap\mathbb R^{p,q+1}].
\]
For the cyclic surface associated with \(f\), one has
\[
\ell_0(F)=[f],
\]
where \([f]\) denotes the image of the negative line \(\mathbb R f\) in
\(\mathbb H^{p,q}/\{\pm1\}\). 

\subsection{Jacobi fields}

\begin{definition}\label{defn:jacobi_alt}
Let \((\rho,f)\) be an equivariant \(n\)-alternating surface in \(\mathbb H^{p,q}\). A \emph{Jacobi field} on \(f\) is an equivariant vector field \(v\in\Gamma(\widetilde\Sigma,f^*\mathrm{T}\mathbb H^{p,q})\) arising as \(v=\left.\partial_t f_t\right|_{t=0}\) from a smooth family of equivariant \(n\)-alternating surfaces \((\rho_t,f_t)\), with \((\rho_0,f_0)=(\rho,f)\) and \(\dot\rho_0=0\).
\end{definition}

Let \(E=\widetilde\Sigma\times_\rho\mathbb R^{p,q+1}\). Since \(\dot\rho_0=0\), the infinitesimal variation descends to a global section of \(E\). Locally, choose adapted frames \(\hat F_t:U\rightarrow G\) for the induced cyclic splittings and set \(\hat\eta=\left.\hat F_0^{-1}\partial_t\hat F_t\right|_{t=0}\). Its block-diagonal part depends on the choice of frames, while its off-diagonal part defines a global cyclic Jacobi field \(\xi\) on $\Sigma$, independent of the choice of frames. The off-diagonal blocks give bundle morphisms \(\eta_{k,j}\in\Gamma(\Sigma,\operatorname{Hom}(V_j,V_k))\), and the variation of the position vector is \(v=\sum_{k=1}^n\eta_{k,0}(f)\).

\begin{proposition}\label{prop:Jacobi_formula_SO}
The off-diagonal blocks \(\eta_{k,j}\) of a Jacobi field are determined by \(v\). For \(k>j\), they satisfy the following relations, with the conventions \(\beta_{n+1}=0\), \(\eta_{n+1,j}=0\), and \(\eta_{k,-1}=0\):
\begin{enumerate}[label=(\roman*)]
    \item For \(k\ge2\), after identifying \(\mathrm{T}\Sigma\) with \(V_1\) by \(df\), \(\eta_{k,1}(X)=(\nabla^{\mathbb H^{p,q}}_Xv)^{V_k}\).
    \item For \(j\ge1\) and \(k\ge j+2\),
    \[
    \eta_{k,j+1}\circ\alpha_{j+1}(X)
    =
    (\nabla^\oplus_X\eta)_{k,j}
    +\alpha_k(X)\eta_{k-1,j}
    +\beta_{k+1}(X)\eta_{k+1,j}
    -\eta_{k,j-1}\beta_j(X).
    \]
\end{enumerate}
Here \((\nabla^\oplus_X\eta)_{k,j}=\nabla_X^{V_k}\eta_{k,j}-\eta_{k,j}\nabla_X^{V_j}\). For \(j<k\), 
\[
\eta_{j,k}=(-1)^{k-j+1}\eta_{k,j}^{*},
\]
where \(*\) denotes the adjoint with respect to the positive definite metrics \(h_j\) and \(h_k\).
\end{proposition}

\begin{proof}
The standard variation formula for the Maurer--Cartan form \(\omega_t=\hat F_t^{-1}d\hat F_t\) gives
\[
\partial_t\omega_t=d\hat\eta_t+[\omega_t,\hat\eta_t].
\]
At \(t=0\), write \(\omega_0=\hat A+\hat\alpha+\hat\beta\). Since each \(f_t\) is \(n\)-alternating, the non-adjacent blocks of \(\omega_t\) vanish. Hence the \((k,j)\)-block of \(\partial_t\omega_t\) is zero whenever \(|k-j|\ge2\). Projecting the variation formula to such a block gives the recurrence in \((\mathrm{ii})\).

For \(j=0\), evaluating the \((k,0)\)-block on the generator of the position line gives \(\eta_{k,1}(X)=(D_Xv)^{V_k}\). Since \(v\) is tangent to \(\mathbb H^{p,q}\) and \(k\ge2\), this equals \((\nabla^{\mathbb H^{p,q}}_Xv)^{V_k}\). The last statement follows by differentiating the orthogonality relation and using
\[
\langle\cdot,\cdot\rangle|_{V_i}=(-1)^{i+1}h_i.
\]
\end{proof}

For a Jacobi field \(v\), define \(\mathscr D^{(0)}v=\hat v\), where \(\hat v(f)=v\), and let \(\mathscr D^{(j)}v\) be the collection of blocks \(\eta_{k,j}\) with \(k>j\). Equivalently, \(\mathscr D^{(1)}v(X)=(\nabla^{\mathbb H^{p,q}}_Xv)^{\ge2}\), and the higher \(\mathscr D^{(j)}v\) are determined by the recurrence in \prettyref{prop:Jacobi_formula_SO}. On the open set where \(\alpha_{j+1}\) is nonzero, the recurrence determines \(\eta_{k,j+1}\) uniquely; since these blocks come from a smooth variation of $n$-alternating curves, they extend smoothly across the zero divisor. This extension is used in the definition of \(\mathscr D^{(j)}v\).

\subsection{Admissibility and infinitesimal rigidity}

Let \(h\) be the positive-definite metric on \(\mathrm{T}\mathbb H^{p,q}|_\Sigma\) obtained by flipping signs on the negative-definite subbundles.

\begin{definition}
The pointwise \(k\)-th order \(\alpha\)-weighted Sobolev norm of a Jacobi field \(v\) is
\[
|v|_{W^k_\alpha,h}=\sum_{j=0}^k|\mathscr D^{(j)}v|_h.
\]
When \(n=2\), this is equivalent to
\[
|v|_{W^1_\alpha,h}\asymp |v|_h+\left|(\nabla^{\mathbb H^{p,q}}v)^\perp\right|_{h,g_f}.
\]
Here \((\nabla^{\mathbb H^{p,q}}v)^\perp\) denotes the normal projection of the ambient covariant derivative; if \(v\) is normal, it is the usual normal connection.
\end{definition}

\begin{definition}\label{defn:admissiblealternating}
A Jacobi field \(v\) is said to be \emph{admissible} if either 
(i) there exists a Liouville-type pair $(g,\mathcal C)$ on $X$ such that the function
\(|v|_{W^{n-1}_\alpha,h}^2\) satisfies the condition
\(\mathcal C\) or 
(ii) there is a complete conformal metric \(g^{\mathcal T_X}\) such that
    \[
    \liminf_{r\to\infty}\frac1r\int_{B(x,r;g^{\mathcal T_X})}|v|_{W^{n-1}_\alpha,h}^2|\alpha|_{h,g^{\mathcal T_X}}\,\mathrm{vol}_{g^{\mathcal T_X}}=0,
    \]
    and     \[
\liminf_{r\to\infty}\frac1r\int_{B(x,r;g^{\mathcal T_X})}|v|_{W^{n-1}_\alpha,h}^p\,\mathrm{vol}_{g^{\mathcal T_X}}=0,\quad\text{some \(p>0\).}
    \]
Here \(\alpha=\sum_{i=1}^n\alpha_i\). A variation is admissible if its Jacobi field is admissible.
\end{definition}

\begin{proposition}\label{prop:admissiblealternating}
A Jacobi field \(v\) is admissible if one of the following holds: \begin{enumerate}
        \item when $X$ is a potential-theoretically parabolic Riemann surface, $|v|_{W^{n-1}_\alpha,h}$ is bounded;
         \item there exists a complete conformal metric $g^{\mathcal{T}_X}$ such that 
        \[\liminf_{r\to+\infty}\frac{1}{r}\int_{B(x,r;g^{\mathcal{T}_X})}|v|_{W^{n-1}_\alpha,h}^p\vol_{g^{\mathcal{T}_X}}=0,\qquad\text{for some } p>2;\]
         In particular, this holds if:
     \begin{itemize}
       \item $v$ is compactly supported;
         \item $|v|_{W^{n-1}_\alpha,h} \in L^p(g^{\mathcal{T}_X})$ for some $p>2$ and some complete conformal metric $g^{\mathcal{T}_X}$;
        \item $X=\mathbb D$ and $|v|_{W^{n-1}_\alpha,h} = O\big((1-|z|)^{1/2}\big)$; 
     \end{itemize}
    \item \(|v|_{W^{n-1}_\alpha,h}\in L^2(\Phi)\), where \(|\Phi|=|\alpha|_h^2\) defines a complete metric.
       \end{enumerate}
\end{proposition}

\begin{proof}
The proof is identical to the one of \prettyref{prop:admissible_cyclic_sufficient}.
\end{proof}

\begin{corollary}
Let \(v\) be a Jacobi field of an equivariant complete maximal spacelike surface in \(\mathbb H^{2,q}\). Set
\[
\|v\|_{W^1}=|v|_h+\left|(\nabla^{\mathbb H^{2,q}}v)^\perp\right|_{h,g_f}.
\]
Then \(v\) is admissible if one of the following holds:
\begin{itemize}
    \item \(v\) is compactly supported;
    \item when $X$ is a potential-theoretically parabolic Riemann surface, \(\|v\|_{W^1}\) is bounded;
    \item \(X=\mathbb D\) and \(\|v\|_{W^1}=O((1-|z|)^{1/2})\);
    \item \(\displaystyle \liminf_{r\to\infty}\frac1r\int_{B(x,r;g_f)}\|v\|_{W^1}^2\,\dd\mathrm{vol}_{g_f}=0.\)
\end{itemize}
\end{corollary}
\begin{proof}
 For a complete maximal spacelike surface in \(\mathbb H^{2,q}\), we have
\[
|\alpha|_{g_f}\le C
\]
by \cite[Theorem 1]{cheng1993space}. Hence the \(L^2\)-growth condition with respect to \(g_f\) in the last item implies both growth conditions in Definition \ref{defn:admissiblealternating}, with \(g^{\mathcal T_X}=g_f\) and \(p=2\). The other cases follow from \prettyref{prop:admissiblealternating}.
\end{proof}

\begin{proposition}\label{prop:correspondence_admissible}
Let \(v\) be a Jacobi field of an \(n\)-alternating surface \(f\), and let \(\xi\) be the corresponding Jacobi field of the associated cyclic surface \(F\). Then \(v\) is admissible if and only if \(\xi\) is admissible.
\end{proposition}

\begin{proof}
The cyclic Jacobi field decomposes into its off-diagonal blocks \(\eta_{k,j}\), with \(\eta_{j,k}=(-1)^{k-j+1}\eta_{k,j}^{*}\), where $*$ means the adjoint with respect to $h_i$'s. By \prettyref{prop:Jacobi_formula_SO}, the upper triangular blocks are precisely the components collected by \(\mathscr D^{(j)}v\). Hence the norm of $\xi$ and the weighted Sobolev norm of $v$ are uniformly equivalent:
\[
|\xi|_h\asymp |v|_{W^{n-1}_\alpha,h}.
\]
The admissibility conditions are obtained from one another by this substitution.
\end{proof}

\begin{definition}
An \(n\)-alternating surface in \(\mathbb H^{p,q}\) is \emph{irreducible} if its image is not contained in \(L\cdot x\) for any proper Levi subgroup \(L\subsetneq\SO^0_{p,q+1}\) and any \(x\in\mathbb H^{p,q}\).
\end{definition}

\begin{theorem}\label{thm:alternationtheorem}
Let \((\rho_t,f_t)\) be an admissible smooth variation of an equivariant irreducible \(n\)-alternating surface \((\rho,f)\), with \(\dot\rho_0=0\). Then there exists a smooth path \(\psi_t\in\operatorname{Diff}^0(\Sigma)\), with \(\psi_0=\operatorname{id}_\Sigma\), such that \(f'_t=f_t\circ\psi_t\) satisfies \(\dot f'_0=0\).
\end{theorem}
\begin{proof}
By \prettyref{lem:unique-alternating-splitting}, the family \(f_t\) determines
a unique family of associated cyclic surfaces
$F_t:\widetilde\Sigma\to\mathbb X_\Theta^{\mathbb R}.$ Moreover, $F_t$ is a family of $\alpha$-cyclic surfaces for the $\Theta$-cyclic root $\alpha=(-\varepsilon_1^{\overline 1},\overline 1),$ see Example \ref{example:sonC} and \ref{example:alternating}.
By \prettyref{prop:correspondence_admissible}, the Jacobi field
\(v=\dot f_0\) corresponds to an admissible cyclic Jacobi field \(\xi\) of
\(F=F_0\).

We first check that \(F\) is irreducible from the irreducibility of $f$. Suppose, for contradiction, that
\(F\) is reducible. Then  
there exists a proper Levi subgroup
$
L\subsetneq G_{\mathbb R}=\SO^0_{p,q+1}$
and a point \(y\in\mathbb X_\Theta^{\mathbb R}\) such that
$
F(\widetilde\Sigma)\subset L\cdot y.$
Applying the \(G_{\mathbb R}\)-equivariant central-line projection
\[
\ell_0:\mathbb X_\Theta^{\mathbb R}\to\mathbb H^{p,q}/\{\pm1\},
\]
we get
\[
[f](\widetilde\Sigma)=\ell_0(F(\widetilde\Sigma))
\subset L\cdot \ell_0(y).
\]
Choose a unit vector \(x\in\mathbb H^{p,q}\) spanning the negative line
\(\ell_0(y)\). Then
\[
f(\widetilde\Sigma)\subset L\cdot x\cup L\cdot(-x).
\]
Since \(\widetilde\Sigma\) is connected and \(f\) is a smooth unit lift of
\([f]\), the sign is constant. Hence
\[
f(\widetilde\Sigma)\subset L\cdot x
\quad\text{or}\quad
f(\widetilde\Sigma)\subset L\cdot(-x).
\]
This contradicts the irreducibility of \(f\). Therefore \(F\) is irreducible.

By \prettyref{thm:main}, there exists a smooth path $\psi_t\in\operatorname{Diff}^0(\Sigma)$ with $\psi_0=\operatorname{id}_\Sigma,$
such that
$\dot F'_0=0,$ for $F'_t=F_t\circ\psi_t.$
By uniqueness of the alternating splitting, \(F'_t\) is the cyclic surface
associated with
$f'_t=f_t\circ\psi_t.$ The condition \(\dot F'_0=0\) implies, in particular, that the central line
$L_0(t)=\mathbb C f'_t$
has zero first variation. Hence $
\dot f'_0=a f'$
for some real function \(a\). Since
$\langle f'_t,f'_t\rangle=-1$,
differentiating at \(t=0\) gives
\[
0=2\langle \dot f'_0,f'\rangle
=2a\langle f',f'\rangle
=-2a.
\]
Therefore \(a=0\), and hence $\dot f'_0=0.$
\end{proof}

\appendix
\section{Examples}
\begin{example}\label{example:slnC}
    Let $G=\mathrm{SL}_n\mathbb{C}$. We fix an integer $m$ and a direct sum decomposition $\mathbb{C}^n=\bigoplus_{\overline{j}\in\mathbb{Z}/m\mathbb{Z}}V_{\overline{j}}$, where $\dim_{\mathbb{C}}V_{\overline{j}}=d_{\overline{j}}$ (so that $\sum_{\overline{j}\in\mathbb{Z}/m\mathbb{Z}}d_{\overline{j}}=n$). Then, the Lie algebra $\mathfrak{g}=\mathfrak{sl}_n\mathbb{C}$ consists of traceless endomorphisms of $\mathbb{C}^n$, which we grade by
    \[\mathfrak{g}_{\overline{k}}:=\left(\bigoplus_{\overline{j}\in\mathbb{Z}/m\mathbb{Z}}\operatorname{Hom}_{\mathbb{C}}\left(V_{\overline{j+k}},V_{\overline{j}}\right)\right)_0,\]
    where the subscript zero means taking the subspace of traceless endomorphisms, and is only meaningful for $\mathfrak{g}_{\overline{0}}$. We call such a grading is of type $(d_{\overline{j}})_{\overline{j}\in\mathbb{Z}/m\mathbb{Z}}$. Then \[G_{\overline{0}}=S\left(\prod_{\overline{j}\in\mathbb{Z}/m\mathbb{Z}}\mathrm{GL}(V_{{\overline{j}}})\right)\cong S\left(\prod_{\overline{j}\in\mathbb{Z}/m\mathbb{Z}}\mathrm{GL}_{d_{\overline{j}}}\mathbb{C}\right).\]

    In such case, the grading automorphism $\theta$ is an inner automorphism induced as follows. Let $I_{(d_{\overline{j}})}\in \prod_{\overline{j}\in\mathbb{Z}/m\mathbb{Z}}\mathrm{GL}(V_{{\overline{j}}})$ defined by $I_{(d_{\overline{j}})}|_{V_{\overline{j}}}=\zeta_m^{-j}\operatorname{id}_{V_{\overline{j}}}$. Then we uniformize it as \[I_{(d_{\overline{j}})}':=\dfrac{I_{(d_{\overline{j}})}}{\det(I_{(d_{\overline{j}})})^{1/n}}\in G_{\overline{0}}.\]
    Thus $\theta$ is induced by $\Theta:=\operatorname{Ad}\left(I_{(d_{\overline{j}})}'\right)\in\operatorname{Aut}_m(G)$. We obtain that $G^\Theta=G_{\overline{0}}$.
    
    Using the matrix, $\tau\colon\mathfrak{g}\to\mathfrak{g}$ can be chosen as the conjugate transpose and $\mathfrak{t}_{\overline{0}}$ can be chosen as 
    \[\left\{\bigoplus_{\overline{j}\in\mathbb{Z}/m\mathbb{Z}}\operatorname{diag}\left(a_1^{(\overline{j})},\dots,a_{d_{\overline{j}}}^{(\overline{j})}\right)\bigg|a_p^{(\overline{j})}\in\mathbb{C},\sum_{\overline{j}\in\mathbb{Z}/m\mathbb{Z}}\sum_{p=1}^{d_{\overline{j}}}a_p^{(\overline{j})}=0\right\}.\]
    Denote by $\varepsilon_{p}^{(\overline{j})}\in\mathfrak{t}_{\overline{0}}^\vee$ the linear function 
    \[
        \bigoplus_{\overline{j}\in\mathbb{Z}/m\mathbb{Z}}\operatorname{diag}\left(a_1^{(\overline{j})},\dots,a_{d_{\overline{j}}}^{(\overline{j})}\right)\mapsto a_p^{(\overline{j})}\]
    Hence $\widetilde{\Delta}=\left\{\left(\varepsilon_{p}^{(\overline{j})}-\varepsilon_{q}^{(\overline{j+k})},\overline{k}\right)\bigg|\ p\neq q\mbox{ when }\overline{k}=\overline{0}\right\}$.

    Now let $(\mathbb{E},\varphi,h)$ be a $\Theta$-cyclic harmonic $\mathrm{SL}_n\mathbb{C}$-Higgs bundle. Using the standard representation of $\mathrm{SL}_n\mathbb{C}$, we obtain that $(\mathbb{E},\varphi,h)$ is equivalent to the following data:
    \begin{itemize}
        \item the direct sum the holomorphic vector bundles $\mathcal{E}_{\overline{j}}:=\mathbb{E}[V_{\overline{j}}]$;
        
        \item a global holomorphic section of
\[\bigoplus_{\overline{j}\in\mathbb{Z}/m\mathbb{Z}}\operatorname{Hom}\left(\mathcal{E}_{\overline{j+1}},\mathcal{E}_{\overline{j}}\right)\otimes\KK_X,\]
which is $\varphi$. We denote by $\varphi_{\overline{j}}$ its $\operatorname{Hom}\left(\mathcal{E}_{\overline{j+1}},\mathcal{E}_{\overline{j}}\right)\otimes\KK_X$ component;

        \item an Hermitian metric $h_{\EE}$ on $\EE$ which is the direct sum of Hermitian metrics $h_{\overline{j}}$ on every $\mathbb{E}_{\overline{j}}$ such that \prettyref{eq:harmonic} holds.
\end{itemize}
These give us a harmonic $m$-cyclic $\mathrm{SL}_n\mathbb{C}$-Higgs bundle in the usual sense. We can express such Higgs bundles as the following diagram:
\[\begin{tikzcd}
	{\mathcal{E}_{\overline{0}}} & {\mathcal{E}_{\overline{1}}} & \cdots & {\mathcal{E}_{\overline{m-1}}}
	\arrow["{\varphi_{\overline{m-1}}}", curve={height=-18pt}, from=1-1, to=1-4]
	\arrow["{\varphi_{\overline{0}}}", curve={height=-6pt}, from=1-2, to=1-1]
	\arrow["{\varphi_{\overline{1}}}", curve={height=-6pt}, from=1-3, to=1-2]
	\arrow["{\varphi_{\overline{m-2}}}", curve={height=-6pt}, from=1-4, to=1-3]
\end{tikzcd}\]

    When $d_{\overline{0}}=d_{\overline{1}}=1$, the root space corresponding to $\left(\varepsilon_{1}^{(\overline{0})}-\varepsilon_{1}^{(\overline{1})},\overline{1}\right)$ is $\operatorname{Hom}_\mathbb{C}(V_{\overline{1}},V_{\overline{0}})$ and it is $\operatorname{Ad}(H^\Theta)$-invariant. Hence $\left(\varepsilon_{1}^{(\overline{0})}-\varepsilon_{1}^{(\overline{1})},\overline{1}\right)$ is a $\Theta$-cyclic root.
\end{example}

\begin{example}\label{example:ccyclic}
    We explicitly illustrate a general \textbf{$\mathfrak{c}$-cyclic surface} by restricting the set-up of \prettyref{example:slnC} to $d_{\overline{0}}=d_{\overline{1}}$.
    
    Fix a linear isomorphism $J \in \operatorname{Hom}(V_{\overline{1}}, V_{\overline{0}})$. The subspace $\mathfrak{c} := \mathbb{C}J \subset \mathfrak{g}_{\overline{1}}$ is one-dimensional. For $d_{\overline{0}} \geqslant 2$, $\mathfrak{c}$ is not natively invariant under the full group $H^\Theta$. To make it invariant, we restrict to the reductive subgroup $U < H^\Theta$ defined by elements $(g_{\overline{j}})_{\overline{j}\in\mathbb{Z}/m\mathbb{Z}}$ satisfying that $g_{\overline{0}} J g_{\overline{1}}^{-1} = \lambda J$ for some scalar $\lambda \in \mathbb{C}^*$. Under the adjoint action of $U$, the subspace $\mathfrak{c}$ is cleanly scaled by $\lambda$. Thus, $\mathfrak{c}$ is a \textbf{$U$-cyclic subspace}.

    Now the $\Theta$-cyclic harmonic $G$-Higgs bundle reducing to $U$ associated to a $\mathfrak{c}$-cyclic surface is equivalent to a harmonic $m$-cyclic $\mathrm{SL}_{n}\mathbb{C}$-Higgs bundle $(\EE,\varphi,h_\EE)$
    satisfying that:
    \begin{itemize}
        \item $\varphi_{\overline{0}}\colon\mathcal{E}_{\overline{1}}\to\mathcal{E}_{\overline{0}}\otimes \mathcal{K}_X$ is an isomorphism;

        \item $h_{\overline{1}}=h_{\overline{0}}\otimes h_{\KK_X}$ for an Hermitian metric $h_{\KK_X}$ on $\KK_X$.
    
    \end{itemize}    
        Note that the metric requirements are not necessarily automatically satisfied even when the isomorphicity of $\varphi_{\overline{0}}$ is true.
\end{example}

\begin{example}\label{example:sonC}
    Let $\mathbf{q}$ be a symmetric bilinear form on $\mathbb{C}^n$, $G=\mathrm{SO}(\mathbf{q})\cong\mathrm{SO}_{n}\mathbb{C}$. We fix an integer $m$ and a direct sum decomposition $\mathbb{C}^n=\bigoplus_{\overline{j}\in\mathbb{Z}/m\mathbb{Z}}V_{\overline{j}}$, such that 
    \[\mathbf{q}|_{V_{\overline{j}}\times V_{\overline{k}}}\begin{cases}
        \mbox{is a perfect pairing}&\mbox{when }\overline{j}+\overline{k}= \overline{0},\\
        \equiv0&\mbox{when }\overline{j}+\overline{k}\neq \overline{0}.
    \end{cases}\] 
    Let $d_{\overline{j}}:=\dim_{\mathbb{C}}V_{\overline{j}}$. Hence  $\sum_{\overline{j}\in\mathbb{Z}/m\mathbb{Z}}d_{\overline{j}}=n$ and $d_{\overline{j}}=d_{-\overline{j}}$. Then, the Lie algebra $\mathfrak{g}=\mathfrak{so}(\mathbf{q})\cong\mathfrak{so}_n\mathbb{C}$ consists of $\mathbf{q}$-skew-symmetric endomorphisms of $\mathbb{C}^n$, which we grade by
    \[\mathfrak{g}_{\overline{k}}:=\mathfrak{so}(\mathbf{q})\cap\bigoplus_{\overline{j}\in\mathbb{Z}/m\mathbb{Z}}\operatorname{Hom}_{\mathbb{C}}\left(V_{\overline{j+k}},V_{\overline{j}}\right).\] 
    We call such a grading is of type $(d_{\overline{j}})_{\overline{j}\in\mathbb{Z}/m\mathbb{Z}}$. 
    %Then \[G_{\overline{0}}=S\left(\prod_{\overline{j}\in\mathbb{Z}/m\mathbb{Z}}\mathrm{GL}(V_{{\overline{j}}})\right)\cong S\left(\prod_{\overline{j}\in\mathbb{Z}/m\mathbb{Z}}\mathrm{GL}_{d_{\overline{j}}}\mathbb{C}\right).\]

    In such case, we still consider the endomorphism $I_{(d_{\overline{j}})}\in \prod_{\overline{j}\in\mathbb{Z}/m\mathbb{Z}}\mathrm{GL}(V_{{\overline{j}}})$ defined by $I_{(d_{\overline{j}})}|_{V_{\overline{j}}}=\zeta_m^{-j}\operatorname{id}_{V_{\overline{j}}}$. Actually, $I_{(d_{\overline{j}})}\in\mathrm{O}(\mathbf{q})\cong\mathrm{O}_n\mathbb{C}$ and we have \begin{itemize}
        \item $I_{(d_{\overline{j}})}\in\mathrm{O}(\mathbf{q})\setminus\mathrm{SO}(\mathbf{q})$ if and only if $m=2\ell$ for some $\ell\in\mathbb{Z}$ and $d_{\overline{\ell}}$ is odd,

        \item $-I_{(d_{\overline{j}})}\in\mathrm{O}(\mathbf{q})\setminus\mathrm{SO}(\mathbf{q})$ if and only if $d_{\overline{0}}$ is odd.
    \end{itemize}
    Let $\Theta:=\operatorname{Ad}\left(I_{(d_{\overline{j}})}\right)\in\operatorname{Aut}_m(G)$. The outer class of $\Theta$ in $\operatorname{Out}(\mathfrak{so}_n\mathbb{C})$ is non-trivial if and only if $\pm I_{(d_{\overline{j}})}\in\mathrm{O}(\mathbf{q})\setminus\mathrm{SO}(\mathbf{q})$, if and only if $m=2\ell$ for some $\ell\in\mathbb{Z}$ and both $d_{\overline{0}}$ and $d_{\overline{\ell}}$ are odd.
    
    We have that 
    \[G^\Theta= S\left(\prod_{\substack{\overline{j}\in\mathbb{Z}/m\mathbb{Z}\\2\overline{j}=\overline{0}}}\mathrm{O}(\mathbf{q}|_{V_{\overline{j}}})\times\prod_{\substack{\{\overline{j},-\overline{j}\}\subseteq\mathbb{Z}/m\mathbb{Z}\\2\overline{j}\neq\overline{0}}}\left(\mathrm{GL}(V_{\overline{j}})\times\mathrm{GL}(V_{-\overline{j}})\right)\cap\mathrm{O}(\mathbf{q}|_{V_{\overline{j}}\oplus V_{-\overline{j}}})\right).\]
    Note that since $\mathbf{q}|_{V_{\overline{j}}\times V_{-\overline{j}}}$ is a perfect pairing, we obtain an isomorphism $\mathbf{Q}_{\overline{j}}\colon V_{\overline{j}}\cong V_{-\overline{j}}^{\vee}$. When $2\overline{j}\neq\overline{0}$, by fixing an arbitrary basis $\left(e_p^{(\overline{j})}\right)_{p=1}^{d_{\overline{j}}}$ of $V_{\overline{j}}$ and its dual basis $\left(e_p^{(-\overline{j})}\right)_{p=1}^{d_{\overline{j}}}$ of $V_{-\overline{j}}$, we can regard $\mathbf{q}|_{V_{\overline{j}}\oplus V_{-\overline{j}}}$ as the matrix
    \[\begin{pmatrix}
        0&I_{d_{\overline{j}}}\\I_{d_{\overline{j}}}&0
    \end{pmatrix}.\]
    Therefore, the element in \[\left(\mathrm{GL}(V_{\overline{j}})\times\mathrm{GL}(V_{-\overline{j}})\right)\cap\mathrm{O}(\mathbf{q}|_{V_{\overline{j}}\oplus V_{-\overline{j}}})\]
    must be $(g,(g^{\mathrm{t}})^{-1})$ under the dual basis. Hence the intersection group actually lie in $\mathrm{SO}(\mathbf{q}|_{V_{\overline{j}}\oplus V_{-\overline{j}}})$. And we have 
    \[\begin{aligned}
        G_{\overline{0}}=& \prod_{\substack{\overline{j}\in\mathbb{Z}/m\mathbb{Z}\\2\overline{j}=\overline{0}}}\mathrm{SO}(\mathbf{q}|_{V_{\overline{j}}})\times\prod_{\substack{\{\overline{j},-\overline{j}\}\subseteq\mathbb{Z}/m\mathbb{Z}\\2\overline{j}\neq\overline{0}}}\left(\mathrm{GL}(V_{\overline{j}})\times\mathrm{GL}(V_{-\overline{j}})\right)\cap\mathrm{SO}(\mathbf{q}|_{V_{\overline{j}}\oplus V_{-\overline{j}}})\\
        \cong&\prod_{\substack{\overline{j}\in\mathbb{Z}/m\mathbb{Z}\\2\overline{j}=\overline{0}}}\mathrm{SO}_{d_{\overline{j}}}\mathbb{C}\times\prod_{\substack{\{\overline{j},-\overline{j}\}\subseteq\mathbb{Z}/m\mathbb{Z}\\2\overline{j}\neq\overline{0}}}\mathrm{GL}_{d_{\overline{j}}}\mathbb{C}.
    \end{aligned}\]
    
     When $2\overline{j}=\overline{0}$, we choose a basis $\left(e_p^{(\overline{j})}\right)_{p=1}^{d_{\overline{j}}}$ of $V_{\overline{j}}$ such that \[\mathbf{q}(e_p^{(\overline{j})},e_q^{(\overline{j})})=\begin{cases}
         1&\mbox{when }p+q=d_{\overline{j}}+1,\\
         0&\mbox{otherwise.}
     \end{cases}\]
     With the chosen basis $\left(e_p^{(\overline{j})}\right)$, $\tau\colon\mathfrak{g}\to\mathfrak{g}$ can be chosen as the conjugate transpose and the elements in $\mathfrak{t}_{\overline{0}}$ have the following form:
    \[\bigoplus_{\substack{\overline{j}\in\mathbb{Z}/m\mathbb{Z}\\ 2\overline{j}=0}}\operatorname{diag}\left(a_1^{(\overline{j})},\dots,a_{d_{\overline{j}}}^{(\overline{j})}\right)\bigoplus_{\substack{\{\overline{j},-\overline{j}\}\subseteq\mathbb{Z}/m\mathbb{Z}\\ 2\overline{j}\neq0}}\operatorname{diag}\left(a_1^{(\overline{j})},\dots,a_{d_{\overline{j}}}^{(\overline{j})},-a_1^{(\overline{j})},\dots,-a_{d_{\overline{j}}}^{(\overline{j})}\right),\]
    where 
    \[a_p^{(\overline{j})}\in\mathbb{C},\qquad a_p^{(\overline{j})}+a_{d_{\overline{j}}+1-p}^{(\overline{j})}=0\mbox{ when }2\overline{j}=\overline{0}.\]
    Denote by $\varepsilon_{p}^{(\overline{j})}\in\mathfrak{t}_{\overline{0}}^\vee$ the linear function sending the element above to $a_p^{(\overline{j})}$.
    Hence $\widetilde{\Delta}=\left\{\left(\pm\varepsilon_{p}^{(\overline{j})}\pm\varepsilon_{q}^{(\overline{j+k})},\overline{k}\right)\bigg|\ p\neq q\mbox{ when }\overline{k}=\overline{0}\right\}$.
    
    Now let $(\mathbb{E},\varphi,h)$ be a $\Theta$-cyclic harmonic $\mathrm{SO}_n\mathbb{C}$-Higgs bundle. Using the standard representation of $\mathrm{SO}_n\mathbb{C}$, we obtain that $(\mathbb{E},\varphi,h)$ is equivalent to the following data:
    \begin{itemize}
        \item a holomorphic orthogonal vector bundle $(\mathcal{E},\mathbf{q}_{\EE})$ of rank $n$, where $\EE$ is the direct sum the holomorphic vector bundles $\mathcal{E}_{\overline{j}}:=\mathbb{E}[V_{\overline{j}}]$ and $\mathbf{q}_\EE$ satisfies that:\[\mathbf{q}_{\EE}|_{\EE_{\overline{j}}\times \EE_{\overline{k}}}\begin{cases}
        \mbox{is a perfect pairing}&\mbox{when }\overline{j}+\overline{k}= \overline{0},\\
        \equiv0&\mbox{when }\overline{j}+\overline{k}\neq \overline{0},
    \end{cases}\]
    which implies that $\EE_{\overline{j}}\cong\EE_{-\overline{j}}^\vee$;
        \item a global holomorphic section $\varphi$ of
\[\bigoplus_{\overline{j}\in\mathbb{Z}/m\mathbb{Z}}\operatorname{Hom}\left(\mathcal{E}_{\overline{j+1}},\mathcal{E}_{\overline{j}}\right)\otimes\KK_X\]
 that is skew-symmetric with respect to $\mathbf{q}_\EE$. We denote by $\varphi_{\overline{j}}$ its $\operatorname{Hom}\left(\mathcal{E}_{\overline{j+1}},\mathcal{E}_{\overline{j}}\right)\otimes\KK_X$ component. Then the skew-symmetricity means that $\varphi_{\overline{j}}$ and $\varphi_{-(\overline{j+1})}$ are skew-symmetric;

        \item an Hermitian metric $h_{\EE}$ on $\EE$ which is the direct sum of Hermitian metrics $h_{\overline{j}}$ on every $\mathbb{E}_{\overline{j}}$ compatible with $\mathbf{q}_\EE$ such that \prettyref{eq:harmonic} holds. Here ``compatible'' means $h_{\EE}$ is the same as its dual metric through the isomorphism between $\EE$ and $\EE^\vee$ induced by the orthogonal structure $\mathbf{q}_\EE$.
\end{itemize}     
    
    When $d_{\overline{0}}=d_{\overline{1}}=1$, the space generated by $\operatorname{Hom}_\mathbb{C}(V_{\overline{1}},V_{\overline{0}}) \oplus \operatorname{Hom}_\mathbb{C}(V_{\overline{0}},V_{-\overline{1}})$ under skew-symmetry is exactly one-dimensional. Since $V_{\overline{0}}$ is self-dual, its torus action is trivial ($\varepsilon_{1}^{(\overline{0})} \equiv 0$), making $-\varepsilon_{1}^{(\overline{1})}$ a $\Theta$-cyclic root of degree $\overline{1}$.
\end{example}

\begin{example}\label{example:alternating}
    Specializing \prettyref{example:sonC} to an even cyclic grading $m=2n$, taking $d_{\overline{0}}=1$, $d_{\overline{j}}=1$ for $1 \leqslant j \leqslant n-1$, then the real form defined by $\Lambda=\Theta^n\circ\tau$ is isomorphic to $\SO^0_{2n-2,2n-3+d_{\overline{n}}}$ (when $n$ is even) or $\SO^0_{2n-2+d_{\overline{n}},2n-3}$ (when $n$ is odd). Then the corresponding cyclic Higgs bundle is given by
    \begin{equation}
    \begin{tikzcd}[column sep=small]
    &&&\mathcal V\arrow[dlll, "-\gamma_n^{\mathrm{t}}"']&&&\\
    L_{n-1}^{-1}\arrow[r, "\gamma_{n-1}"']&
    \cdots\arrow[r, "\gamma_2"']&
    L_1^{-1}\arrow[r, "1"']&
    \mathcal O\arrow[r, "-1"']&
    L_1\arrow[r, "-\gamma_2^{\mathrm{t}}"']&
    \cdots\arrow[r, "-\gamma_{n-1}^{\mathrm{t}}"']&
    L_{n-1}\arrow[lllu, "\gamma_n"']&
    \end{tikzcd}
    \end{equation}
    where  $\mathcal{E}_{\overline{0}} = \mathcal{O}$ and $\mathcal{E}_{\overline{j}} = L_j^{-1}$ (with duals correctly tracking as $\mathcal{E}_{-\overline{j}} = L_j$), and $\VV=\EE_{\overline{n}}$.
\end{example}

\begin{example}\label{exa:alpha_cyclic}
    To contrast an example of the $U$-reduction required in \prettyref{example:ccyclic}, we restrict the orthogonal setup of \prettyref{example:sonC} to $m=3$, with $d_{\overline{0}}=n$ and $d_{\overline{1}}=d_{-\overline{1}}=2$. Now $G = \mathrm{SO}_{n+4}\mathbb{C}$. Denoting $W=V_{\overline{1}}$, we have $V_{-\overline{1}}=W^\vee$ and $G^\Theta \cong \mathrm{SO}(V_{\overline{0}}) \times \mathrm{GL}(W) \cong \mathrm{SO}_n\mathbb{C} \times \mathrm{GL}_2\mathbb{C}$.

    The degree $\overline{1}$ block $\mathfrak{g}_{\overline{1}}$ contains the subspace $\operatorname{Hom}(V_{-\overline{1}}, V_{\overline{1}}) \cap \mathfrak{so}(\mathbf{q}) \cong \operatorname{Hom}(W^\vee, W) \cap \mathfrak{so}(\mathbf{q})$. Such maps naturally identify with skew-symmetric tensors in $W \otimes W \cong \Lambda^2 W$. Since $\dim W = 2$, this subspace $\mathfrak{c} := \Lambda^2 W \subset \mathfrak{g}_{\overline{1}}$ is one-dimensional.
    
    Let $\mathfrak{t}_{\overline{0}}$ be the maximal abelian subalgebra of $\mathfrak{g}_{\overline{0}}$ as in \prettyref{example:sonC}. The action of $\mathsf{t} \in \mathfrak{t}_{\overline{0}}$ on $W = V_{\overline{1}}$ is represented by $\operatorname{diag}(a_1^{(\overline{1})}, a_2^{(\overline{1})})$. Let $\varepsilon_1^{(\overline{1})}, \varepsilon_2^{(\overline{1})} \in \mathfrak{t}_{\overline{0}}^\vee$ be the corresponding coordinate functionals. The adjoint action of $\mathsf{t}$ on $A \in \mathfrak{c} \cong \Lambda^2 W$ naturally has weight $a_1^{(\overline{1})} + a_2^{(\overline{1})}$. Therefore, $\mathfrak{c}$ is precisely the one-dimensional root space $\mathfrak{g}_{\boldsymbol{\alpha}}$ for the root $\boldsymbol{\alpha} = (\alpha, \overline{1})$, where $\alpha = \varepsilon_1^{(\overline{1})} + \varepsilon_2^{(\overline{1})}$.

    We describe the corresponding $\Theta$-cyclic harmonic $\mathrm{SO}_{n+4}\mathbb{C}$-Higgs bundle $(\mathbb{E},\varphi,h)$ via vector bundles. Firstly, we have the orthogonal vector bundle $\EE_{\overline{0}}=\mathbb{E}[V_{\overline{0}}]$ of rank $n$. We also have the vector bundle $\EE_{\overline{1}}=\mathbb{E}[W]$ of rank $2$ and its dual $\EE_{-\overline{1}}$. Now the $\boldsymbol{\alpha}$-cyclic implies that the skew-symmetric morphism $\varphi_{\overline{1}}\colon\EE_{-\overline{1}}\to\EE_{\overline{1}}\otimes\KK_X$ is an isomorphism. Therefore, there exists a vector bundle $\VV$ of rank $2$ with trivial determinant, i.e. with a symplectic form, such that $\EE_{\overline{1}}\cong\VV\otimes\KK_{X}^{-1/2}$, $\EE_{-\overline{1}}\cong\VV\otimes\KK_{X}^{1/2}$ and $\varphi_{\overline{1}}=\operatorname{id}_{\VV}$.
    Therefore, the $\Theta$-cyclic $\SO_{n+4}\mathbb{C}$-Higgs bundle is of the following form:
    \begin{equation}
    % https://q.uiver.app/#q=WzAsMyxbMSwwLCJcXG1hdGhjYWx7RX1fe1xcb3ZlcmxpbmV7MH19Il0sWzAsMCwiXFxtYXRoY2Fse1Z9XFxvdGltZXNcXG1hdGhjYWx7S31fWF57LTEvMn0iXSxbMiwwLCJcXG1hdGhjYWx7Vn1cXG90aW1lc1xcbWF0aGNhbHtLfV9YXnsxLzJ9Il0sWzEsMCwiXFxiZXRhIiwwLHsiY3VydmUiOi0yfV0sWzAsMiwiLVxcYmV0YV57XFxtYXRocm17dH19IiwwLHsiY3VydmUiOi0yfV0sWzIsMSwiXFxvcGVyYXRvcm5hbWV7aWR9X3tcXG1hdGhjYWx7Vn19IiwwLHsiY3VydmUiOi0zfV1d
\begin{tikzcd}
	{\mathcal{V}\otimes\mathcal{K}_X^{-1/2}} & {\mathcal{E}_{\overline{0}}} & {\mathcal{V}\otimes\mathcal{K}_X^{1/2}}
	\arrow["\beta", curve={height=-18pt}, from=1-1, to=1-2]
	\arrow["{-\beta^{\mathrm{t}}}", curve={height=-18pt}, from=1-2, to=1-3]
	\arrow["{\operatorname{id}_{\mathcal{V}}}", curve={height=-24pt}, from=1-3, to=1-1]
\end{tikzcd}
    \end{equation}
    Here $\varphi_{\overline{0}} = \beta$ and $\varphi_{-\overline{1}} = -\beta^{\mathrm{t}}$. A $\Theta$-cyclic harmonic metric takes the form $(h_{\mathcal V}, h_{\mathcal U}, h_{\mathcal V}g^{-1})$ where $g$ is an Hermitian metric on $\mathcal K_X^{-1}.$
\end{example}

\bibliographystyle{alpha}
\bibliography{bib}

\end{document}